\documentclass[12pt]{amsart}
\usepackage{amssymb,amsmath, amsthm,latexsym,verbatim}
\usepackage{a4wide}
\usepackage{mathrsfs}
\usepackage[backref=page]{hyperref}
\usepackage{enumitem}
\usepackage{todonotes}
\theoremstyle{plain}

\newtheorem{lem}{Lemma}[section]
\newtheorem{theo}[lem]{Theorem}
\newtheorem{prop}[lem]{Proposition}
\newtheorem{cor}[lem]{Corollary}
\newtheorem{remark}[lem]{Remark}

\newtheorem{conjecture}[lem]{Conjecture}

\parskip=\medskipamount
\font\k=cmr7

  \newcommand {\cu}{\mbox{\k cus}}
  \newcommand {\di}{\mbox{\k dis}}
  \newcommand {\red}{\mbox{\k red}}

  \newcommand{\unip}{\operatorname{unip}}
  
  \newcommand {\sico}{\mbox{\k sc}}
  
  \newcommand {\reg}{\text{reg}}
  \newcommand {\res}{\mbox{\k res}}
  \newcommand {\spec}{\text{spec}}
  \newcommand {\geo}{\mbox{\k geom}}

  \newcommand {\temp}{\mbox{\k temp}}

  \newcommand {\C}{{\mathbb C}}
  \newcommand {\bH}{{\mathbb H}}
  \newcommand {\N}{{\mathbb N}}
  \newcommand {\R}{{\mathbb R}}
  \newcommand {\Z}{{\mathbb Z}}
  \newcommand {\Q}{{\mathbb Q}}
  \newcommand {\A}{{\mathbb A}}
  
  \newcommand{\bG}{G}
  \newcommand {\af}{{\mathfrak a}}
  \newcommand {\gf}{{\mathfrak g}}
  \newcommand {\mf}{{\mathfrak m}}
  \newcommand {\kf}{{\mathfrak k}}

  \newcommand {\of}{{\mathfrak o}}
  \newcommand {\nf}{{\mathfrak n}}

  \newcommand {\cf}{{\mathfrak c}}

 \newcommand{\ho}{{\mathfrak o}}

\renewcommand {\H}{{\mathcal H}}
  \newcommand {\M}{{\mathcal M}}
  
  \newcommand {\Co}{{\mathcal C}}
 \newcommand {\cO}{{\mathcal O}}
 
\newcommand {\cF}{{\mathcal F}}
 \newcommand {\ccC}{{\mathscr C}}
\newcommand {\bbG}{\bf G}
\newcommand {\G}{G}

\newcommand {\CmL}{{\mathcal L}}
  \newcommand {\cE}{{\mathcal E}}
  
  \newcommand {\cH}{{\mathcal H}}
  
 \newcommand {\cP}{{\mathcal P}}
 \newcommand {\cL}{{\mathcal L}}
 \newcommand {\cA}{{\mathcal A}}

\newcommand  {\cZ}{{\mathcal Z}}
\newcommand {\bs}{\backslash}

 \newcommand {\ov}{\overline}

\newcommand{\cU}{{\mathcal U}}

\newcommand{\levis}{{\mathcal L}}
\newcommand{\Ai}{A_M}

\newcommand{\Ag}{A_G}
\newcommand{\AG}{A_G}

\renewcommand{\Im}{\operatorname{Im}}
\renewcommand{\Re}{\operatorname{Re}}
\newcommand{\Tr}{\operatorname{Tr}}

\newcommand{\End}{\operatorname{End}}

\newcommand{\tr}{\operatorname{tr}}
\newcommand{\Id}{\operatorname{Id}}
\newcommand{\Hom}{\operatorname{Hom}}

\newcommand{\Ind}{\operatorname{Ind}}

\newcommand{\rk}{\operatorname{rank}}
\newcommand{\vol}{\operatorname{vol}}

\newcommand{\Res}{\operatorname{Res}}
\newcommand{\SL}{\operatorname{SL}}
\newcommand{\GL}{\operatorname{GL}}
\newcommand{\SO}{\operatorname{SO}}

\newcommand{\PSL}{\operatorname{PSL}}

\renewcommand{\det}{\operatorname{det}}

\newcommand{\norm}[1]{\lVert#1\rVert}
\newcommand{\abs}[1]{\lvert#1\rvert}

\newcommand{\one}{\mathbf 1}
\newcommand{\oP}{\overline P}
\newcommand{\Lie}{\operatorname{Lie}}

\newcommand{\aaa}{\mathfrak{a}}
  \newcommand {\K}{{\bf K}}

  \newcommand{\Ht}{H}
\newcommand{\sprod}[2]{\left\langle#1,#2\right\rangle}
\newcommand{\PPP}{\mathcal{P}}
\newcommand{\FFF}{{\mathcal F}}
\newcommand{\rts}{\Sigma}

\newcommand{\disc}{\operatorname{disc}}
\newcommand{\srts}{\Delta}
\newcommand{\modulus}{\delta}
\newcommand{\AF}{{\mathcal A}}
\newcommand{\zzz}{\mathfrak{Z}}
\newcommand{\iii}{{\mathrm i}}
\newcommand{\LieG}{\mathfrak{g}}
\newcommand{\bases}{\mathfrak{B}}
\newcommand{\bss}{\underline{\beta}}
\newcommand{\dtup}{\mathcal{X}}
\newcommand{\card}[1]{\lvert#1\rvert}

\newcommand{\CmP}{\mathcal{P}}
\newcommand{\CmU}{\mathcal{U}}

\newcommand{\CmO}{\mathcal{O}}

\newcommand{\cpt}{\mathbf{K}}
\newcommand{\One}{\mathbf{1}}
\newcommand{\level}{\operatorname{level}}

\begin{document}

\title[]
{The Weyl law for congruence subgroups  and arbitrary
$K_\infty$-types}
\date{\today}

\author{Werner M\"uller}
\address{Universit\"at Bonn\\
Mathematisches Institut\\
Endenicher Allee 60\\
D -- 53115 Bonn, Germany}
\email{mueller@math.uni-bonn.de}



\begin{abstract}
Let $G$ be a reductive algebraic group over $\Q$  and 
$\Gamma\subset G(\Q)$ an
arithmetic subgroup. Let $K_\infty\subset G(\R)$ be a maximal compact subgroup.
We study the asymptotic behavior of the counting functions  of the cuspidal
and residual spectrum, respectively, of the regular representation of $G(\R)$ in
$L^2(\Gamma\bs G(\R))$ of a fixed $K_\infty$-type $\sigma$.
A conjecture, which is due to Sarnak, states that the counting function of
the cuspidal spectrum of type $\sigma$  satisfies Weyl's law  and
the residual spectrum is of lower order growth. Using the Arthur trace formula
we reduce the conjecture to a problem about $L$-functions occurring in the
constant terms of Eisenstein series. If $G$ satisfies property (L),
introduced by Finis and Lapid, we establish the conjecture.
This includes classical groups over a number field.
\end{abstract}

\maketitle
\setcounter{tocdepth}{1}
\tableofcontents

\section{Introduction}
Let $G$ be a connected semisimple algebraic group over $\Q$ and
$\Gamma\subset G(\Q)$ an arithmetic subgroup, which we assume to be torsion
free. A basic problem in the theory of
automorphic forms is the study of the spectral resolution of the regular
representation $R_\Gamma$ of $G(\R)$ in $L^2(\Gamma\bs G(\R))$. Of particular
importance is the determination of the structure of the discrete spectrum. 
Let $L^2_{\di}(\Gamma\bs G(\R))$
be the discrete part of $L^2(\Gamma\bs G(\R))$, i.e., the closure of the span
of all irreducible subrepresentations of $R_\Gamma$. Denote by $R_{\Gamma,\di}$
the corresponding restriction of $R_\Gamma$. Denote by $\Pi_{\di}(G(\R))$ the
set of isomorphism classes of irreducible unitary representations of $G(\R)$,
which occur in $R_\Gamma$. By definition we have
\begin{equation}\label{reg-repr}
R_{\Gamma,\di}=\widehat\bigoplus_{\pi\in\Pi_{\di}(G(\R))}m_\Gamma(\pi)\pi,
\end{equation}
where
\[
m_\Gamma(\pi)=\dim\Hom_{G(\R)}(\pi,R_\Gamma)=\dim\Hom_{G(\R)}(\pi,R_{\Gamma,\di})
\]
is the multiplicity with which $\pi$ occurs in $R_\Gamma$.
Apart from special cases, as for example discrete series representations, one
cannot hope to describe the multiplicity function $m_\Gamma$ on $\Pi(G(\R))$
explicitly. Therefor it is feasible to study asymptotic questions such as
the limit multiplicity problem \cite{FLM2} and the Weyl law, which is the
subject of this article.

To begin with we recall that the discrete spectrum decomposes into the cuspidal
and the residual spectrum. Let $\K_\infty$ be a maximal compact subgroup of
$G(\R)$. Let $\cZ(\gf_\C)$ be the center of the
universal enveloping algebra of the complexification of the Lie algebra $\gf$
of $G(\R)$. Recall that a cusp form for $\Gamma$ is a smooth and right
$\K_\infty$-finite function $\phi\colon\Gamma\bs G(\R)\to \C$ which is a
simultaneous eigenfunction of $\cZ(\gf_\C)$ and which satisfies 
\begin{equation}
\int_{\Gamma\cap N_P(\R)\bs N_P(\R)}\phi(nx) dn=0
\end{equation}
for all unipotent radicals $N_P$ of proper rational parabolic subgroups $P$ of
$G$. By Langlands' theory of Eisenstein series \cite{La1},  cusp forms are the
building blocks of the spectral resolution. 
We note that each cusp form $\phi\in C^\infty(\Gamma\bs G(\R))$ is
rapidly decreasing on $\Gamma\bs G(\R)$ and hence square integrable. Let
$L^2_{\cu}(\Gamma\bs G(\R))$ be the closure of the linear span of all cusp
forms. The restriction of the regular representation $R_\Gamma$ to
$L^2_{\cu}(\Gamma\bs G(\R))$ decomposes discretely and
$L^2_{\cu}(\Gamma\bs G(\R))$ is a subspace of $L^2_{\di}(\Gamma\bs G(\R))$. Denote
by $L^2_{\res}(\Gamma\bs G(\R))$ the orthogonal complement of
$L^2_{\cu}(\Gamma\bs G(\R))$ in $L^2_{\di}(\Gamma\bs G(\R))$. This is the
residual subspace. Let $(\sigma,V_\sigma)$ be an irreducible unitary
representation of $\K_\infty$. Set
\begin{equation}\label{sigma-iso}
L^2(\Gamma\bs G(\R),\sigma):=(L^2(\Gamma\bs G(\R))\otimes V_\sigma)^{\K_\infty}.
\end{equation}
Define the subspaces $L^2_{\di}(\Gamma\bs G(\R),\sigma)$,
$L^2_{\cu}(\Gamma\bs G(\R),\sigma)$ and $L^2_{\res}(\Gamma\bs G(\R),\sigma)$ in a
similar way. Then $L^2_{\cu}(\Gamma\bs G(\R),\sigma)$ is the space of cusp
forms with fixed $\K_\infty$-type $\sigma$. 
Let $\widetilde X:=G(\R)/\K_\infty$ be the Riemannian symmetric space associated
to $G(\R)$ and $X=\Gamma\bs \widetilde X$ the corresponding locally symmetric
space. Since we assume that $\Gamma$ is torsion free, $X$ is a manifold. Let
$E_\sigma\to \Gamma\bs X$ be the locally homogeneous vector bundle
associated to $\sigma$ and let $L^2(X,E_\sigma)$ be the space
of square integrable sections of $E_\sigma$. There is a canonical isomorphism
\begin{equation}\label{homog-vect-bdl}
L^2(\Gamma\bs G(\R),\sigma)\cong L^2(X,E_\sigma).
\end{equation}
Let $\Omega_{G(\R)}\in\cZ(\gf_\C)$ be the Casimir element of $G(\R)$. Then
$-\Omega_{G(\R)}\otimes\Id$ induces a self-adjoint operator $\Delta_\sigma$ in the
Hilbert space $L^2(\Gamma\bs G(\R),\sigma)$ which is bounded from below. With
respect to the isomorphism \eqref{homog-vect-bdl} we have
\begin{equation}
\Delta_\sigma=(\nabla^\sigma)^\ast\nabla^\sigma-\lambda_\sigma\Id,
\end{equation}
where $\nabla^\sigma$ is the canonical invariant connection in $E_\sigma$ and
$\lambda_\sigma$ denotes the Casimir eigenvalue of $\sigma$. In particular,
if $\sigma_0$ is the trivial representation, then $L^2(\Gamma\bs G(\R),\sigma_0)
\cong L^2(X)$ and $\Delta_{\sigma_0}$ equals the Laplacian $\Delta$ on $X$.

The restriction of $\Delta_\sigma$ to the subspace
$L^2_{\di}(\Gamma\bs G(\R),\sigma)$ has pure point spectrum consisting of
eigenvalues $\lambda_0(\sigma)<\lambda_1(\sigma)<\cdots$ of finite
multiplicities. Let $\cE(\lambda_i(\sigma))$ be the eigenspace corresponding to
$\lambda_i(\sigma)$. Then we define the eigenvalue counting function
$N_{\Gamma,\di}(\lambda,\sigma)$, $\lambda\ge 0$, by
\begin{equation}
N_{\Gamma,\di}(\lambda,\sigma)=\sum_{\lambda_i(\sigma)\le \lambda}
\dim\cE(\lambda_i(\sigma)).
\end{equation}
The counting functions $N_{\Gamma,\cu}(\lambda,\sigma)$ and
$N_{\Gamma,\res}(\lambda,\sigma)$ of
the cuspidal and residual spectrum are defined  by considering
the restriction of $\Delta_\sigma$ to the cuspidal and residual subspace,
respectively. 
The main goal is to determine the asymptotic behavior of the counting functions
as $\lambda\to \infty$. If $X$ is compact, the Weyl law holds. 
Recall that for a compact Riemannian manifold $X$ of dimension $n$, the Weyl law
states that the number $N_X(\lambda)$ of eigenvalues $\lambda_i\le\lambda$,
counted with multiplicity,  of the Laplace operator $\Delta$ of $X$  satisfies
\begin{equation}\label{weyl-law1}
N_X(\lambda)=\frac{\vol(X)}{(4\pi)^n\Gamma(\frac{n}{2}+1)}\lambda^{n/2}
+o(\lambda^{(n/2})
\end{equation}
as $\lambda\to\infty$. A standard method to prove \eqref{weyl-law1} is the heat
equation method. Using the wave equation, one gets a more precise version with
an estimation of the remainder term:
\begin{equation}
N_X(\lambda)=\frac{\vol(X)}{(4\pi)^n\Gamma(\frac{n}{2}+1)}\lambda^{n/2}
+O(\lambda^{(n-1)/2})
\end{equation}
as $\lambda\to\infty$. This is due to Avakumovic \cite{Av} and H\"ormander.
Without
further assumptions on the Riemannian manifold, the remainder term is optimal
\cite{Av}. More generally, one can consider the Bochner-Laplace operator
$\Delta_E$ for a Hermitian vector bundle $E\to X$ with Hermitian connection.
There is a similar formula \eqref{weyl-law1}
for the eigenvalue counting function $N_X(\lambda,E)$ of $\Delta_E$. The only
difference is the rank of $E$ which appears on the right hand side in the
leading coefficient.

For non-uniform lattices $\Gamma$ the self-adjoint
operator $\Delta_\sigma$ has a large continuous spectrum so that alomost all
eigenvalues of $\Delta_\sigma$ will be embedded in the continuous spectrum which
makes it very difficult to study them. A number of
results are known for the spherical cuspidal spectrum. 
The first results concerning the growth of the cuspidal spectrum are due to
Selberg \cite{Se1}. He proved that for every congruence subgroup
$\Gamma\subset\SL(2,\Z)$,
the counting function of the cuspidal spectrum satisfies Weyl's law, i.e.,
one has
\begin{equation}\label{weyl-law4}
N_\Gamma(\lambda)=\frac{\vol(\Gamma\bs\bH^2)}{4\pi}\lambda^2-\frac{2m}{\pi}
\lambda\log\lambda+O(\lambda),
\end{equation}
as $\lambda\to\infty$.
This shows that for congruence subgroups eigenvalues
exist in abundance. On the other hand, based on their work on the dissolution
of cusp forms under deformation of lattices, Phillips and Sarnak \cite{Sa2}
conjectured that except for the Teichm\"uller space of the once punctured torus,
the point spectrum of the Laplacian on $\Gamma\bs\bH^2$ for a generic
non-uniform lattice $\Gamma$ in $\SL(2,\R)$ is finite and is contained in
$[0,1/4)$. In the more general context of manifolds with cusps
Colin de Verdère \cite{CV} has shown that under a generic compactly supported
conformal deformation of the metric of a non-compact hyperbolic surface of
finite area all eigenvalues $\lambda\ge 1/4$ are dissolved. 

If $\rk(G)>1$, the situation is very different.
By the results of Margulis, we have rigidity of irreducible lattices and
irreducible lattices are arithmetic. One expects that arithmetic groups
have a large discrete spectrum. The following conjecture is due to Sarnak
\cite{Sa1}.
\begin{conjecture}\label{conj1}
Let $\Gamma\subset G(\Q)$ be a congruence subgroup. Then for every $\nu\in
\Pi(\K_\infty)$, $N_{\Gamma,\cu}(\lambda;\nu)$ satisfies Weyl's law and
$N_{\Gamma,\res}(\lambda;\nu)$ is of lower order growth.
\end{conjecture}

There are some general results concerning the conjecture. Let $G$ be a
connected real semisimple Lie group, $K$ a maximal compact subgroup of $G$,
and $\Gamma\subset G$ a torsion free lattice. Let $n=\dim G/K$. Donnelly
\cite[Theorem 9.1]{Do} has established the following upper bound for
the cuspidal spectrum
\begin{equation}\label{upper-bd}
\limsup_{\lambda\to\infty}\frac{N_{\Gamma,\cu}(\lambda;\nu)}{\lambda^{n/2}}\le
\frac{\dim(\nu)\vol(\Gamma\bs G/K)}{(4\pi)^{n/2}\Gamma(\frac{n}{2}+1)},
\end{equation}
which holds for every $\nu\in\Pi(K)$. Concerning the residual spectrum,
 it was proved in \cite[Theorem 0.1]{Mu1} that for a general lattice one has
\begin{equation}\label{bd-res-spec}
N_{\Gamma,\res}(\lambda;\nu)\ll 1+\lambda^{2n}.
\end{equation}
However, this is not the optimal bound that one expects. In general,
one would expect
that the residual spectrum is of order $O(\lambda^{n/2})$ and for arithmetic
groups of order $O(\lambda^{(n-1)/2})$ as $\lambda\to\infty$. 

Conjecture \eqref{conj1} has been verified in a number of cases. Most
of the results are obtained for the spherical spectrum. The first result in
higher rank  is due to S.D. Miller \cite{Mil}
who established the Weyl law for spherical cusp forms for 
$\Gamma=\SL(3,\Z)$. The author \cite{Mu3} proved it for a principal congruence
subgroup $\Gamma\subset \SL(n,\Z)$. The method of proof
follows Selberg's approach and uses the trace formula. Then
Lindenstrauss and Venkatesh  \cite{LV} proved the Weyl law for
spherical cusps forms in great generality, namely for congruence subgroups
$\Gamma\subset \bbG(\R)$, where $\bbG$ is a split adjoint semisimple group
over $\Q$.
The method is different. It uses Hecke operators to eliminate the
contribution of Eisenstein series. For congruence subgroups of $\SL(n,\Z)$,
E. Lapid and W. M\"uller \cite{LM} established the Weyl law for the cuspidal
spectrum with an estimation of the remainder term. The order of the remainder
term is $O(\lambda^{(d-1)/2}(\log\lambda)^{\max(n,3)})$ where
$d=\dim\SL(n,\R)/\SO(n)$. The method is also based on the Arthur trace formula
as in \cite{Mu3}. However, the argument is simplified and strengthened, which
corresponds to the use of the wave equation in the derivation of the Weyl law
for a compact Riemannian manifold.
Recently T. Finis and E. Lapid \cite{FL2} estimated
the remainder term for the cuspidal spectrum of a locally symmetric space
$X=\Gamma\bs {\bbG}(\R)/\K_\infty$, where $\bbG$ is a simply connected,
simple Chevalley group and $\Gamma$ a congruence subgroup of $\bbG(\Z)$. The
method
also uses Hecke operators as in \cite{LV}, but in a slightly different way. The
estimation they obtain is $O(\lambda^{d-\delta})$, where $d=\dim X$ and
$\delta>0$ some constant which is not further specified. In \cite{FM}, T. Finis
and J. Matz included Hecke operators. They studied the asymptotic behavior of
the traces of Hecke operators for the spherical discrete spectrum. 

For the non-spherical case, 
the Weyl law was proved in \cite{Mu3} for a principal congruence subgroup of
$\SL(n,\Z)$. Recently, A. Maiti \cite{Ma} has generalized the approach of
Lindenstrauss and Venkatesh \cite{LV} to establish the Weyl law for cusp forms
and arbitrary $K_\infty$-types. As in \cite{LV}, the method works
for a semi-simple, split, adjoint linear algebraic group over $\Q$. It provides
no results for the residual spectrum. 

Concerning the residual spectrum, there is the general upper bound
\eqref{bd-res-spec}, which, however, is not the expected optimal one.
For $\rk(G)=1$,
the residual spectrum is known to be finite. For $\GL(n)$ the residual
spectrum has been determined by M{\oe}glin and Waldspurger \cite{MW}. This
has been used in \cite[Proposition 3.6]{Mu3} to prove that in this case the
residual spectrum is of lower order growth.,

The main goal of the present paper is to prove Conjecture \eqref{conj1} for a
certain class of reductive groups including classical groups over a number
field. We use the Arthur trace formula to reduce the proof of the conjecture to
a problem about automorphic $L$-functions occurring in the constant terms of
Eisenstein series. This problem can be dealt with if the reductive group $G$ 
satisfies property (L), which was introduced by Finis and Lapid in  
\cite[Definition 3.4]{FL1}. 
Let $G$ be a reductive group over $\Q$. 
As usual, let $G(\R)^1$ denote the intersection of the kernels of the
homorphisms $|\chi|\colon G(\R)\to \R^{>0}$, where $\chi$ ranges over the
$\Q$-rational characters of $G$. Then our main result is the following theorem. 
 vvv
\begin{theo}\label{main-thm}
Let $\G_0$ be a connected reductive algebraic group over a number field $F$
which satisfies property (L). Let $G=\Res_{F/\Q}(G_0/F)$. Let $\K_\infty\subset
\G(\R)^1$ be a maximal compact subgroup and let $n=\dim \G(\R)^1/\K_\infty$. 
Let $\Gamma \subset \G(\Q)$ be a torsion free congruence subgroup. Then for
every $\nu\in\Pi(\K_\infty)$ we have
\begin{equation}\label{weyl-law3}
N_{\Gamma,\cu}(\lambda;\nu)\sim\frac{\dim(\nu)vol(\Gamma\bs \G(\R)^1/\K_\infty)}
{(4\pi)^{n/2}\Gamma(\frac{n}{2}+1)}\lambda^{n/2},\quad \lambda\to\infty.
\end{equation}
and
\begin{equation}\label{res-spec-bound}
N_{\Gamma,\res}(\lambda;\nu)\ll 1+\lambda^{(n-1)/2},\quad \lambda>0.
\end{equation}
\end{theo}
Thus in order to establish the Weyl law and the estimation of the residual
spectrum for every $\K_\infty$-type, we are reduced to the verification that
$G_0$
satisfies property (L). For $\GL(n)$ the relevant $L$-functions are the
Rankin-Selberg $L$-functions, which are known to satisfy the pertinent
properties. Using Arthur's work on functoriality from classical groups to
$\GL(n)$,  T. Finis and E. Lapid \cite[Theorem 3.11]{FL1} proved that
quasi-split
classical groups over a number field $F$ satisfy property (L). Moreover, they
also proved that inner forms of $\GL(n)$ and the exceptional group $G_2$ over a
number field $F$ satisfy property (L). In fact, one expects that property (L)
holds for all reductive groups. Currently, we only know
\cite[Theorem 3.11]{FL1}.
Together with Theorem \ref{main-thm} this leads to the following corollary.

\begin{cor}\label{cor-class-groups}
Let $F$ be a number field and let $\G_0$ be one of the following groups over
$F$:
\begin{enumerate}
\item $\GL(n)$ and its inner forms.
\item Quasi-split classical groups.
\item The exceptional group $G_2$.
\end{enumerate}
Let ${\G}=\Res_{F/\Q}({\G_0}/F)$. Let $\Gamma\subset \G(\Q)$ be a congruence
subgroup
and $\nu\in\Pi(\K_\infty)$. Then \eqref{weyl-law3} and \eqref{res-spec-bound}
hold.
\end{cor}

Our approach to prove Theorem \ref{main-thm} is a generalization of the heat
equation method to the non-compact setting. The basic tool is the Arthur
trace formula. This requires to pass to the adelic setting. We will work with
reductive groups over a number field $F$. However, for the rest of the
introduction we will assume that $F=\Q$. So let $G$ be a connected reductive
group defined over $\Q$. Let $\A$ be the ring of adeles of $\Q$. Let $G(\A)^1
:=\cap_\chi \ker|\chi|$, where $\chi$ runs over the rational characters of $G$.
Denote by $T_G$ the split component of the center of $G$ and
let $\AG$ be the component of the identity of $T_G(\R)$. Then
\[
G(\A)=A_G\times G(\A)^1.
\]
We replace $\Gamma\bs G(\R)^1$ by the adelic quotient
$\AG G(\Q)\bs G(\A)/K_f=G(\Q)\bs G(\A)^1/K_f$, where
$K_f\subset G(\A_f)$ is an open compact subgroup.
Let $\Pi(G(\A))$ (resp. $\Pi(G(\A)^1)$) be the set of equivalence classes of
irreducible unitary representations of $G(\A)$ (resp. $G(\A)^1$). We identify
a representation of $G(\A)^1$ with a representation of $G(\A)$, which is trivial
on $A_G$. Let $L^2_{\di}(\AG G(\Q)\bs G(\A))$ be the
closure of the span of all irreducible subrepresentations of
the regular representation $R$ of $G(\A)$ in $L^2(\AG G(\Q)\bs G(\A))$. 
Denote by $\Pi_{\di}(G(\A))$ the subspace of all $\pi\in\Pi(G(\A))$
which are equivalent to a subrepresentation of the regular representation of
$G(\A)$ in $L^2(\AG G(\Q)\bs G(\A))$. Note that this is a countable set.
Denote by $R_{\di}$ the restriction
of $R$ to $L^2_{\di}(\AG G(\Q)\bs G(\A))$. Then
\begin{equation}\label{spec-decomp}
R_{\di}\cong \widehat\bigoplus_{\pi\in\Pi_{\di}(G(\A))}m(\pi)\pi,
\end{equation}
where
\begin{equation}\label{multipl-adelic}
m(\pi)=\dim\Hom(\pi,L^2(\AG G(\Q)\bs G(\A))
\end{equation}
is the multiplicity with which $\pi$ occurs in $L^2(\AG G(\Q)\bs G(\A))$.
Any $\pi\in\Pi(G(\A))$ can be written as  $\pi=\pi_\infty\otimes\pi_f$, where
$\pi_\infty$ and $\pi_f$ are irreducible unitary representations of $G(\R)$ and
$G(\A_f)$, respectively. Let $\H_{\pi_\infty}$ and $\H_{\pi_f}$ denote the Hilbert
space of the representation $\pi_\infty$ and $\pi_f$, respectively. 
Let $K_f\subset G(\A_f)$ be an open compact subgroup. Denote by $\H_{\pi_f}^{K_f}$
the subspace of $K_f$-invariant vectors in $\H_{\pi_f}$. Let 
$G(\R)^1=G(\A)^1\cap G(\R)$. Given $\pi\in\Pi(G(\A))$, denote by
$\lambda_{\pi_\infty}$ the Casimir eigenvalue of the restriction of $\pi_\infty$ to
$G(\R)^1$. Let $\nu\in\Pi(\K_\infty)$. Then we define the adelic counting
function of the discrete spectrum by
\begin{equation}\label{adelic-count-fct1}
N_{\di}^{K_f,\nu}(\lambda):=\sum_{\substack{\pi\in\Pi_{\di}(G(\A))\\
-\lambda_{\pi_\infty}\le\lambda}}m(\pi)\dim(\cH_{\pi_f}^{K_f})
\dim(\cH_{\pi_\infty}\otimes V_\nu)^{K_\infty}.
\end{equation}
In the same way we define the counting functions $N_{\cu}^{K_f,\nu}(\lambda)$ and
$N_{\res}^{K_f,\nu}(\lambda)$ of the cuspidal and residual spectrum, respectively.
The adelic version of Theorem \ref{main-thm} is then
\begin{theo}\label{main-thm-adelic}
Let $G_0$ be a connected reductive algebraic group over a number field $F$.
Assume that $G_0$ satisfies property (L). Let $G=\Res_{F/\Q}(G_0)$ be the
group that is obtained from $G_0$ by restriction of scalars.
Let $\K_\infty$ be a maximal compact subgroup of $G(\R)^1$. Let
$d:=\dim(G(\R)^1/\K_\infty)$.
Let $K_f\subset G(\A_f)$ be an open compact subgroup and let
$\nu\in \Pi(\K_\infty)$. Then we have
\begin{equation}\label{weyl-law-adelic}
N_{\cu}^{K_f,\nu}(\lambda)=\frac{\dim(\nu)\vol(\AG G(\Q)\bs G(\A)/K_f)}
{(4\pi)^{d/2}\Gamma(\frac{d}{2}+1)}\lambda^{d/2}+o(\lambda^{d/2}),\quad
\lambda\to\infty,
\end{equation}
and
\begin{equation}\label{res-spec-adelic}
N_{\res}^{K_f,\nu}(\lambda)\ll (1+\lambda^{(n-1)/2}), \quad \lambda>0.
\end{equation}
\end{theo}
To deduce Theorem \ref{main-thm} from the adelic version, we recall that there
exist finitely many congruence subgroups $\Gamma_i\subset G(\Q)$, $i=1,...,m$,
such that
\begin{equation}
\AG G(\Q)\bs G(\A)/K_f=\bigsqcup_{i=1}^m \Gamma_i\bs G(\R)^1.
\end{equation}
(see sect. \ref{sect-arithm-mfds}). Denote by $N_{\Gamma_i,\cu}(\lambda,\nu)$ the
counting function for the cuspidal spectrum
$L^2(\Gamma_i\bs G(\R)^1)\otimes V_\nu)^{K_\infty}$. Then it follows that
\begin{equation}\label{rel-count-fct}
N_{\cu}^{K_f,\nu}(\lambda)=\sum_{i=1}^m N_{\Gamma_i,\cu}(\lambda,\nu).
\end{equation}
This is used to derive Theorem \ref{main-thm} from Theorem
\ref{main-thm-adelic}.

To prove Theorem \ref{main-thm-adelic} we start with the estimation of the
residual counting function, which is needed to establish the Weyl law.
For this purpose we use Langlands' description of the residual spectrum in
terms of iterated residues of Eisenstein series \cite[Ch. 7]{La1},
\cite[V.3.13]{MW}. Using the Maass-Selberg
relations, the problem is finally reduced to the estimation of the number
of real poles of the normalizing factors of intertwining operators, which
appear in the constant terms of Eisenstein series. To obtain the appropriate
bounds, we need that $G$ satisfies property (L) which was introduced by Finis
and Lapid \cite[Definition 3.4]{FL1}.
In this way we get \eqref{res-spec-adelic}.

To prove the Weyl law, we use the Arthur trace formula. We will work with
groups over a number field $F$. However, in order to explain the method
we will simply assume that $F=\Q$. We proceed as in \cite{Mu3}. We choose test functions
$\phi^\nu_t\in C_c^\infty(G(\A)^1)$, $t>0$, which at the infinite place are
obtained from the heat kernel $H_t^\nu$ of the Bochner-Laplace operator
$\widetilde \Delta_\nu$ on the symmetric space $\widetilde X=G(\R)^1/\K_\infty$
and which at the finite places is given by the normalized characteristic
function of $K_f$ (see \eqref{adelic-heat-ker} for the precise definition).
Then we insert $\phi_t^\nu$ into the spectral side $J_{\spec}$ of the trace
formula and study the asymptotic behavior of $J_{\spec}(\phi_t^\nu)$ as
$t\to 0$. The spectral side is a sum of distributions $J_{\spec,M}$ associated
to conjugacy classes of Levi subgroups $M$ of $G$. For $M=G$ we have
\begin{equation}\label{discr-heat-tr}
J_{\spec,G}(\phi^\nu_t)=\sum_{\pi\in\Pi_{\di}(G(\A))} m(\pi) e^{t\lambda_{\pi_\infty}}
\dim(\cH_{\pi_f}^{K_f})\dim(\cH_{\pi_\infty}\otimes V_\nu)^{K_\infty},
\end{equation}
which is the contribution of the discrete spectrum to the spectral side. 
For $M\neq G$, the main ingredients of the distributions
$J_{\spec,M}$  are logarithmic derivatives of intertwining operators. The
intertwining operators can be normalized by certain meromorphic functions.
Then  the logarithmic derivatives of the intertwining operators are expressed
in terms of logarithmic derivatives of the normalizing
factors and logarithmic derivatives of the local normalized intertwining
operators. In fact, we only need to control integrals of logarithmic
derivatives which simplifies the problem. To deal with the integrals of 
logarithmic derivatives of the normalizing factors, we  use
property (TWN+) \cite[Definition 3.3]{FL1}.  By \cite[Proposition 3.8]{FL1},
property (TWN+) is a consequence of property (L) \cite[Definition 3.4]{FL1},
which we assume to be satisfied by $G$. To deal with the local intertwining
operators, we follow essentially the approach used in \cite{FLM2}. The final
result is Theorem \ref{thm-spec-side}, which states that if $G$ satisfies
property (L), then
\begin{equation}\label{spec-side-3}
J_{\spec}(\phi_t^\nu)=J_{\spec,G}(\phi_t^\nu)+O(t^{-(d-1)/2})
\end{equation}
as $t\to 0$.

Next we come to the geometric side $J_{\geo}(\phi_t^\nu)$. Its asymptotic behavior
as $t\to 0$ has been determined in \cite[Theorem 1.1]{MM2}. We will briefly
recall the main steps of the proof and determine the leading coefficient.
By the trace formula we have $J_{\spec}(\phi_t^\nu)=J_{\geo}(\phi_t^\nu)$, which
together with \eqref{discr-heat-tr} and \eqref{spec-side-3} leads to
\begin{equation}\label{adelic-heat-trace}
\begin{split}
\sum_{\pi\in\Pi_{\di}(G(\A))}m(\pi)\dim(\cH_{\pi_f}^{K_f})
\dim&\left(\cH_{\pi_\infty}\otimes V_\nu\right)^{K_\infty}\,e^{t\lambda_{\pi_\infty}}\\
&=\frac{\dim(\nu)\,\vol(X(K_f))}{(4\pi)^{d/2}}t^{-d/2}+O(t^{-(d-1)/2})
\end{split}
\end{equation}
as $t\to 0$.
Applying Karamata's theorem, we obtain the adelic Weyl law
\eqref{weyl-law-adelic}.

\section{Preliminaries}\label{sec-prelim}
\setcounter{equation}{0}

We will mostly use the notation of \cite{FLM1}. Let $G$ be a reductive algebraic
group defined over a number field $F$. We fix a minimal parabolic subgroup
$P_0$ of $\bG$ defined over $F$ and a Levi decomposition $P_0=M_0 U_0$, both
defined over $F$. Let $T_0$ be the $F$-split component of the center of $M_0$.
Let $\cF$ be the set of parabolic subgroups of $G$ which
contain $M_0$ and are defined over $F$. Let $\cL$ be the set of subgroups of
$\bG$ which contain $M_0$ and are Levi components of groups in $\cF$. 
For any $P\in\cF$ we write
\[
P=M_PN_P,
\]
where $N_P$ is the unipotent radical of $P$ and $M_P$ belongs to $\cL$. 

Let $M\in\cL$. Denote by $T_M$ the $F$-split component of the
center of $M$. 
Put $T_P=T_{M_P}$. With our previous notation, we have $T_0=T_{M_0}$. Let
$L\in\cL$ and assume that $L$ contains
$M$. Then $L$ is
a reductive group defined over $F$ and $M$ is a Levi subgroup of $L$. We 
shall denote the set of Levi subgroups of $L$ which contain $M$ by $\cL^L(M)$.
We also write $\cF^L(M)$ for the set of parabolic subgroups of $L$, defined 
over $F$, which contain $M$, and $\cP^L(M)$ for the set of groups in $\cF^L(M)$
for which $M$ is a Levi component. Each of these three sets is finite. If 
$L=\bG$, we shall usually denote these sets by $\cL(M)$, $\cF(M)$ and $\cP(M)$.

Let $W_0=N_{\bG(F)}(T_0)/M_0$ be the Weyl group of $(\bG,T_0)$,
where $N_{\bG(F)}(H)$ denotes the normalizer of $H$ in $\bG(F)$.
For any $s\in W_0$ we choose a representative $w_s\in \bG(F)$.
Note that $W_0$ acts on $\levis$ by $sM=w_s M w_s^{-1}$. For $M\in\cL$ let
$W(M)=N_{\bG(F)}(M)/M$, which can be identified with a subgroup of $W_0$.

Let $X(M)_F$ be the group of characters of $M$ which are defined over $F$. 
Put
\begin{equation}\label{liealg}
\af_{M}:=\Hom(X(M)_F,\R).
\end{equation}
This is a real vector space whose dimension equals that of $T_M$. Its dual 
space is
\[
\af_{M}^\ast=X(M)_F\otimes \R.
\]
 We shall write, 
\begin{equation}\label{liealg1}
\af_P=\af_{M_P} \quad\text{and}\quad \af_0=\af_{M_0}.
\end{equation}
 
For any $L\in\cL(M)$ we identify $\af_L^\ast$ with a subspace of $\af_M^\ast$.
We denote by $\af_M^L$ the annihilator of $\af_L^\ast$ in $\af_M$. Then $r=\dim \af_0^{\bG}$ is the semisimple rank of $\bG$.
We set
\begin{equation}\label{l1}
\levis_1(M)=\{L\in\levis(M):\dim\aaa_M^L=1\}
\end{equation}
and
\begin{equation}\label{f1}
\cF_1(M)=\bigcup_{L\in\levis_1(M)}\cP(L).
\end{equation}
Let $\Sigma_P\subset \af_P^\ast$ be the set of reduced roots of $T_P$ on the
Lie algebra $\nf_P$ of $N_P$. Let $\Delta_P$ be the subset of simple roots of
$P$, which is a basis for $(\af_P^G)^\ast$. Denote by $\Sigma_M$ the set of
reduced roots of $T_M$ on the Lie algebra of $G$. 
For any $\alpha\in\rts_M$ we denote by $\alpha^\vee\in\aaa_M$
the corresponding co--root. Let $P_1$ and $P_2$ be parabolic subgroups with
$P_1\subset P_2$. Then $\af_{P_2}^\ast$ is embedded into $\af_{P_1}^\ast$, while
$\af_{P_2}$ is a natural quotient vector space of $\af_{P_1}$. The group
$M_{P_2}\cap P_1$ is a parabolic subgroup of $M_{P_2}$. Let $\Delta_{P_1}^{P_2}$
denote the set of simple roots of $(M_{P_2}\cap P_1,T_{P_1})$. It is a subset
of $\Delta_{P_1}$. For a parabolic subgroup $P$ with $P_0\subset P$ we write
$\Delta_0^P:=\Delta_{P_0}^P$. 

Let $\A$  be the ring of adeles of $F$, $\A_f$ the ring of finite adeles and
$F_\infty=F\otimes_\Q\R$.
We fix a maximal compact subgroup $\K=\prod_v \K_v = \K_\infty\cdot 
\K_{f}$ of $\bG(\A)=\bG(F_\infty)\cdot \bG(\A_{f})$. We assume that the maximal 
compact subgroup $\K \subset \bG(\A)$ is admissible with respect to 
$M_0$ \cite[$\S$ 1]{Ar6}. Let $\Ht_M: M(\A)\rightarrow\aaa_M$ be the 
homomorphism given by 
\begin{equation}\label{homo-M}
e^{\sprod{\chi}{\Ht_M(m)}}=\abs{\chi (m)}_\A = \prod_v\abs{\chi(m_v)}_v
\end{equation}
for any $\chi\in X(M)_F$ and denote by $M(\A)^1 \subset M(\A)$ the kernel 
of $\Ht_M$. 

Let $G_1=\Res_{F/\Q}(G)$ be the group over $\Q$ obtained from $G$ by restriction
of scalars \cite{We}. Similar for any $M\in\cL$ let $M_1:=\Res_{F/\Q}(M)$.
Let $T_{M_1}$ be the $\Q$-split component of the center of $M_1$. 
For $M\in\cL$ let $\Ai$ denote the connected component of the identity of
$T_{M_1}(\R)$, which is viewed as a subgroup of $T_M(\A_F)$ via the diagonal
embedding of $\R$ into $F_\infty$. Note that it follows from the properties of
the restriction of scalars that $M_1(\A_\Q)^1\cong M(\A_F)^1$. Thus we have
\[
M(\A_F)=\Ai\times M(\A_F)^1.
\]
Let $L^2_{\disc}(\Ai M(F)\bs M(\A))$ be the discrete part of 
$L^2(\Ai M(F)\bs M(\A))$, i.e., the
closure of the sum of all irreducible subrepresentations of the regular 
representation of $M(\A)$.
We denote by $\Pi_{\disc}(M(\A))$ the countable set of equivalence classes of 
irreducible unitary
representations of $M(\A)$ which occur in the decomposition of the discrete 
subspace $L^2_{\disc}(\Ai M(F)\bs M(\A))$ into irreducible representations.
Let $L^2_{\cu}(\Ai M(F)\bs M(\A))$ be the subspace of cusp forms. Denote by
$\Pi_{\cu}(M(\A))$ the set of equivalence classes of irreducible unitary
representations of $M(\A)$ which occur in the decomposition of the space of
cusp forms $L^2_{\cu}(\Ai M(F)\bs M(\A))$ into irreducible representations.

Let $\gf$ and $\kf$ denote the Lie algebras of $\bG(F_\infty)$ and $\K_\infty$,
respectively. Let $\theta$ be the Cartan involution of $\bG(F_\infty)$ with
respect to $\K_\infty$. It induces a Cartan decomposition
$\mathfrak{g}= \mathfrak{p} \oplus \mathfrak{k}$. 
We fix an invariant bi-linear form $B$ on $\mathfrak{g}$ which is positive 
definite on $\mathfrak{p}$ and negative definite on $\mathfrak{k}$.
This choice defines a Casimir operator $\Omega$ on $\bG(F_\infty)$,
and we denote the Casimir eigenvalue of any $\pi \in \Pi (\bG(F_\infty))$ by 
$\lambda_\pi$. Similarly, we obtain
a Casimir operator $\Omega_{\K_\infty}$ on $\K_\infty$ and write $\lambda_\tau$ for 
the Casimir eigenvalue of a
representation $\tau \in \Pi (\K_\infty)$ (cf. \cite[$\S$ 2.3]{BG}).
The form $B$ induces a Euclidean scalar product $(X,Y) = - B (X,\theta(Y))$ on 
$\mathfrak{g}$ and all its subspaces.
For $\tau \in \Pi (\K_\infty)$ we define $\norm{\tau}$ as in 
\cite[$\S$ 2.2]{CD}. Note that the restriction of the scalar product 
$(\cdot,\cdot)$ on $\gf$ to $\af_0$ gives $\af_0$ the structure of a 
Euclidean space. In particular, this fixes Haar measures on the spaces 
$\af_M^L$ and their duals $(\af_M^L)^\ast$. We follow Arthur in the 
corresponding normalization of Haar measures on the groups $M(\A)$ 
(\cite[$\S$ 1]{Ar1}).

Let $H$ be a topological group. We will denote by $\Pi(H)$ the set of
equivalence classes of irreducible unitary representations of $H$. 

Next we introduce the space $\Co(G(\A)^1)$ of Schwartz functions. 
For any compact open subgroup
$K_f$ of $G(\A_f)$ the space $G(\A)^1/K_f$ is the countable disjoint union of
copies of $G(F_\infty)^1=G(F_\infty)\cap G(\A)^1$ and therefore, it is a differentiable
manifold. Any element $X\in\mathcal{U}(\gf^1_\infty)$ of the universal 
enveloping algebra of the Lie algebra $\gf_\infty^1$ of $G(F_\infty)^1$ defines a
left invariant differential operator $f\mapsto f\ast X$ on $G(\A)^1/K_f$. Let
$\Co(G(\A)^1;K_f)$ be the space of smooth right $K_f$-invariant functions on
$G(\A)^1$ which belong, together with all their derivatives, to $L^1(G(\A)^1)$.
The space $\Co(G(\A)^1;K_f)$ becomes a Fr\'echet space under the seminorms
\[
\|f\ast X\|_{L^1(G(\A)^1)},\quad X\in\mathcal{U}(\gf^1_\infty).
\]
Denote by $\Co(G(\A)^1)$ the union of the spaces $\Co(G(\A)^1;K_f)$ as $K_f$ 
varies over the compact open subgroups of $G(\A_f)$ and endow 
$\Co(G(\A)^1)$ with the inductive
limit topology.

\section{Arithmetic manifolds}\label{sect-arithm-mfds}
\setcounter{equation}{0}

In this section we introduce the adelic description of the locally symmetric
spaces we will work with. We also explain the relation to the usual set up.

Let $G$ be a reductive algebraic group over a number field $F$. Fix a faithful
$F$-rational representation $\rho\colon G\to\GL(V)$ and an ${\cO}_F$-lattice $\Lambda$ in the representation space $V$ such that the stabilizer of
$\hat\Lambda:=\hat\cO_F\otimes\Lambda\subset\A_f\otimes V$ in $G(\A_f)$ is the
group $\K_f$. Since the maximal compact
subgroups of $\GL(\A_f\otimes V)$ are precisely the stabilizers of
lattices, it is easy to see that such a lattice exists. For any non-zero
ideal $\nf$ of $\cO_F$, let
\[
\K(\nf)=\K_G(\nf)=\{g\in G(\A_f)\colon \rho(g)v\equiv v\quad
(\hskip-10pt\mod\;\nf\hat\Lambda),\;v\in\hat\Lambda\}
\]
be the principal congruence subgroup of level $\nf$. Note that $\K(\nf)$ is a
factorizable normal subgroup of $\K_f$. Moreover, the groups $\K(\nf)$ form
a neighborhood base of the identity element in $G(\A_f)$, i.e., every
compact open subgroup $K_f\subset G(\A_f)$ contains a $K(\nf)$ for some ideal
$\nf$. We denote by $N(\nf):=[\ho_F\colon\nf]$ the ideal norm of $\nf$. 

A subgroup $\Gamma\subset G(F)$ is a congruence subgroup if it contains a
a finite-index subgroup of the form $\Gamma(\nf):=G(F)\cap \K(\nf)$ for some
ideal $\nf$. This definition of a congruence subgroup is independent of the
choice of a faithful representation, i.e., it is intrinsic to the $F$-group $G$.
Let $K_f\subset G(\A_f)$ be a  compact open subgroup. Then there exists an ideal
$\nf$ of $\cO_F$ such that $\K(\nf)\subset K_f$. Let $\Gamma_{K_f}:=G(F)\cap K_f$. Then $\Gamma(\nf)\subset \Gamma_{K_f}$ is a finite index subgroup. Thus
$\Gamma_{K_f}$ is a congruence subgroup of $G(F)$.

By \cite[$\S$ 5.6]{Bo1}  the double coset space
$G(F)\bs G(\A)/G(F_\infty)K_f$ is finite. Let $x_1=1, x_2,\dots,x_l$ be a set of 
representatives in $G(\A_f)$ of the double cosets. Then the groups
\begin{equation}\label{congr-subgr}
\Gamma_i:=\left( G(F_\infty)\times x_i K_f x_i^{-1}\right)\cap G(F),\quad 1\le i\le l,
\end{equation}
are arithmetic subgroups of $G(F_\infty)$ and the action of $G(F_\infty)$ on the
space of double
cosets $\AG G(F)\bs G(\A)/K_f$ induces the following decomposition into
$G(F_\infty)$-orbits:
\begin{equation}\label{adel-quot1}
\AG G(F)\bs G(\A)/K_f\cong \bigsqcup_{i=1}^l
\left(\Gamma_i\bs G(F_\infty)^1\right),
\end{equation}
where $G(F_\infty)^1=G(F_\infty)/\Ag$. Given a function $f$ on $G(\A)$, let
$f_i$ be the function on $G(F_\infty)$ which is defined by
$g\mapsto f(x_i\cdot g)$, $g\in G(F_\infty)$. Then the map
$f\mapsto (f_i)_{i=1}^l$ yields an isomorphism of $G(F_\infty)$-modules
\begin{equation}\label{g-modules}
L^2(\Ag G(F)\bs G(\A))^{K_f}\cong \bigoplus_{i=1}^lL^2(\Gamma_i\bs G(F_\infty)^1)
\end{equation}
\cite[4.3]{BJ}. We note that, in general, $l>1$. However, if
$G$  is semisimple,
simply connected, and without any $F$-simple factors $H$ for which $H(F_\infty)$
is compact, then by strong approximation we have 
\[
G(F)\bs G(\A)/K_f\cong \Gamma\bs G(F_\infty),
\] 
where $\Gamma=(G(F_\infty)\times K_f)\cap G(F)$. In particular this is the case
for $G=\SL(n)$.
Since \eqref{g-modules} is an isomorphism of $G(F_\infty)$-modules, it holds also
for the discrete spectrum, i.e., we have an isomorphism of $G(F_\infty)$-modules
\begin{equation}\label{g-modules-dis}
L^2_{\di}(\Ag G(F)\bs G(\A))^{K_f}\cong
\bigoplus_{i=1}^lL^2_{\di}(\Gamma_i\bs G(F_\infty)^1).
\end{equation}
Let $P\in\cL$ with Levi decomposition $P=M\ltimes N$.
Let $f\in L^2(\Ag G(F)\bs G(\A))$ correspond to
$(f_i)_{i=1}^l\in\oplus_{i=1}^lL^2(\Gamma_i\bs G(F_\infty)^1)$ as above. As
explained in \cite[Sect. 4.4]{BJ}, $f$ is a cusp function if and only if each
$f_i$, $i=1,...,l$, is a cusp function. Hence \eqref{g-modules} induces an
isomorphism
\begin{equation}\label{g-modules-cus}
L^2_{\cu}(\Ag G(F)\bs G(\A))^{K_f}\cong \bigoplus_{i=1}^l
L^2_{\cu}(\Gamma_i\bs G(F_\infty)^1)
\end{equation}
of $G(F_\infty)$-modules. Since $L^2_{\res}(\cdot)$ is the orthogonal complement
of $L^2_{\cu}(\cdot)$ in $L^2(\cdot)$, it follows from \eqref{g-modules-dis}
and \eqref{g-modules-cus}, that we also have an isomorphism of
$G(F_\infty)$-modules
\begin{equation}\label{g-modules-res}
L^2_{\res}(\Ag G(F)\bs G(\A))^{K_f}\cong \bigoplus_{i=1}^l
L^2_{\res}(\Gamma_i\bs G(F_\infty)^1).
\end{equation}
Given $\tau\in\Pi(G(F_\infty))$, let $m(\tau)$
be the multiplicity with which the representation $\tau$ occurs in
$L^2_{\di}(\Ag G(F)\bs G(\A))^{K_f}$. Then
\begin{equation}\label{multipl1}
m(\tau)=\sum_{\substack{\pi\in\Pi_{\di}(G(\A))\\\tau=
\pi_\infty}}m(\pi)\dim(\cH_{\pi^\prime_f}^{K_f}),
\end{equation}
where $\pi=\pi_\infty\otimes\pi_f$. Similarly, let $m_{\Gamma_i}(\tau)$ be the
multiplicity with which $\tau$ occurs in $L^2_{\di}(\Gamma_i\bs G(F_\infty)^1)$. 
Since \eqref{g-modules-dis} is an isomorphism of $G(F_\infty)^1$-modules, it
follows that
\begin{equation}\label{multipl2}
\sum_{\substack{\pi\in\Pi_{\di}(G(\A))\\
\pi_\infty=\tau}}m(\pi)\dim(\cH_{\pi^\prime_f}^{K_f})=\sum_{j=1}^l m_{\Gamma_j}(\tau).
\end{equation}

Let $\K_\infty\subset G(F_\infty)^1$ be a maximal compact subgroup. Let
\begin{equation}\label{symspace}
\widetilde X:=G(F_\infty)^1/\K_\infty
\end{equation}
be the associated global Riemannian symmetric space.
Given an open compact subgroup $K_f\subset G(\A_f)$, we  define the
arithmetic manifold $X(K_f)$ by
\begin{equation}\label{adel-quot2}
X(K_f):= G(F)\bs (\widetilde X \times G(\A_f)/K_f).
\end{equation}
By \eqref{adel-quot1} we have
\begin{equation}
X(K_f)=\bigsqcup_{i=1}^l \left(\Gamma_i\bs \widetilde X\right),
\end{equation}
where each component $\Gamma_i\bs \widetilde X$ is a locally symmetric space.
We will assume that $K_f$ is neat. Then $X(K_f)$ is a locally symmetric
manifold of finite volume.

Let $\nu\in\Pi(\K_\infty)$. Let $\widetilde E_\nu\to \widetilde X$ be the
homogeneous vector bundle associated to $\nu$. Denote by $C^\infty(\widetilde
X,\widetilde E_\nu)$ the space of smooth sections of $\widetilde E_\nu$. Let
\begin{align}\label{globsect}
\begin{split}
  C^{\infty}(G(F_\infty)^1,\nu):=\{f:G(F_\infty)^1\rightarrow V_{\nu}\colon f\in
  C^\infty,\:
f(gk)=&\nu(k^{-1})f(g),\\
&\forall g\in G(F_\infty)^1, \,\forall k\in \K_\infty\}.
\end{split}
\end{align}
Let $L^2(G(F_\infty)^1,\nu)$ be the corresponding $L^2$-space. There is a
canonical isomorphism
\begin{equation}\label{iso-glsect}
\widetilde A\colon C^\infty(\widetilde X,\widetilde E_\nu)\cong 
C^\infty(G(F_\infty)^1,\nu),
\end{equation}
(see \cite[p. 4]{Mia}). $\widetilde A$ extends to an isometry of the 
corresponding $L^2$-spaces.

Over each component of $X(K_f)$, $\widetilde E_\sigma$ induces a
locally homogeneous Hermitian vector bundle
$E_{i,\sigma}\to\Gamma_i\bs \widetilde X$. Let
\[
E_\sigma:=\bigsqcup_{i=1}^l E_{i,\sigma}.
\]
Then $E_\sigma$ is a vector bundle over $X(K_f)$ which is locally homogeneous.
Let $L^2(X(K_f),E_\sigma)$ be the space of square integrable sections of
$E_\sigma$.

\section{Eisenstein series and intertwining operators}\label{sect-eisenstein}
\setcounter{equation}{0}

In this section we recall some basic facts about Eisenstein series and
intertwining operators, which are the main ingredients of the spectral
side of the Arthur trace formula.

Let $M\in\cL$ and $P\in\cP(M)$ with $P=M\ltimes N_P$.
Recall that we denote  by $\rts_P\subset\af_P^*$ the set of reduced roots of 
$T_M$ on the Lie algebra $\mathfrak{n}_P$ of $N_P$.
Let $\srts_P$ be the subset of simple roots of $P$, which is a basis for 
$(\af_P^G)^*$.
Write $\af_{P,+}^*$ for the closure of the Weyl chamber of $P$, i.e.
\[
\aaa_{P,+}^*=\{\lambda\in\aaa_M^*:\sprod{\lambda}{\alpha^\vee}\ge0
\text{ for all }\alpha\in\rts_P\}
=\{\lambda\in\aaa_M^*:\sprod{\lambda}{\alpha^\vee}\ge0\text{ for all }
\alpha\in\srts_P\}.
\]
Denote by $\modulus_P$ the modulus function of $P(\A)$.
Let $\bar\AF^2(P)$ be the Hilbert space completion of
\[
\{\phi\in C^\infty(M(F)U_P(\A)\bs G(\A)):\modulus_P^{-\frac12}\phi(\cdot x)\in
L^2_{\disc}(\Ai M(F)\bs M(\A)),\ \forall x\in G(\A)\}
\]
with respect to the inner product
\[
(\phi_1,\phi_2)=\int_{\Ai M(F)N_P(\A)\bs \bG(\A)}\phi_1(g)
\overline{\phi_2(g)}\ dg.
\]

Let $\alpha\in\rts_M$.
We say that two parabolic subgroups $P,Q\in\cP(M)$ are \emph{adjacent} along 
$\alpha$, and write $P|^\alpha Q$, if $\rts_P\cap-\rts_Q=\{\alpha\}$.
Alternatively, $P$ and $Q$ are adjacent if the group $\langle P,Q\rangle$
generated by $P$ and $Q$ belongs to $\cF_1(M)$ (see \eqref{f1} for its
definition).
Any $R\in\cF_1(\M)$ is of the form $\langle P,Q\rangle$, where $P,Q$ are
the elements of $\cP(M)$ contained in $R$. We have $P|^\alpha Q$ with 
$\alpha^\vee\in\rts_P^\vee \cap\af^R_M$.
Interchanging $P$ and $Q$ changes $\alpha$ to $-\alpha$.

For any $P\in\cP(M)$ let $\Ht_P\colon G(\A)\rightarrow\af_P$ be the 
extension of $\Ht_M$ to a left $N_P(\A)$-and right $\K$-invariant map.
Denote by $\cA^2(P)$ the dense subspace of $\bar\cA^2(P)$ consisting of its 
$\K$- and $\zzz$-finite vectors,
where $\zzz$ is the center of the universal enveloping algebra of 
$\mathfrak{g} \otimes \C$.
That is, $\cA^2(P)$ is the space of automorphic forms $\phi$ on 
$N_P(\A)M(F)\bs G(\A)$ such that
$\modulus_P^{-\frac12}\phi(\cdot k)$ is a square-integrable automorphic form on
$\Ai M(F)\bs M(\A)$ for all $k\in\K$.
Let $\rho(P,\lambda)$, $\lambda\in\af_{M,\C}^*$, be the induced
representation of $G(\A)$ on $\bar\cA^2(P)$ given by
\[
(\rho(P,\lambda,y)\phi)(x)=\phi(xy)e^{\sprod{\lambda}{\Ht_P(xy)-\Ht_P(x)}}.
\]
It is isomorphic to the induced representation 
\[
\Ind_{P(\A)}^{G(\A)}\left(L^2_{\disc}(\Ai M(F)\bs M(\A))
\otimes e^{\sprod{\lambda}{\Ht_M(\cdot)}}\right).
\]
For $\phi\in\cA^2(P)$ and $\lambda\in\af^\ast_{P,\C}$, the associated
Eisenstein series is defined by
\begin{equation}\label{eisen-ser}
E(g,\phi,\lambda):=\sum_{\gamma\in P(F)\bs G(F)}\phi(\gamma g)
e^{(\lambda+\rho_p)(H_P(\gamma g))}.
\end{equation}
The series converges absolutely and locally uniformly in $g$ and $\lambda$
for $\Re(\lambda)$ sufficiently regular in the positive Weyl chamber of
$\af_P^\ast$ (\cite[II.1.5]{MW}.
By Langlands \cite{La1} the Eisenstein series  can be continued analytically
to a meromorphic
function of $\lambda\in\af^\ast_{P,\C}$. Its singularities lie along hypersurfaces
defined by root equations. 

Let $M,M_1\in\cL$. Let $W(\af_M,\af_{M_1})$ be the set of
isomorphisms from $\af_M$ onto $\af_{M_1}$ obtained by restricting elements in
$W_0$, the Weyl group of $(G,T_0)$, to $\af_M$. Each $s\in W(\af_M,\af_{M_1})$
has a representative $w_s$ in $G(F)$. Given $s\in W(\af_M,\af_{M_1})$, $P\in
\cP(M)$ and $P_1\in\cP(M_1)$, let
\begin{equation}\label{intertw0}
M_{P_1|P}(s,\lambda):\cA^2(P)\to\cA^2(P_1),\quad\lambda\in\af_{M,\C}^*,
\end{equation}
be the standard \emph{intertwining operator} \cite[$\S$ 1]{Ar9}, which is the 
meromorphic continuation in $\lambda$ of the integral
\begin{equation}
\begin{split}
[M_{P_1|P}&(s,\lambda)\phi](x)\\
&=\int_{N_{P_1}(\A)\cap w_sN_P(\A)w_s^{-1}\bs N_{P_1}(\A)}
\phi(w_s^{-1}nx)
e^{(\lambda+\rho_P)(\Ht_P(w_s^{-1}nx))}e^{-(s\lambda+\rho_{P_1})(\Ht_{P_1}(x))}\,dn,
\end{split}
\end{equation}
for $\phi\in\cA^2(P)$, $ x\in G(\A)$. Let $M=M_1$. Then for $Q,P\in\cP(M)$ and
$1\in W(\af_M)$ the identity element, we put
\begin{equation}\label{intertw1}
M_{Q|P}(\lambda):=M_{Q|P}(1,\lambda).
\end{equation}

Recall that $L^2_{\di}(A_M M(F)\bs M(\A))$ decomposes as the completed direct
sum of its $\pi$-isotopic components for $\pi\in\Pi_{\di}(M(\A))$. We have a
corresponding decomposition of $\bar\cA^2(P)$ as a direct sum of Hilbert spaces
$\hat\oplus_{\pi\in\Pi_{\di}(M(\A))}\bar\cA^2_\pi(P)$ and the corresponding algebraic
sum decomposition
\begin{equation}\label{pi-isotypic}
\cA^2(P)=\bigoplus_{\pi\in\Pi_{\di}(M(\A))}\cA^2_\pi(P).
\end{equation}
We further decompose $\cA^2_\pi(P)$ according to the action of $\K_\infty$ into
isotypic subspaces
\begin{equation}\label{isotypic-subsp}
\cA^2_\pi(P)=\bigoplus_{\nu\in\Pi(\K_\infty)}\cA^2_\pi(P)^\nu.
\end{equation}
Furthermore, for an open compact subgroup $K_f\subset G(\A_f)$ let
$\cA^2_\pi(P)^{K_f}$ be the subspace of $K_f$-invariant functions in $\cA^2_\pi(P)$
and for $\nu\in\Pi(\K_\infty)$ we let $\cA^2_\pi(P)^{K_f,\nu}$ be the $\nu$-isotypic
subspace of $\cA^2_\pi(P)^{K_f}$.

Given $\pi\in\Pi_{\di}(M(\A))$, let  $(\Ind_{P(\A)}^{G(\A)}(\pi),\cH_P(\pi))$ be the
induced representation. Let $\cH^0_P(\pi)$ be the subspace of $\cH_P(\pi)$,
consisting of all $\phi\in\cH_P(\pi)$ which are right $\K$-finite and right
$\cZ(\gf_\C)$-finite. 
There is a canonical isomorphism of 
$G(\A_f)\times(\LieG_{\C},\K_\infty)$-modules
\begin{equation}\label{ind-rep-1}
j_P:\Hom(\pi,L^2(\Ai M(F)\bs M(\A)))\otimes 
\cH_P^0(\pi)\rightarrow\cA^2_\pi(P).
\end{equation}
If we fix a unitary structure on $\pi$ and endow 
$\Hom(\pi,L^2(\Ai M(F)\bs M(\A)))$ with the inner product 
$(A,B)=B^\ast A$
(which is a scalar operator on the space of $\pi$), the isomorphism $j_P$ 
becomes an isometry. Let
\begin{equation}\label{intertw-restr}
M_{Q|P}(\pi,\lambda):=M_{Q|P}(\lambda)|_{\cA^2_\pi(P)}
\end{equation}
be the restriction of $M_{Q|P}(\lambda)$ to the subspace $\cA^2_\pi(P)$.
Suppose that $P|^\alpha Q$.
The operator $M_{Q|P}(\pi,z):=M_{Q|P}(\pi,z\varpi)$, where $\varpi\in\af_M^\star$
is such that $\langle\varpi,\alpha^\vee\rangle=1$, admits a 
normalization by a global factor
$n_\alpha(\pi,z)$ which is a meromorphic function in $z\in\C$. We may write
\begin{equation} \label{normalization}
M_{Q|P}(\pi,z)\circ j_P=n_\alpha(\pi,z)\cdot j_Q\circ(\Id\otimes R_{Q|P}(\pi,z))
\end{equation}
where $R_{Q|P}(\pi,z)=\otimes_v R_{Q|P}(\pi_v,z)$ is the product
of the locally defined normalized intertwining operators and 
$\pi=\otimes_v\pi_v$
\cite[$\S$ 6]{Ar9}, (cf.~\cite[(2.17)]{Mu2}). In many cases, the 
normalizing factors can be expressed in terms automorphic $L$-functions 
\cite{Sh1}, \cite{Sh2}.

For any $P,Q\in\cP(M)$ there exists a sequence of parabolic subgroups 
$P_0,...,P_k$ and roots $\alpha_1,...,\alpha_k\in\Sigma_M$ such that $P=P_0$, 
$Q=P_k$, and $P_{i-1}|^{\alpha_i}P_i$ for $i=1,...,k$. By the product rule for
intertwining operators we have
\begin{equation}\label{prod-form}
M_{Q|P}(\pi,\lambda)=M_{P_k|P_{k-1}}(\pi,\lambda)\circ
M_{P_{k-1}|P_{k-2}}(\pi,\lambda)\circ\cdots
\circ M_{P_1|P_0}(\pi,\lambda).
\end{equation}
Thus the study of the operators $M_{Q|P}(\pi,\lambda)$ is reduced to the case where
$Q,P\in\cP(M)$ are adjacent along some root $\alpha\in\Sigma_M$. Let
\begin{equation}\label{norm-factors}
n_{Q|P}(\pi,\lambda):=\prod_{\alpha\in\Sigma_P\cap \Sigma_{\bar Q}} 
n_\alpha(\pi,\lambda(\alpha^\vee))
\end{equation}
The product is a meromorphic function of $\lambda\in\af_{M,\C}^\ast$.
Then $M_{Q|P}(\pi,\lambda)$ is normalized by $n_{Q|P}(\pi,\lambda)$, i.e.,
\begin{equation} \label{normalization1}
M_{Q|P}(\pi,\lambda)\circ j_P=n_{Q|P}(\pi,\lambda)\cdot j_Q\circ(\Id\otimes R_{Q|P}(\pi,\lambda)).
\end{equation}
Recall that $\pi=\otimes_v\pi_v$, where $\pi_v\in\Pi(M(F_v))$. If $\H_P(\pi_v)$
is the Hilbert space of the induced representation $\Ind_{P(F_v)}^{G(F_v)}(\pi_v)$,
then one has
\[
\H_P(\pi)\cong\bigotimes_v\H_P(\pi_v)
\]
and, with respect to this isomorphism, it follows that $R_{Q|P}(\pi,\lambda)$
is the product of the corresponding local normalized intertwining operators 
\begin{equation}\label{norm-intertw-op2}
R_{Q|P}(\pi,\lambda)=\otimes_v R_{Q|P}(\pi_v,\lambda)
\end{equation}
\cite{Ar4}, \cite[$\S$ 6]{Ar9}, \cite[$\S$ 2]{Mu2}.

\section{Normalizing factors} 
\setcounter{equation}{0}

In this section we consider the global normalizing factors of intertwining
operators. The goal is to estimate the number of singular hyperplanes of
normalizing factors which intersect a given compact set.
The normalizing factors can be expressed in terms of $L$-functions.
To begin with we recall some basic facts about $L$-functions. As above, we
assume that $G$ is a reductive group over a number field $F$. Recall that
$A_G=T_{G_1}(\R)^0$, where $T_{G_1}$ is the $\Q$-split part of the connected
component of the center of $G_1=\Res_{F/\Q}(G)$, viewed as a subgroup of
$T_{G_1}(\A_\Q)$ and hence of $G(\A_F)$. 

Recall that we denote by $\Pi_{\di}(G(\A))$ the set of equivalence classes of
automorphic representations of $G(\A)$ which occur in the discrete spectrum of
$L^2(\Ag G(F)\bs G(\A))$.
For any $\pi=\otimes_v\pi_v\in\Pi_{\di}(G(\A))$ let $S(\pi)$ be the finite set of
places of $F$ containing all archimedean places and such that for each finite
place $v\in S(\pi)$ at least one of the following conditions holds:
\begin{enumerate}
\item $v$ is archimedean.
\item $F/\Q$ is ramified at $v$.
\item $G$ is ramified at $v$, i.e., either $G$ is not quasi-split over $F_v$ or
$G$ does not split over an unramified extension of $F_v$.
\item For every hyperspecial maximal compact subgroup $K_v$ of $G(F_v)$, $\pi_v$
does not have a nonzero vector which is invariant under $K_v$.
\end{enumerate}
Let $S_\infty$ denote the set of archimedean places of $F$ and let $S_f(\pi)$
denote the set of non-archimedean places in $S(\pi)$. Thus $S(\pi)=S_\infty\cup
S_f(\pi)$. For any $v\in S_f(\pi)$ let $q_v$ denote the order of the residue
field of $F_v$. Let $S_{\Q,f}(\pi)$ be the set of rational primes which lie
below the primes in $S_f(\pi)$. Also set
$S_\Q(\pi):=\{\infty\}\cup S_{\Q,f}(\pi\}$.

Let $W_F$ be the Weil group of $F$ and let $^L G$ be the Langlands $L$-group
of $G$ \cite{Bo2}. Let $r\colon ^L G\to \GL(N,\C)$ be a continuous and 
$W_F$-semisimple $N$-dimensional complex representation of $^L G$. For any
$\pi\in\Pi_{\di}(G(\A))$ and any place $v$ of $F$ with $v\not\in S(\pi)$
let $t_{\pi_v}\in ^L G$ be the Hecke-Frobenius parameter of $\pi_v$. Then the
local $L$-function $L_v(s,\pi,r)$ is defined by
\begin{equation}\label{local-l}
L_v(s,\pi,r):=\det\left(\Id-r(t_{\pi_v})q_v^{-s}\right)^{-1}.
\end{equation}
Since $\pi$ is unitary, the $|r(t_{\pi_v})|$ are bounded by $q_v^c$, where $c$ 
depends only on $G$ and $r$, \cite{Bo2}, \cite{La2}. Therefore, for $S\supset 
S(\pi)$ the partial $L$-function
\begin{equation}
L^S(s,\pi,r):=\prod_{v\not\in S} L_v(s,\pi,r)
\end{equation}
converges absolutely and uniformly on compact subsets of $\Re(s)>c+1$. One of
the goals of the Langlands program is to show that each of these $L$-functions
admits a meromorphic extension to the entire complex plane and satisfies a
functional equation. This is far from being proved. In 
\cite[Definition 2.1]{FL1}, Finis and Lapid formulated a precise version of the
expected functional equation. According to this definition, $(G,r)$ has
property {\bf (FE)}, if for any $\pi\in\Pi_{\di}(G(\A_F)$ the partial
$L$-function $L^{S(\pi)}(s,\pi,r)$ admits a meromorphic continuation to $\C$
with a functional equation of the form
\begin{equation}\label{funct-equ3}
L^{S(\pi)}(s,\pi,r)=\left(\prod_{p\in S_\Q(\pi)}\gamma_p(s,\pi,r))\right)
L^{S(\pi)}(1-s,\pi,r^\vee),
\end{equation}
where for each $p\in S_{\Q,f}(\pi)$, $\gamma_p(s,\pi,r)=R_p(p^{-s})$ for some
rational function $R_p$ and
\begin{equation}\label{gamma-fact-inf}
\gamma_\infty(s,\pi,r)=C_\infty\prod_{i=1}^m\frac{\Gamma_\R(1-s+\alpha_i^\vee)}
{\Gamma_\R(s+\alpha_i)}
\end{equation}
for certain parameters $\alpha_1,...,\alpha_m,\alpha_1^\vee,...,\alpha_m^\vee\in
\C$ and a constant $C_\infty$. By \cite[Lemma 2.2]{FL1} the parameters
$\alpha_1^\vee,...,\alpha_m^\vee$ are determined by $\alpha_1,...,\alpha_m$.
Moreover, the integer $m$ is uniquely determined. The parameters $\alpha_1,...,
\alpha_m$ are said to be {\it reduced}, if $\alpha_i+\alpha_j$ is not a negative
odd integer for any $1\le i,j\le m$. By \cite[Lemma 2.2]{FL1} one may choose the
parameters $\alpha_1,...,\alpha_m$ to be reduced. Assuming that this is
satisfied, Finis and Lapid introduce the reduced $L$-factor at the Archimedean
place by
\begin{equation}\label{red-l-fct-inf}
L_{\infty}^{\red}(s,\pi,r):=\prod_{i=1}^m \Gamma_\R(s+\alpha_i).
\end{equation}
Now let $p\in S_{\Q,f}(\pi)$. Then by \cite[(2.7)]{FL1}, 
$\gamma_p(s,\pi,r)$ can be written in a unique way as
\begin{equation}
\gamma_p(s,\pi,r)=c_pp^{\left(\frac{1}{2}-s\right){\mathfrak e}_p(\pi,r)}P_p(p^{-s})/
\bar P_p(p^{s-1}),
\end{equation}
where $c_p\in\C^\ast$, ${\mathfrak e}_p(\pi,r)\in\Z$, and $P_p$ is a polynomial with
$P_p(0)=1$ such that no zeros $\alpha$ and $\beta$ of $P_p$ satisfy $\alpha
\bar\beta=p^{-1}$. Then Finis and Lapid define the reduced $L$-factor at $p$ by
\begin{equation}\label{red-l-fct-finite}
L_p^{\red}(s,\pi,r):=P_p(p^{-s})^{-1},\quad p\in S_{\Q,f}(\pi),
\end{equation}
and introduce the {\it reduced completed $L$-function} by
\begin{equation}\label{red-l-fct}
L^{\red}(s,\pi,r):=\bigl(\prod_{p\in S_\Q(\pi)} L_p^{\red}(s,\pi,r)\bigr)
L^{S(\pi)}(s,\pi,r).
\end{equation}
There is also a corresponding reduced epsilon factor $\epsilon^{\red}(s,\pi,r)$,
which is defined by
\begin{equation}\label{red-eps-factor}
\epsilon^{\red}(s,\pi,r)=c_\infty\prod_{p\in S_{\Q,f}(\pi)} c_p
p^{\left(\frac{1}{2}-s\right){\mathfrak e}_p(\pi,r)}={\mathfrak n}(\pi,r)^{\frac{1}{2}-s}
\prod_{p\in S_\Q(\pi)}c_p,
\end{equation}
where
\begin{equation}\label{red-eps-fct2}
{\mathfrak n}(\pi,r)=\prod_{p\in S_{\Q,f}(\pi)} p^{{\mathfrak e}_p(\pi,r)}.
\end{equation}
Then the functional equation \eqref{funct-equ3} becomes
\begin{equation}\label{funct-equ-2}
L^{\red}(s,\pi,r)=\epsilon^{\red}(s,\pi,r)L^{\red}(1-s,\pi,r^\vee).
\end{equation}
In \cite[Definition 2.4]{FL1} a stronger version of property {\bf (FE)}
is introduced. The pair $(G,r)$ is said to satisfy property
{\bf (FE+)}, if it satisfies {\bf (FE)} and in addition some uniformity
conditions for $\gamma_\infty$ and $P_p$ are fulfilled. For the precise
statement see \cite[Definition 2.4]{FL1}.

The normalizing factors are described in \cite[Sect. 3]{FL1}. To recall the
description, we need to introduce some notation. Let $M\in\cL$
and $\alpha\in\Sigma_M$. Let $\tilde M_\alpha$ be the Levi subgroup of $M$ of
co-rank one, defined in \cite[p. 254]{FL1}, together with the map
$p^{\sico}\colon \tilde M^{\sico}_\alpha\to\tilde M_\alpha$, which is also defined
in \cite[p. 254]{FL1}. Furthermore, let $U_\alpha$ be the unipotent subgroup of
$G$ corresponding to  $\alpha$. Thus the eigenvalues
of $T_M$ acting on the Lie algebra of $U_\alpha$ are positive integer multiples
of $\alpha$. The adjoint action of $^L M$ on $\Lie(^L U_\alpha)$ factors through
the composed homomorphism $^L M\to ^L \tilde M_\alpha$. The contragredient
of the adjoint representation of $^L\tilde M_\alpha$ on $\Lie(^L U_\alpha)$ is
decomposed as $\oplus_{j=1}^l r_j$ into irreducible representations $r_j$.

By T. Finis and E. Lapid \cite[Definition 3.4]{FL1}, $G$ satisfies
property (L), if for any standard Levi subgroup $M$, any $\alpha\in\Sigma_M$,
and any irreducible constituent $r=r_j$ as above, the pair $(\tilde M_\alpha,r)$
satisfies properties {\bf (FE+)} \cite[Definition 2.4]{FL1} and the
conductor condition {\bf (CC)} \cite[Definition 2.9]{FL1}.

Assume that $G$ satisfies property {\bf (L)}. Then one can describe the
normalizing
factors in terms of $L$-functions. Let $\pi\in\Pi_{\di}(M(\A))$ and let
$n_\alpha(\pi,s)$ be the normalizing factor as in \eqref{normalization}. First
note that $n_\alpha(\pi,s)$ satisfies the functional equation
\begin{equation}\label{funct-equ-5}
n_\alpha(\pi,s)\overline{n_\alpha(\pi,-\bar{s})}=1.
\end{equation}
Next recall that for $\Re(s)\gg 0$, $n_\alpha(\pi,s)$ factorizes as 
\begin{equation}
n_\alpha(\pi,s)=\prod_v n_{\alpha,v}(\pi_v,s).
\end{equation}
By \cite[Lemma 2.13]{FL1} there exist $\sigma\in\Pi_{\di}(\tilde M_\alpha(\A))$
and a character $\chi$ of $\tilde M_\alpha(\A)$, which is trivial on
$\tilde M_\alpha(F) p^{\sico}(\tilde M_\alpha^{\sico})$, such 
that $\sigma\chi$ is a subrepresentation of $\pi|_{\tilde M_\alpha(\A)}$. Let 
\begin{equation}\label{fct-m}
m(\sigma,s):=\prod_{j=1}^l\frac{\epsilon^{\red}(1,\sigma,r_j)
L^{\red}(js,\sigma,r_j)}{L^{\red}(js+1,\sigma,r_j)}.
\end{equation}
Then at the end of the proof of Proposition 3.8 in \cite[p. 259]{FL1} it has 
been shown that there exists $C>0$ such that
\begin{equation}\label{norm-fact2}
n_\alpha(\pi,s)=C\cdot m(\sigma,s)\cdot\prod_{j=1}^l\prod_{p\in S_\Q(\sigma)}
\frac{L_p^{\red}(js+1,\sigma,r_j)}{L_p^{\red}(js,\sigma,r_j)}
\prod_{v\in S(\sigma)} n_{\alpha,v}(\pi,s).
\end{equation}
We use this formula to estimate the number of poles of $n_\alpha(\pi,s)$. 
First we consider the two finite products. For this purpose we need to estimate
the cardinality of $S(\sigma)$ and $S_\Q(\sigma)$. Recall that the level of 
$\sigma$ is defined as $\level(\sigma)=N(\nf)$, where $\nf$ is the largest
ideal of $\cO_F$ such that $\sigma^{\K(\nf)\cap \tilde M_\alpha}\neq 0$. We obviously
have
\[
|S(\sigma)|, |S_\Q(\sigma)|\le C(1+\log\level(\sigma))
\]
for some constant $C>0$ which is independent of $\sigma$. Furthermore, recall
that by \cite[Sec. 2.3]{FL1},
$\level(\pi;p^{\sico})$ is defined as  $\level(\pi;p^{\sico})=N(\nf)$, where
$\nf$ is the largest ideal in $\cO_F$ such that
$\pi^{\K(\nf)\cap p^{\sico}(\tilde M^{\sico}_\alpha)}\neq 0$. 

By \cite[Lemma 2.13]{FL1}
there exists $N_1\in\N$, which depends only on $p^{\sico}$ and $G$, such that
for every $\pi\in\Pi_{\di}(M(\A))$ the corresponding representation $\sigma
\in\Pi_{\di}(\tilde M_\alpha(\A))$ is such that $\level(\sigma)$ divides
$N_1\level(\pi;p^{\sico})$. Thus there exists $C>0$ such that
\begin{equation}\label{S-estim}
|S_{\Q}(\sigma)|,|S(\sigma)|\le C\log\level(\pi;p^{\sico})
\end{equation}
for all $\pi$ and $\sigma$ which are related as above.

Now consider the last product on the right hand side of \eqref{norm-fact2}.
Let $v\in S_f(\sigma)$. By \cite[p. 255, (1)]{FL1},
$n_{\alpha,v}(\pi,s)$ is a rational function in $X=q_v^{-s}$, whose degree is
bounded in terms of $G$ only and which is regular and non-zero at $X=0$. 
Hence the poles of $n_{\alpha,v}(\pi,s)$ form a finite union of
arithmetic progressions with imaginary difference and the number of progressions
is bounded by a constant that depends only on $G$. 

If $v\in S_\infty(\sigma)$, then by \cite[p. 255, (2)]{FL1} we have
\begin{equation}
n_{\alpha,v}(\pi,s)=c_v\prod_{i=1}^{N_v}\frac{\Gamma_\R(j_is+\alpha_i)}
{\Gamma_\R(j_is+\alpha_i+1)},
\end{equation}
where $c_v\neq 0$, $\alpha_1,...,\alpha_{N_v}\in\C$ and the integers $N_v\ge 1$
and $j_i\ge 1$, $i=1,...,N_v$ are bounded in terms of $G$ only.
Recall that
$\Gamma(z)$ has no zeros and the poles are simple and occur at the negative
integers. Let $R>0$. It follow that the number of 
poles of $n_{\alpha,v}(\pi,s)$ in a fixed half-strip $|\Im(s)|\le R$,
$\Re(s)\ge -R$ is bounded by a constant independent of $\pi$. 
Thus by \eqref{S-estim} it follows that for every $R>0$ there exists $C_1>0$,
which is independent of $\pi$, such that the number of poles in the half-strip
$|\Im(s)|\le R$, $\Re(s)\ge -R$, counted with their order, of the last product 
is bounded by $C_1\log\level(\pi;p^{\sico})$.

Next we deal with the product over 
$S_\Q(\sigma)$,  Let $p\in S_{\Q,f}(\sigma)$. By \eqref{red-l-fct-finite},
there exist a polynomial $P_p(x;\sigma,r_j)$ such that 
$L_p^{\red}(s,\sigma,r_j)=P_p(p^{-s};\sigma,r_j)^{-1}$. By definition, $(G,r_j)$
satisfies property {\bf (FE+)} \cite[Definition 2.4]{FL1}. By (2) of this
definition, the degree of $P_p(x;\sigma,r_j)$ is bounded
in terms of $(G,r_j)$ only. Thus the poles of $L_p(s,\sigma,r_j)$ form a finite
union of arithmetic progressions with imaginary difference and the number of
progressions is bounded by a constant that depends only on $(G,r_j)$. Hence
for every $R>0$ there exists $C>0$, which depends only on $(G,r_j)$,
 such that the number of poles of $L_p(s,\sigma,r_j)$ in the 
strip $|\Im(s)|\le R$ is bounded by $C$. For $p=\infty$ we use
\eqref{red-l-fct-inf}. By \cite[Definition 2.4, (3)]{FL1}, there exists
$\beta\in\R$ which depends only on $(G,r_j)$ such that the reduced
parameters $\alpha_i$ satisfy
$\Re(\alpha_i)\ge -\beta$, $i=1,...,l$,  and $\gamma_\infty(s,\sigma,r_j)$ has
no zeros in $\Re(s)>\beta$. So it follows as above, that the number of poles
of $L_p^{\red}(js+1,\sigma,r_j)/L_p^{\red}(js,\sigma,r_j)$ in the half-strip
$|\Im(s)|\le R$, $\Re(s)\ge -R$, counted with their order, is bounded by a
constant independent of $\pi$. Using \eqref{S-estim} it follows that for each
$R>0$ there exists $C_2>0$ such that the number of poles of the product over
$S_{\Q}(\sigma)$, counted with their order, in the half-strip
$|\Im(s)|\le R$, $\Re(s)\ge -R$, is bounded by $C_2\log\level(\pi;p^{\sico})$. 

So it remains to consider $m(\sigma,s)$. Let $r:=r_j$ for some $j$ and let
\begin{equation}\label{lambda1}
\Lambda(s,\sigma,r):=\nf(\sigma,r)^{s/2}L^{\red}(s,\sigma,r),
\end{equation}
where $\nf(\sigma,r)$ is defined by \eqref{red-eps-fct2}
Then, using functional equation \eqref{funct-equ-2} and the definition of the
epsilon factor by \eqref{red-eps-factor}, it follows that $\Lambda(s,\sigma,r)$
satisfies
\begin{equation}\label{funct-equ-lambda}
\Lambda(s,\sigma,r)=\epsilon^{\red}(\frac{1}{2},\sigma,r)\overline{\Lambda(
1-\bar s,\sigma,r)}.
\end{equation}
By \eqref{red-eps-factor} and \eqref{lambda1} we get
\begin{equation}
\frac{\Lambda(s,\sigma,r)}{\overline{\Lambda(-\bar s,\sigma,r)}}=
\frac{\epsilon^{\red}(1,\sigma,r)L^{\red}(s,\sigma,r)}{L^{\red}(s+1,\sigma,r)}.
\end{equation}

Thus by the definition \eqref{fct-m} it follows that
\begin{equation}\label{m-prod}
m(\sigma,s)=\prod_{j=1}^l\frac{\Lambda(js,\sigma,r_j)}
{\overline{\Lambda(-j\bar s,\sigma,r_j)}}.
\end{equation}
As explained in the proof of \cite[Proposition 2.6]{FL1}, $\Lambda(s,\sigma,r)$
is the quotient of two holomorphic functions of order one. Therefore 
$\Lambda(s,\sigma,r)$ admits a Hadamard factorization
\begin{equation}\label{hadam-fact}
\Lambda(s,\sigma,r)=e^{a+bs}s^{n(0)}\prod_{\rho\neq0}
\bigl[(1-s/\rho)e^{s/\rho}\bigr]^{n(\rho)},
\end{equation}
where $a,b\in\C$, the product ranges over the zeros and poles of $\Lambda(s)$
different from 0, and $n(\rho)$ is the order of the function $\Lambda(s)$ at
$s=\rho$. In the poof of \cite[Proposition 2.6]{FL1} it was shown that the
conditions of property {\bf (FE+)}, \cite[Definition 2.4]{FL1}, together with
the functional equation 
\eqref{funct-equ-lambda} imply that there exists $A\ge 1$, depending only
on $G$ and $r$,  such that all zeros and poles of $\Lambda(s,\sigma,r)$ lie
in the strip $1-A\le\Re(s)\le A$. Moreover, by \cite[(2.12)]{FL1} we have for
$T\ge 0$
\begin{equation}\label{sum-order-poles}
\sum_{\rho\colon |\Im(\rho)-T|<2}|n(\rho)|\ll\log\nf(\sigma,r)+\sum_{p\in S_{\Q,f}(\sigma)}
\log p+\log\cf_\infty(\sigma,r)+\log(1+|T|)+1,
\end{equation}
where  $\cf_\infty(\pi,r)$ is the archimedean conductor defined by
\cite[(2.6)]{FL1} and ${\mathfrak n}(\sigma,r)$ the finite conductor 
\eqref{red-eps-fct2}.
Since we assume that $G$ satisfies property {\bf (L)}, $(\tilde M_\alpha,r_j)$ 
satisfies property (CC). Let $\Lambda(\pi_\infty;p^{\sico})$ be defined by
\cite[(2.18)]{FL1}. Then by \cite[(2.14), (2.15)]{FL1} and
\cite[Lemma 2.13]{FL1} we obtain
\begin{equation}\label{sum-order-poles-1}
\sum_{\rho\colon|\Im(\rho)-T|<2} |n(\rho)|\ll \log(|T|+2)+
\log\level(\pi;p^{\sico})+\log\Lambda(\pi_\infty;p^{\sico}).
\end{equation}

We combine \eqref{sum-order-poles-1} with the results above concerning the
other factors occurring in \eqref{norm-fact2}. Note that by the functional
equation \eqref{funct-equ-5}, the poles of $n_\alpha(\pi,s)$ are contained
in a strip $|\Re(s)|\le C$ for some $C>0$. We can summarize our results as
follows.
Denote by $\Sigma_\alpha(\pi)$ the poles of $n_\alpha(\pi,s)$. Given $\rho\in
\Sigma_\alpha(\pi)$, denote by $n(\rho)$ its order. Then combined with the
results above concerning the other factors occurring in \eqref{norm-fact2},
we obtain the following proposition.
\begin{prop}
Assume that $G$ satisfies property (L). Let $M\in\cL$, $\pi\in\Pi_{\di}(M(\A))$, and $\alpha\in\Sigma_M$. Let $\Sigma_\alpha(\pi)$ be the set of poles of
$n_\alpha(\pi,s)$ and for any $\rho\in\Sigma_\alpha(\pi)$ denote by $n(\rho)$
the order of the pole $\rho$. Then for every $R>0$ there exist $C>0$ such that
\begin{equation}\label{poles-norm-fct}
\sum_{\rho\in\Sigma_\alpha(\pi),|\Im(\rho)|<R} |n(\rho)|\le C(1 +
\log\level(\pi;p^{\sico})+\log\Lambda(\pi_\infty;p^{\sico})).
\end{equation}
\end{prop}
Let $\pi\in\Pi_{\di}(M(\A))$. Let $W_P(\pi_\infty)$ be the set of minimal $K_\infty$-types of
$\Ind_{P(\R)}^{G(\R)}(\pi_\infty)$. Then $W_P(\pi_\infty)$ is a non empty finite subset of $\Pi(K_\infty)$.
Let $\lambda_{\pi_\infty}$ be the Casimir eigenvalue of $\pi_\infty$ and for each
$\tau\in\Pi(K_\infty)$, let $\lambda_\tau$ be the Casimir eigenvalue of $\tau$.
Put
\begin{equation}\label{lambda-1}
\Lambda_{\pi_\infty}:=\min_{\tau\in W_P(\pi_\infty)}\sqrt{\lambda_{\pi_\infty}^2+\lambda_\tau^2}.
\end{equation}
Then by \cite[(10)]{FLM2} one has
\begin{equation}\label{lambda-2}
\Lambda(\pi_\infty;p^{\sico})\ll_G 1+\Lambda^2_{\pi_\infty}.
\end{equation}
Let $K_f$ be an open compact subgroup of $G(\A_f)$. Put
\begin{equation}\label{repr-subset1}
\Pi(M(\A);K_f):=\{\pi\in\Pi(M(\A)\colon \pi_f^{K_f\cap M(\A_f)}\neq 0\}.
\end{equation}
Furthermore, given $\nu\in\Pi(\K_\infty)$, let
\begin{equation}\label{repr-subset2}
\Pi(M(\A);K_f,\nu)=\{\pi\in\Pi(M(\A);K_f)\colon [\Ind_{P(\R)}^{G(\R)}(\pi_\infty)|_{K_\infty}\colon\nu]>0\}
\end{equation}
and put
\begin{equation}\label{repr-subset-3}
\Pi_{\di}(M(\A);K_f,\nu)=\Pi_{\di}(M(\A))\cap \Pi(M(\A);K_f,\nu).
\end{equation}
Now recall the definition of $\level(\pi;p^{\sico})$ \cite[Sec. 2.3]{FL1}.
It follows that there exists $C>0$ such that all for $\pi\in\Pi(M(\A);K_f)$
we have $\level(\pi;p^{\sico})\le C$. Furthermore, by \eqref{lambda-1} there
exists $C_1>0$ such that for all $\pi\in\Pi(M(\A);K_f,\nu)$ one has
$\Lambda_{\pi_\infty}^2\le C_1(1+\lambda_{\pi_\infty}^2)$, and it follows from
\eqref{lambda-2} that in this case $\Lambda(\pi_\infty;p^{\sico})\ll_G
(1+\lambda_{\pi_\infty}^2)$. In this way we get
\begin{cor}\label{cor-est-poles}
Assume that $G$ satisfies property (L). Let $K_f$ be an open compact subgroup of
$G(\A_f)$ and $\nu\in\Pi(K_\infty)$. Let $M\in\cL$ and $\alpha\in\Sigma_M$.
Let the notation be as above. For every $R>0$ there exist $C>0$ such that
\begin{equation}\label{poles-norm-fct1}
\sum_{\rho\in\Sigma_\alpha(\pi),|\Im(\rho)|<R} |n(\rho)|\le
C(1 +\log(1+\lambda_{\pi_\infty}^2))
\end{equation}
for all $\pi\in\Pi_{\di}(M(\A);K_f,\nu)$.
\end{cor}

\section{Logarithmic derivatives of local intertwining operators}
\label{sec-locint} 
\setcounter{equation}{0}

In this section we prove some auxiliary results for local intertwining 
operators. To begin with we recall some facts concerning local intertwining
operators and normalizing factors. Let $M\in\cL$ and $P,Q\in\cP(M)$. Let
$v$ be a place of $F$. If $v$ is finite, let $K_v$ be an open compact subgroup
of $G(F_v)$ and if $v\in S_\infty$, let $K_v$ be a maximal compact subgroup of
$G(F_v)$. 
Let $\pi_v\in\Pi(M(F_v))$. Given $\lambda\in \af^\ast_{M,\C}$, let
$(I_P^G(\pi_v,\lambda),\H_P(\pi_v))$ denote the induced
representation. Let $\H_P^0(\pi_v)\subset\H_P(\pi_v)$ be the subspace of 
$K_v$-finite functions. Let 
\[
J_{Q|P}(\pi_v,\lambda)\colon \H_P^0(\pi_v)\to\H_Q^0(\pi_v)
\]
be the local intertwining operator between the induced representations 
$I_P^G(\pi_v,\lambda)$ and $I_Q^G(\pi_v,\lambda)$ \cite{Sh1}. It is proved 
in \cite{Ar4}, \cite[Lecture 15]{CLL} that there exist scalar valued 
meromorphic functions $r_{Q|P}(\pi_v,\lambda)$ of $\lambda\in\af_{P,\C}^\ast$ 
such that the normalized intertwining operators
\begin{equation}\label{norm-inter}
R_{Q|P}(\pi_v,\lambda)=r_{Q|P}(\pi_v,\lambda)^{-1}J_{Q|P}(\pi_v,\lambda)
\end{equation}
satisfy the conditions $(R_1)-(R_8)$ of Theorem~2.1 of
\cite{Ar4}. We recall some facts about the local normalizing factors. First
assume that $v$ is a finite valuation of $F$ with $q_v\in\N$ the cardinality
of the residue field of $F_v$. Furthermore assume that 
$\dim(\af_M/\af_G)=1$ and $\pi_v$ is square integrable. Let $P\in\cP(M)$ and
let $\alpha$ be the unique simple root of $(P,T_M)$. Then Langlands
\cite[Lecture 15]{CLL} has shown that there exists a rational function 
$V_P(\pi_v,z)$ of one variable such that 
\begin{equation}\label{rat-funct}
r_{\ov P|P}(\pi_v,\lambda)=V_P(\pi_v,q_v^{-\lambda(\widetilde\alpha)}),
\end{equation}
where $\widetilde\alpha\in\af_M$ is uniquely determined by $\alpha$. 
For the construction of $V_P$ see also \cite[Sect. 3]{Mu2}. In this reference,
only the case $\Q_v$ has been discussed. However, the case $F_v$ can be dealt
with in exactly the same way. We need the
following lemma.
\begin{lem}\label{lem-loc-int}
Let $M\in\cL$ be such that $\dim(\af_M/\af_G)=1$.
There exists $C>0$ such that for all $P\in\cP(M)$ and all 
$\pi\in\Pi(M(F_v))$ the number of zeros of the rational function 
$V_P(\pi,z)$ is less than or equal to $C$.
\end{lem}
For the proof see \cite[Lemma 10.1]{MM2}. Again, the proof has been carried out
for $\Q_v$. It extends to $F_v$ without any changes.

The main goal of this section is to estimate the logarithmic derivatives
of the normalized intertwining operators $R_{Q|P}(\pi,\lambda)$. For $\bG=
\GL(n)$ such estimates were derived in \cite[Proposition~0.2]{MS}. The proof
depends on a weak version of the 
Ramanujan conjecture, which is not available in general. Therefore we will
establish only an integrated version of it, which however, is sufficient for
our purpose. For $\pi\in\Pi_{\di}(M(\A))$ denote by $\H_P(\pi)$ the Hilbert
space of the induced representation $I_P^G(\pi,\lambda)$. Furthermore, for 
an open compact subgroup $K_f\subset G(\A_f)$ and $\nu\in\Pi(\K_\infty)$, denote
by $\H_P(\pi)^{K_f}$ the subspace of vectors, which are invariant under $K_f$
and let $\H_P(\pi)^{K_f,\nu}$ denote the $\nu$-isotypical subspace of 
$\H_P(\pi)^{K_f}$. Let $P,Q\in\cP(M)$ be adjacent parabolic subgroups. Then 
$R_{Q|P}(\pi,\lambda)$ depends on a single variable $s\in\C$ and we will
write 
\[
R^\prime_{Q|P}(\pi,s_0):=\frac{d}{ds}R_{Q|P}(\pi,s)\big|_{s=s_0}
\]
for any regular $s_0\in\C$.
\begin{prop}\label{prop-log-der}
Let $M\in\cL$, and let $P,Q\in\cP(M)$ be adjacent parabolic subgroups. Let
$K_f\subset \bG(\A_f)$ be an open compact subgroup and let 
$\nu\in\Pi(\K_\infty)$.
Then there exists $C>0$ such that
\begin{equation}\label{int-log-der}
\int_{\R}\Big\|R_{Q|P}(\pi,iu)^{-1}R^\prime_{Q|P}(\pi,iu)\big|_{\H_P(\pi)^{K_f,\nu}}
\Big\|e^{-tu^2}du\le C t^{-1/2}
\end{equation}
for all $0<t\le 1$ and $\pi\in\Pi_{\di}(M(\A))$ with $\H_P(\pi)^{K_f,\nu}\neq 0$. 
\end{prop}
\begin{proof}
We may assume that $K_f$ is factorisable, i.e., $K_f=\prod_v K_v$. Let $S$ be
the finite set of finite places such that $K_v$ is not hyperspecial. Since
$P$ and $Q$ are adjacent, by standard properties of normalized intertwining
operators \cite[Theorem~2.1]{Ar4} we may assume that $P$ is a maximal
parabolic subgroup and $Q=\overline P$, the opposite parabolic subgroup to
$P$. By
\cite[Theorem 2.1, (R8)]{Ar4}, $R_{\oP|P}(\pi_v,s)^{K_v}$ is independent of $s$ if
$v$ is finite and $v\notin S$. Thus we have

\begin{equation}\label{log-der-inter1}
\begin{split}
R_{\oP|P}(\pi,s)^{-1}R^\prime_{\oP|P}(\pi,s)\big|_{H_P(\pi)^{K_f,\nu}}&=
R_{\oP|P}(\pi_\infty,s)^{-1}R^\prime_{\oP|P}(\pi_\infty,s)\big|_{\H_P(\pi_\infty)^\nu}\\
&+\sum_{v\in S}R_{\oP|P}(\pi_v,s)^{-1}R^\prime_{\oP|P}(\pi_v,s)\big|_{\H_P(\pi_v)^{K_v}}
\end{split}
\end{equation}
This reduces our problem to the operators at the local places. We distinguish
between the archimedean and the non-archimedean case.

{\bf Case 1:} $v<\infty$. Define $A_v\colon \C\to \End(\cH_P(\pi_v)^{K_v})$ by
\[
A_v(q_v^{-s}):=R_{\oP|P}(\pi_v,s)\big|_{\H_P(\pi_v)^{K_v}}.
\]
This is a meromorphic function with values in the space of endomorphisms of a
finite dimensional vector space. It has the following properties. 
By the unitarity of $R_{\oP|P}(\pi_v,iu)$, $u\in\R$, it follows that $A_v(z)$ is
holomorphic for $z\in S^1$ and satisfies $\|A_v(z)\|\le 1$, $|z|=1$. By
\cite[Theorem~2.1]{Ar4}, the matrix coefficients of $A_v(z)$ are rational 
functions. Recall that the operators $R_{\oP|P}(\pi_v,iu)$ are unitary. As in
\cite[(14)]{FLM2} we get
\begin{equation}
\begin{split}
\int_{\R}\Big\|&R_{\oP|P}(\pi_v,iu)^{-1}R^\prime_{\oP|P}(\pi_v,iu)\big|_{\H_P(\pi_v)^{K_v}}
\Big\|e^{-tu^2}du=\int_\R\left\|R^\prime_{\oP|P}(\pi_v,iu)\Big|_{\H_P(\pi_v)^{K_v}}\right\|e^{-tu^2}du\\
&\le 2\sum_{n=0}^\infty \exp\left(-t\frac{4\pi^2 n^2}{(\log q_v)^2}\right)
\int_0^{\frac{2\pi}{\log q_v}}\left\|R^\prime_{\oP|P}(\pi_v,iu)
\Big|_{\H_P(\pi_v)^{K_v}}\right\|du\\
&\le 2\left(1+\int_0^\infty \exp\left(-t\frac{4\pi^2 x^2}{(\log q_v)^2}\right)dx
\right)\int_0^{\frac{2\pi}{\log q_v}}\left\|R^\prime_{\oP|P}(\pi_v,iu)
\Big|_{\H_P(\pi_v)^{K_v}}\right\|du\\
&=\left(2+\frac{\log q_v}{\pi}\cdot t^{-1/2}\right)\int_{S^1}\left\|A_v^\prime(z)
\right\| |dz|.
\end{split}
\end{equation}
As explained above, $A_v$ satisfies the assumptions of 
\cite[Corollary~5.18]{FLM2}. Denote by $z_1,...,z_m\in \C\setminus S^1$ be the
poles
of $A_v(z)$. Then $(z-z_1)\cdots(z-z_m)A_v(z)$ is a polynomial of degree $n$
with coefficients in $\End(\cH_P(\pi_v)^{K_v})$ and by 
\cite[Corollary~5.18]{FLM2} we get
\begin{equation}\label{est-int1}
\|A_v^\prime(z)\|\le\max\left(\max(n-m,0)+\sum_{j\colon |z_j|>1}
\frac{|z_j|^2-1}{|z_j-z|^2},\; \sum_{j\colon |z_j|<1}\frac{1-|z_j|^2}{|z_j-z|^2}
\right),\quad z\in S^1.
\end{equation}
Now observe that 
\[
\frac{1}{2\pi}\int_{S^1}\frac{1-|z_0|^2}{|z-z_0|^2}|dz|=1.
\]
$z_0\in \C$, $|z_0|<1$. This follows from the fact that
the integrant is the Poisson kernel and so the integral is the unique
harmonic function on the unit disc which is equal to 1 on the boundary. This
is the constant function 1. Hence by \eqref{est-int1} we get
\begin{equation}\label{est-int2}
\int_{S^1}\|A^\prime_v(z)\|\;|dz|\le 2\pi \max(m,n).
\end{equation}

Next we estimate $m$ and $n$. First consider $m$. Let $J_{\oP|P}(\pi_v,s)$ be
the usual intertwining  operator so that
\[
R_{\oP|P}(\pi_v,s)=r_{\oP|P}(\pi_v,s)^{-1}J_{\oP|P}(\pi_v,s),
\]
where $r_{\oP|P}(\pi_v,s)$ is the normalizing factor \cite{Ar4}. By 
\cite[Theorem~2.2.2]{Sh1} there exists a polynomial $p(z)$ with $p(0)=1$ 
whose degree is bounded independently of $\pi_v$, such that 
$p(q_v^{-s})J_{\oP|P}(\pi_v,s)$ is holomorphic on $\C$. To deal with the normalizing factor we use \eqref{rat-funct} together with Lemma~\ref{lem-loc-int} to
count the
number of poles of $r_{\oP|P}(\pi_v,s)^{-1}$. This leads to a bound for $m$
which depends only on $\bG$. To estimate $n$ we fix an open compact subgroup
$K_v$ of $\bG(F_v)$. Our goal is now to estimate the order at $\infty$ of any
matrix coefficient of $R_{\oP|P}(\pi_v,s)$ regarded as a function of
$z=q_v^{-s}$.
Write $\pi_v$ as Langlands quotient $\pi_v=J_R^M(\delta_v,\mu)$ where $R$ is a
parabolic subgroup of $M$, $\delta_v$ a square integrable representation of
$M_R(F_v)$ and $\mu\in(\af^\ast_R/\af^\ast_M)_\C$ with $\Re(\mu)$ in the chamber
attached to $R$. Then by \cite[p. 30]{Ar4} we have
\[
 R_{\oP|P}(\pi_v,s)=R_{\oP(R)|P(R)}(\delta_v,s+\mu)
\]
with respect to the identifications described in \cite[p. 30]{Ar4}. Here $s$ is
identified with a point in $(\af_R^\ast/\af_G^\ast)_\C$ with respect to the 
canonical embedding $\af_M^\ast\subset\af_G^\ast$. Using again the factorization
of normalized intertwining operators we reduce the problem to the case of a 
square-integrable representation $\delta_v$. Moreover $\delta_v$ has to 
satisfy $[I_P^G(\delta_v,s)\big|_{K_v}\colon \one]\ge 1$. By \cite[Lemma~1]{Si}
we have 
\begin{equation}\label{multipl}
[I_P^G(\delta_v,s)\big|_{K_v}\colon \one]\ge 1\Leftrightarrow
[\delta_v\big|_{K_v\cap M(F_v)}\colon\one]\ge 1
\end{equation}
Let $\Pi_2(M(F_v))$ be the space of square-integrable representations of
$M(F_v)$. This space has a manifold structure \cite{HC1}, \cite{Si}.
By \cite[Theorem~10]{HC1} the set of square-integrable representations 
$\Pi_2(M(F_v),K_v)$ of $M(F_v)$ with $[\delta_v\big|_{K_v\cap M(F_v)}
\colon\one]\ge 1$ is a compact subset of $\Pi_2(M(F_v))$. Under the canonical 
action of $i\af_M$, the set $\Pi_2(M(F_v),K_v)$ decomposes into a finite number
of orbits. For $\mu\in i\af_M$ and $\delta_v\in\Pi_2(M(F_v),K_v)$, let 
$(\delta_v)_\mu\in \Pi_2(M(F_v),K_v)$ be the result of the canonical action.
Then it follows that
\[
R_{\oP|P}((\delta_v)_\mu,\lambda)=R_{\oP|P}(\delta_v,\lambda+\mu).
\]
In this way our problem is finally reduced to the consideration of
the matrix coefficients of $R_{\oP|P}(\pi_v,s)\big|_{K_v}$ for a finite number of
representations $\pi_v$. This implies that $n$ is bounded by a constant which is
independent of $\pi_v$. Together with \eqref{est-int2} it follows 
that for each finite place
$v$ of $F$ and each open compact subgroup $K_v$ of $G(F_v)$ there exists 
$C_v>0$ such that
\begin{equation}\label{loc-int-est1}
\int_{\R} \Big\|R_{\oP|P}(\pi_v,iu)^{-1}R^\prime_{\oP|P}(\pi_v,iu)\big|_{\H_P(\pi_v)^{K_v}}
\Big\|e^{-tu^2}du\le C_vt^{-1/2}
\end{equation}
for all $0<t\le 1$ and $\pi_v\in\Pi(M(F_v))$ with
$I_P^G(\pi_v)\big|_{\H_P(\pi_v)^{K_v}}\neq0$.

{\bf Case 2:} $v=\infty$. To begin with we need a modification of
\cite[Lemma 5.19]{FLM2}.
\begin{lem}\label{lem-inequ-A}
Let $z_j\in\C\setminus i\R$, $j=1,...,m$, and let $b(z)=(z-z_1)\cdots (z-z_m)$.
Suppose that $A\colon\C\setminus\{z_1,...,z_m\}\to V$ is such that $\|A(z)\|\le
1$ for all $z\in i\R$ and $b(z)A(z)$ is a polynomial in $z\in\C$ (necessarily of
degree $\le m$) with coefficients in $V$. Then
\[
\int_{\R}\|A^\prime(iu)\| e^{-t u^2} du\le 2\pi m 
\]
for all $0\le t$.
\end{lem}
\begin{proof}
For any $w\in\C$, let $\phi_w(z):=\frac{z+\bar{w}}{z-w}$, and set
\[
\phi_{>}(z)=\sum_{j:\Re(z_j)>0}\phi_{z_j}(z),\quad \phi_{<}(z)=
\sum_{j:\Re(z_j)<0}\phi_{z_j}(z).
\]
Applying \cite[Theorem 4]{BE}, it follows that
\begin{equation}\label{estim-int-A}
\|A^\prime(z)\|\le\max\{|\phi^\prime_{>}(z)|,|\phi^\prime_{<}(z)|\}\le
|\phi^\prime_{>}(z)|+|\phi^\prime_{<}(z)|,\quad z\in i\R.
\end{equation}
Now observe that for $w=x+iy\in\C\setminus i\R$, one has
$|\phi_w^\prime(z)|=\frac{2|x|}{|z-w|^2}=
\frac{2|x|}{x^2+(u-y)^2}$ for $z=iu$, $u\in \R$. So we get
\[
\begin{split}
\int_{\R}|\phi^\prime_w(iu)| e^{-tu^2} du&\le\int_\R\frac{2|x|}{x^2+(u-y)^2}
du=\frac{2}{|x|}\int_\R\frac{1}{1+(\frac{u}{|x|}-\frac{y}{|x|})^2} du\\
&=2\int_\R\frac{1}{1+(u-\frac{y}{|x|})^2} du
=2\int_\R\frac{du}{1+u^2}=2\pi.
\end{split}
\]
Together with \eqref{estim-int-A} the lemma follows.
\end{proof}

As above let $M\in\cL$ with $\dim(\af_M/\af_G)=1$
and $P\in\cP(M)$. Let $\pi_\infty\in\Pi(M(F_\infty))$ and $\nu\in\Pi(\K_\infty)$.
As explained in \cite[Appendix]{MS}, there exist $w_1,...,w_r\in\C$ and 
$m\in\N$ such that the poles of $R_{\oP|P}(\pi_\infty,s)\big|_{\H_P(\pi_\infty)^\nu}$
are contained in $\cup_{j=1}^r\{w_j-k\colon k=1,...,m\}$. Moreover, by
\cite[Proposition~A.2]{MS} there exists $c>0$ which depends only on $G$, such 
that
\begin{equation}\label{est-const}
r\le c,\quad m\le c(1+\|\nu\|).
\end{equation}
Let $A\colon \C\to \H_P(\pi_\infty)^\nu$ be defined by 
\[
A(z):=R_{\oP|P}(\pi_\infty,z)\big|_{\H_P(\pi_\infty)^\nu}
\]
and let $b(z)=\prod_{j=1}^r\prod_{k=1}^m(z-w_j+k)$. Then it follows from ($R_6$) of
\cite[Theorem~2.1]{Ar4} that  $b(z)A(z)$ is a polynomial function. Moreover,
by unitarity of $R_{\oP|P}(\pi_\infty,it)$, $t\in\R$, we have $\|A(it)\|=1$. 
Thus $A(z)$ satisfies the assumptions of \cite[Lemma~5.19]{FLM2}. Thus by
Lemma \ref{lem-inequ-A} and \eqref{est-const} we get
\begin{equation}\label{log-esti-archim}
\begin{split}
\int_\R\Big\|R_{\oP|P}(\pi_\infty,iu)^{-1}R_{\oP|P}^\prime(\pi_\infty,iu)
\big|_{\H_P(\pi_\infty)^\nu}&\Big\|e^{-tu^2} du=\int_\R\| A^\prime(iu)\|e^{-tu^2} du\\
&\le 2\pi r\cdot m\le 2\pi c^2(1+\|\nu\|).
\end{split}
\end{equation} 
Combining \eqref{log-der-inter1}, \eqref{loc-int-est1} and \eqref{log-esti-archim},
the proposition follows.
\begin{remark}
For $G=\GL_n$ it is proved in \cite[Proposition 0.2]{MS} that the corresponding
bounds hold for the derivatives of the local intertwining operators itself.
This follows from a weak version of the Ramanujan conjecture, which implies that
the poles of the local intertwining operators are uniformly bounded away from
the imaginary axis. For the integrated derivatives the distance of the poles
from the imaginary axis does not matter.
\end{remark}
\end{proof}

\section{The residual spectrum} 
\setcounter{equation}{0}

The goal of this section is to estimate the growth of the counting function of
the residual spectrum. To this end we recall the construction of the residual
spectrum. By Langlands
\cite[Ch. 7]{La1}, \cite[V.3.13]{MW}, $L^2_{\res}(G(F)\bs G(\A)^1)$ is spanned by
{\it iterated residues} of cuspidal Eisenstein series. Let us briefly recall
this construction.

Let $P=M\ltimes N$ be a $F$-rational parabolic subgroup of $G$. If $\alpha
\in\Sigma_P$, denote by $\alpha^\vee$ the co-root associated to $\alpha$.
Given $\alpha\in\Sigma_P$ and $c\in\R$, we set 
\[
H(\alpha,c):=\{\Lambda\in\af^\ast_\C\colon \Lambda(\alpha^\vee)=c\}.
\]
An affine subspace $\cH\subset \af^\ast_\C$ is called admissible, if $\cH$ is
the intersection of such hyperplanes. Suppose that $\cH_1\supset\cH_2$ are two
admissible affine subspaces of $\af^\ast_\C$ and $\cH_2$ is of co-dimension one
in $\cH_1$. Let $\Phi(\Lambda)$ be a meromorphic function on $\cH_1$ whose
singularities lie along hyperplanes which are admissible as subspaces of
$\af^\ast_\C$. Choose a real unit vector $\Lambda_0$ in $\cH_1$ which is normal
to $\cH_2$. Let $\delta>0$ be such that $\Phi(\Lambda+z\Lambda_0)$ has no
singularities in the punctured disc $0<|z|<2\delta$. 
Then we can define a meromorphic function $\Res_{\cH_2}\Phi$ on $\cH_2$
by
\[
\Res_{\cH_2}\Phi(\Lambda):=\frac{\delta}{2\pi i}\int_0^1
\Phi(\Lambda+\delta e^{2\pi i\vartheta}\Lambda_0)d(e^{2\pi i\vartheta}),
\]
The singularities of $\Res_{\cH_2}\Phi$ lie on the intersections with $\cH_2$ of
the singular hyperplanes of $\Phi$ different from $\cH_2$. Now consider a
complete flag
\[
\af^\ast_\C=\cH_p\supset\cH_{p-1}\supset\cdots\supset\cH_1\supset\cH_0=
\{\Lambda_0\}
\]
of affine admissible subspaces of $\af^\ast_\C$ and let $\Lambda_i\in\H_i$ be a
real unit vector which is normal to $\cH_{i-1}$, $i=1,\dots,p$. We call $\cF=
\{\cH_i,\Lambda_i\}$ an admissible flag. Let $\Phi$ be a meromorphic function on
$\af^\ast_\C$ whose singularities lie along admissible hyperplanes of
$\af^\ast_\C$. Then we define $\Phi_i$ inductively by
\[
\Phi_p=\Phi,\; \Phi_i=\Res_{\cH_i}\Phi_{i+1},\quad i=0,\dots,p-1.
\]
Set
\[
\Res_{\cF}\Phi:=\Phi_0.
\]
This is the iterated residue of $\Phi$ at $\Lambda_0$.

Now let $\cA^2_{\cu}(P)$ the subspace of functions $\phi\in\cA^2(P)$ such that
for almost all $x\in G(\A)$, the function $\phi_x(m):=\phi(mx)$ on
$M(F)\bs M(\A)^1$ lies in the space $L^2_{\cu}(M(F)\bs M(\A)^1)$. For
$\pi\in\Pi_{\cu}(M(\A)^1)$ let $\cA^2_{\cu,\pi}(P)$ be the subspace of functions
$\phi\in\cA^2_{\cu}(P)$ such that each of the functions $\phi_x$ lies in the
subspace $L^2_{\cu,\pi}(M(F)\bs M(A)^1)$ ({\color{red} isotypical subspace}). Let
 $\phi\in\cA^2_{\cu}(P)$. As shown by Langlands
\cite[\S 7]{La1}, the singularities of the Eisenstein series $E(\phi,\lambda)$
lie along hyperplanes of $\af^\ast_\C$ which are defined by equations of the form
$\Lambda(\alpha^\vee)=w$, $w\in\C$, $\alpha\in\Sigma_P$. Let $H(\alpha_i,c_i)$,
$i=0,\dots,p-1$, be a set of singular hyperplanes of $E(\phi,\lambda)$ with
$\cap_i H(\alpha_i,c_i)=\{\Lambda_0\}$. Set $\cH_i:=\cap_{j\ge i} H(\alpha_j,c_j)$,
$i=0,\dots,p-1$, and $\cH_P=\af^\ast_\C$. Choose real unit vectors $\Lambda_i\in
\cH_i$ normal to $\cH_{i-1}$. Then $\cF:=\{\cH_i,\Lambda_i\}$ is an admissible
flag. Furthermore, let $\varphi\in C_c^\infty(\af)$ and let
$\hat\varphi(\Lambda)$ be its Fourier transform. It is holomorphic on
$\af^\ast_\C$. Put
\[
\psi:=\Res_{\cF}\bigl[E(\phi,\Lambda)\hat\varphi(\Lambda)\bigr].
\]
Note that $\psi$ depends only on the derivatives of $\hat\varphi$ at
$\Lambda_0$. Let $\Co(\af^\ast)$ be the positive cone in $\af^\ast$ spanned by
the simple roots of $(P,A)$. If $\Lambda_0\in\Co(\af^\ast)$ then $\psi$ is
square integrable. Let $\Omega_M\in\cZ(\mf_\C)$ be the Casimir element of
$M(F_\infty)$. Assume that $\Omega_M\phi=\mu\phi$. Then it follows that
\begin{equation}\label{eigen-val}
\Omega\psi=(\|\Lambda_0\|^2-\|\rho_P\|^2+\mu)\psi.
\end{equation}
\cite[p. 29]{HC2}. Moreover $\|\Lambda_0\|\le\|\rho_P\|$. 
As shown by Langlands \cite[Theorem 7.2]{La1}, \cite[V.3.13]{MW},
$L^2_{\res}(G(F)\bs G(\A)^1)$ is spanned by all such $\psi$, where $P$ runs over
the standard parabolic subgroups of $G$, $\pi$ over $\Pi_{\cu}(M_P(\A)^1)$ and
$\phi$ over a basis of $\cA^2_{\cu,\pi}(P)$. For all details concerning the
description of the discrete and
residual spectrum see \cite[Theorem V.3.13, p. 221]{MW}, and
\cite[Corollary, p. 224]{MW}. Moreover, the question of positivity is dealt
with by \cite[Corollary VI.1.6 (d), p. 255]{MW}.

Let $K_f$ be an open compact subgroup of $G(\A_f)$.
Let $L^2_{\res}(\AG G(F)\bs G(\A))^{K_f}$ be the subspace of $K_f$-invariant
functions of $L^2_{\res}(\AG G(F)\bs G(\A))$. Moreover, for
$\nu\in\Pi(\K_\infty)$ let
$L^2_{\res}(\AG G(F)\bs G(\A))^{K_f,\nu}$ be the $\nu$-isotypical
subspace of $L^2_{\res}(\AG G(F)\bs G(\A))^{K_f}$. Then
$L^2_{\res}(\AG G(F)\bs G(\A))^{K_f,\nu}$ is spanned by residues as above,
where for a given pair $(P,\pi)$, $\phi$ runs over a basis of
$\cA^2_{\cu,\pi}(P)^{K_f,\sigma}$. Recall that $\cA^2_{\cu,\pi}(P)^{K_f,\sigma}$ is finite
dimensional. So the estimation of the counting function of the residual
spectrum is reduced to the following problems:
\begin{enumerate}
\item[(1)] Estimation of $\dim\cA^2_{\cu,\pi}(P)^{K_f,\sigma}$ in terms of $\pi$,
$K_f$, and $\sigma$.
\item[(2)] For a given cuspidal Eisenstein
series $E(\phi,\Lambda)$, $\phi\in\cA^2_{\cu,\pi}(P)$, we need to estimate the
number of its singular hyperplanes, counted to multiplicity, which are real and
intersect a given compact set containing the origin.
\end{enumerate}
We start with (1). Let $\pi\in\Pi_{\cu}(M(\A))$, $\pi=\pi_\infty\otimes\pi_f$.
Let $\cH_P(\pi_\infty)$ (resp. $\cH_P(\pi_f)$) be the Hilbert space of the induced
representation $\Ind_{P(F_\infty)}^{G(F_\infty)}(\pi_\infty)$ (resp. $\Ind_{P(\A_f)}^{G(\A_f)}$). 
Denote by $\cH_P(\pi_\infty)_\sigma$ the $\sigma$-isotypical subspace of
$\cH_P(\pi_\infty)$ and by $\cH_P(\pi_f)^{K_f}$ the subspace of $K_f$-invariant
vectors of $\cH_P(\pi_f)$.  Let $m(\pi)$ denote the multiplicity with which
$\pi\in\Pi_{\cu}(M(\A))$ occurs in $L^2_{\cu}(\AG M(F)\bs M(\A))$. Then
by \eqref{ind-rep-1}
we obtain
\begin{equation}\label{dim-A2}
\dim\cA^2_\pi(P)^{K_f,\sigma}=m(\pi)\dim(\cH_P(\pi_f)^{K_f})
\dim(\cH_P(\pi_\infty)_\sigma).
\end{equation}
Using Frobenius reciprocity \cite[p. 208]{Kn} we get
\[
[\Ind_{P(F_\infty)}^{G(F_\infty)}(\pi_\infty)|_{\K_\infty}\colon\sigma]=
\sum_{\tau\in\Pi(\K_{M,\infty})}
[\pi_\infty|_{\K_{M,\infty}}\colon\tau]\cdot[\sigma|_{\K_{M,\infty}}\colon\tau].
\]
For $\tau\in\Pi(\K_{M,\infty})$ let $\cH_{\pi_\infty}(\tau)$ denote the
$\tau$-isotypical subspace of $\cH_{\pi_\infty}$. Then we obtain
\begin{equation}\label{dim-isotyp}
\dim(\cH_P(\pi_\infty)_\sigma)\le \dim(\sigma)\sum_{\tau\in\Pi(\K_{M,\infty})}
\dim(\cH_{\pi_\infty}(\tau))\cdot[\sigma|_{\K_{M,\infty}}\colon\tau].
\end{equation}
Next we consider $\pi_f=\otimes_{v<\infty}\pi_v$. Replacing $K_f$ by a subgroup
of finite index, if necessary, we can assume that $K_f=\prod_{v<\infty}K_v$. For
any $v<\infty$, denote by $\cH_P(\pi_v)$ the Hilbert space of the induced
representation $\Ind_{P(F_v)}^{G(F_v)}(\pi_v)$. Let $\cH_P(\pi_v)^{K_v}$ be the
subspace of $K_v$-invariant vectors. Then $\dim\cH_P(\pi_v)^{K_v}=1$ for almost
all $v$ and
\[
\cH_P(\pi_f)^{K_f}\cong \bigotimes_{v<\infty}\cH_P(\pi_v)^{K_v}.
\]
\begin{equation}\label{K-finite-inv}
\begin{split}
\Ind^{G(F_v)}_{P(F_v)}(\pi_v)^{K_v}&=\left(\Ind_{P(\cO_v)}^{G(\cO_v)}(\pi_v)\right)^{K_v}
\\
&\hookrightarrow \bigoplus_{G(\cO_v)/K_v}\Ind^{K_v}_{K_v\cap P}(\pi_v)^{K_v}\\
&\cong\bigoplus_{G(\cO_v)/K_v}\pi_v^{K_v\cap P}.
\end{split}
\end{equation}
Let $K_{M,f}=K_f\cap M(\A_f)$. For $\sigma\in\Pi(K_\infty)$ let
\[
\cF_M(\sigma):=\{\tau\in\Pi(K_{M,\infty})\colon [\sigma|_{K_{M,\infty}}\colon\tau]>0\}.
\]
Combining \eqref{dim-A2} - \eqref{K-finite-inv}, it follows that there exists
$C>0$, which depends on $\sigma$, such that
\begin{equation}\label{est-dim-A2}
\dim\cA^2_\pi(P)^{K_f,\sigma}\le C m(\pi)\dim(\cH_{\pi_f}^{K_{M,f}})
\sum_{\tau\in\cF(\sigma)}\dim(\cH_{\pi_\infty}(\tau)).
\end{equation}
In order to deal with (2), we use the inner product formula for truncated
cuspidal Eisenstein series proved by Langlands \cite[Sect. 9]{La3},
\cite[Lemma 4.2]{Ar2}. We recall the formula. Let $T\in\af_0^+$ be sufficiently regular. For $\phi\in\cA^2_{\cu}(P)$ let $\Lambda^TE(g,\phi,\lambda)$ be the
truncated Eisenstein series \cite[Sect. 1]{Ar2}. Let $\phi\in\cA_{\cu}^2(P)$
and $\phi^\prime\in\cA^2_{\cu}(P^\prime)$. Then we have the following inner
product formula 
\begin{equation}\label{inner-prod}
\begin{split}
&\int_{G(F)\bs G(\A)^1}\Lambda^TE(g,\phi,\lambda)\overline{\Lambda^T E(g,\phi^\prime,
\lambda^\prime)}dg\\
&=\sum_Q\sum_s\sum_{s^\prime}\vol(\af^G_Q/\Z(\Delta^\vee_Q))
\frac{e^{(s\lambda+s^\prime\lambda^\prime)(T)}}
{\prod_{\alpha\in\Delta_Q}(s\lambda+s^\prime\bar{\lambda^\prime})(\alpha^\vee)}
(M_{Q|P}(s,\lambda)\phi,M_{Q|P^\prime}(s^\prime,\lambda^\prime)\phi^\prime),
\end{split}
\end{equation}
where $Q$ runs over all standard parabolic subgroups, $s\in W(\af_P,\af_Q)$,
and $s^\prime\in W(\af_{P^\prime},\af_Q)$, as meromorphic functions of
$\lambda\in\af^\ast_{P,\C}$ and $\lambda^\prime\in\af^\ast_{P^\prime,\C}$
\cite[Prop. 15.3]{Ar5}, \cite[Lemma 4.2]{Ar2}. It follows from the inner product
formula that in order to
settle (2), it suffices to  estimate the corresponding number of singular
hyperplanes of the intertwining operators $M_{Q|P}(s,\lambda)|_{\cA^2_{\cu,\pi}(P)}$
for $\pi\in\Pi_{\cu}(M(\A))$. 
To deal with this problem, we reduce it to the case of
$M_{Q|P}(\lambda)|_{\cA^2_{\cu,\pi}(P)}$, $Q,P\in\cP(M)$,
$\pi\in\Pi_{\cu}(M(\A))$. Let $M,M_1\in\cL$ and let $P\in\cP(M)$,
$P_1\in\cP(M_1)$. Suppose that $P$ and $P_1$ are associated and let
$t\in W(\af_M,\af_{M_1})$. Let $w_t\in G(F)$ be a representative of $t$. Then
$M_1:=w_t M w_t^{-1}$ and $tP=w_t P w_t^{-1}$ is a parabolic subgroup which
belongs to $\cP(M_1)$. The restriction of $t$ to $\af_M\subset\af_0$ defines
an element in $W(\af_M,\af_{tM})$. Let
\[
t\colon\cA^2(P)\to\cA^2(tP)
\]
be the linear operator defined by $(t\phi)(x)=\phi(w_t^{-1}x)$, $x\in G(\A)$.
By \cite[Lemma 1.1]{Ar3} there exists $T_0\in\af_0$ such that
\[
H_{P_0}(w_t^{-1}) =T_0-t^{-1}T_0.
\]
Then by \cite[(1.5)]{Ar9} one has
\begin{equation}\label{aux-res}
M_{P_1|P}(s,\lambda)t^{-1}=M_{P_1|tP}(st^{-1},t\lambda)e^{(\lambda+\rho_P)(T_0-t^{-1}
T_0)}
\end{equation}
for $s\in W(\af_M,\af_{M_1})$. Thus as far as the singular hyperplanes of
$M_{P_1|P}(s,\lambda)$ are concerned, we can assume that $P_1$ and $P$ have the
same Levi component $M$, that is $P,P_1\in\cP(M)$. Let $t\in W(\af_M)$. By the
functional equation \cite[(1.2)]{Ar9} we have
\begin{equation}\label{funct-equ1}
M_{P_1|P}(t,\lambda)=M_{P_1|tP}(1,t\lambda)M_{tP|P}(t,\lambda).
\end{equation}
Using \eqref{aux-res} with $s=t$, we get
\[
M_{tP|P}(t,\lambda)=M_{tP|tP}(1,t\lambda)te^{(\lambda+\rho_P)(T_0-t^{-1}T_0)}.
\]
Since $M_{tP|tP}(1,\lambda)=\Id$, it follows that
\begin{equation}
M_{P_1|P}(t,\lambda)=M_{P_1|tP}(1,t\lambda)te^{(\lambda+\rho_P)(T_0-t^{-1}T_0)}.
\end{equation}
Thus we are reduced to
the consideration of the singular hyperplanes of the intertwining operators
$M_{Q|P}(\lambda)=M_{Q|P}(1,\lambda)$. Given $\pi\in\Pi_{\cu}(M_P(\A))$, an open
compact subgroup $K_f\subset G(\A_f)$ and $\sigma\in\Pi(\K_\infty)$,
we need to estimate the singular hyperplanes, counted to multiplicity,
of $M_{Q|P}(\pi,\lambda)^{K_f,\sigma}$, which are real and intersect a
fixed compact set. By \eqref{normalization1} the problem is reduces to the
consideration of the normalizing factors $n_{Q|P}(\pi,\lambda)$ and the
normalized intertwining operators $R_{Q|P}(\pi,\lambda)$ restricted to
$\cA^2_{\cu,\pi}(P)^{K_f,\sigma}$.

To begin with we consider the normalized intertwining operators
\eqref{norm-intertw-op2}. Let $v$ be a place of $F$. For $\pi_v\in\Pi(G(F_v))$
let
\[
J_{Q|P}(\pi_v,\lambda)\colon\cH_P^0(\pi_v)\to\cH_Q^0(\pi_v),\quad\lambda\in
\af_{Q,\C}^\ast,
\]
be the local intertwining operator and 
\begin{equation}\label{lo-int-op}
R_{Q|P}(\pi_v,\lambda):=n_{Q|P}(\pi_v,\lambda)^{-1} J_{Q|P}(\pi_v,\lambda).
\end{equation}
the local normalized intertwining operator. The operators
$R_{Q|P}(\pi_v,\lambda)$  satisfy properties
$(R_1),...,(R_8)$ of \cite[Theorem 2.1]{Ar4}. Let
$\pi=\otimes_v\pi_v\in\Pi_{\di}(G(\A))$. There exists a finite set of
places $S(\pi)$ of $F$, containing the Archimedean places, such that for all
$v\not\in S$, $G/F_v$ and $\pi_v$ are unramified. For $v\not\in S(\pi)$, let
$\K_v$ be hyperspecial and assume that $\phi_v\in\cH_P(\pi_v)^{\K_v}$. Then by
$(R_8)$ we have 
\[
R_{Q|P}(\pi_v,\lambda)\phi_v=\phi_v,\quad s\not\in S(\pi).
\]
Hence the product \eqref{norm-intertw-op2} runs only over $v\in S(\pi)$ and
therefore, it is well defined for all $\lambda\in\af^\ast_{M,\C}$. So it suffices
to consider the local intertwining operators $R_{Q|P}(\pi_v,\lambda)$.
Let $P_0,...,P_k\in\cP(M)$ and $\alpha_1,...,\alpha_k\in\Sigma_M$  such that
$P=P_0$, $Q=P_k$ and $P_{i-1}\big|^{\alpha_i}P_i$ for $i=1,...,k$. By
\cite[($R_2$)]{Ar4} we get
\begin{equation}
R_{Q|P}(\pi_v,\lambda)=R_{P_k|P_{k-1}}(\pi_v,\lambda)\circ
R_{R_{k-1}|P_{k-2}}(\pi_v,\lambda)\circ\cdots\circ R_{P_1|P_0}(\pi_v,\lambda).
\end{equation}
Hence we can assume that $Q,P\in\cP(M)$ are adjacent along some root
$\alpha\in\Sigma_M$. Then $R_{Q|P}(\pi_v,\lambda)$ depends only on
$\lambda(\alpha^\vee)$. First consider the case $v<\infty$. Let $q_v$ be the
order of the residue field of $F_v$. By $(R_6)$, $R_{Q|P}(\pi_v,\lambda)$ is a
rational function of $q_v^{-\lambda(\alpha^\vee)}$. Furthermore,
by  \cite[Theorem 2.2.2]{Sh1} there exists a polynomial $Q_v(\pi_v,s)$ with $
Q_v(\pi_v,0)=1$,
such that
\[
Q_v(\pi_v,q_v^{-\lambda(\alpha^\vee)})J_{Q|P}(\pi_v,\lambda)
\]
is a holomorphic and non-zero function of $\lambda\in\af_{M,\C}^\ast$. Moreover,
the degree of the polynomial $Q_v$ is independent of  $\pi_v\in\Pi(M(F_v))$
By \eqref{lo-int-op} it follows that
\[
n_{Q|P}(\pi_v,\lambda)Q_v(\pi_v,q_v^{-\lambda(\alpha^\vee)})R_{Q|P}(\pi_v,\lambda)
\]
is holomorphic on $\af_{M,\C}^\ast$. The normalizing factors
$n_{Q|P}(\pi_v,\lambda)$ satisfy properties similar to the corresponding
properties $(R_1),...,(R_8)$ satisfied by the local intertwining
operators \cite[Theorem 2.1]{Ar4}. In particular, there exists a rational
function
$n_\alpha(\pi_v,s)$ such that
\[
n_{Q|P}(\pi_v,\lambda)=n_\alpha(\pi_v,q_v^{-\lambda(\alpha^\vee)}).
\]
Now observe that for every $R>0$ and $z\in\C$ the number of solutions of
$q_v^{-s}=z$ in the disc $|s|\le R$ is bounded by $1+(2\pi)^{-1}\log(q_v)R$.
Hence it suffices to estimate the number of zeros of $n_\alpha(\pi_v,s)$
and $Q_v(\pi_v,s)$, respectively. As mentioned above, the degree of
$Q_v(\pi,s)$ is bounded independently of $\pi\in\Pi(M(F_v))$.
The rational function 
$n_\alpha(\pi_v,s)$ has been described in \cite[(3.6)]{Mu2}. It follows from
\cite[Lemma 3.1]{Mu2} and \cite[(3.6)]{Mu2} that there exists $C>0$ 
such that for all $M\in\cL(M_0)$ and all square integrable $\pi\in\Pi(M(F_v))$
the number of poles and zeros of $n_\alpha(\pi,s)$ is less than $C$. Now let
$\pi$ be tempered. It is known that $\pi$ is an irreducible constituent of an
induced representation $\Ind_R^M(\sigma)$, where $M_R$ is an admissible Levi
subgroup of $M$ and $\sigma\in\Pi(M_R(F_v))$ is square integrable modulo $A_R$.
Then by \cite[(2.2)]{Ar4} we are reduced to the square integrable case. In
general, $\pi$ is a Langlands quotient of an induced representation
$\Ind_R^M(\sigma,\mu)$, where $M_R$ is an admissible Levi subgroup of $M$,
$\sigma\in\Pi_{\temp}(M_R(F_v))$, and $\mu$ is point in the chamber of
$\af_R^\ast/\af_M^\ast$. Now we use \cite[(2.3)]{Ar4} to reduce to the tempered
case, which proves that there exists $C>0$ such that the number of poles and
zeros of $n_\alpha(\pi,s)$ is less than $C$ for all $\pi\in\Pi(M(F_v))$.

The case $v\in S_\infty$ has been already treated in section \ref{sec-locint}.
See \eqref{est-const} and the text above \eqref{est-const}.

Now we can summarize our results.
Using \eqref{norm-factors}, we obtain the following proposition.
\begin{prop}\label{prop-norm-intertw}
Let $M\in\cL(M_0)$. Let $K_f\subset G(\A_f)$ be an open compact subgroup and
$\sigma\in\Pi(\K_\infty)$. For every $R>0$ there exists $C>0$
such that for all $Q,P\in\cP(M)$ and $\pi\in\Pi(M(\A))$ the number of
singular hyperplanes of $R_{Q|P}(\pi,\lambda)\big|_{\cA^2_\pi(P)^{K_f,\sigma}}$, which
intersects the ball of radius $R$ in $\af_{M,\C}^\ast$, is bounded by $C$.
\end{prop}

Next we consider the global normalizing factors. By \eqref{norm-factors},
$n_{Q|P}(\pi,\lambda)$ is the
product of the normalizing factors $n_\alpha(\pi,\lambda(\alpha^\vee))$,
$\alpha\in\Sigma_P\cap\Sigma_{\bar{Q}}$. Thus our problem is reduced to the
estimation of the number of real poles, counted to multiplicity, of the
meromorphic function $n_\alpha(\pi,s)$. Let $\Sigma_\alpha^\R(\pi)$ be
the set of real poles of $n_\alpha(\pi,s)$. Let
$\pi\in\Pi_{\di}(M(\A);K_f,\sigma)$. Then it follows from Corollary
\ref{cor-est-poles} that there exists $C>0$ 
\begin{equation}\label{num-poles}
\sum_{\rho\in\Sigma_\alpha^\R(\pi)}n(\rho)\le C(1+\log(1+\lambda^2_{\pi_\infty})).
\end{equation}
Now we can summarize our results as follows. For arbitrary $Q,P\in\cP(M)$,
the global normalizing factor $n_{Q|P}(\pi,\lambda)$ is the product of 
the functions $n_\alpha(\pi,\lambda(\alpha^\vee))$,
$\alpha\in\Sigma_P\cap\Sigma_{\bar Q}$  \eqref{norm-factors}. Note that
$\#\Sigma_P\le\dim \nf_P$. Let $d_P:=\dim\nf_P$. Then it follows from 
\eqref{num-poles} and Proposition \eqref{prop-norm-intertw} that there exists
$C>0$ such for every $\pi\in\Pi_{\cu}(M(\A);K_f,\sigma)$ and every
$\phi\in\cA_{\cu,\pi}^2(P)$ the number of singular hyperplanes of the
Eisenstein series $E(\phi,\lambda)$, counted with multiplicity, which are
real and intersect the ball $\{\lambda\in\af_{M,\C}^\ast\colon\|\lambda\|\le \|\rho_P\|\}$  is bounded by
\begin{equation}
C(1+\log(1+\lambda^2_{\pi_\infty}))^{d_P},
\end{equation}
Now recall from the beginning of this section that the residual spectrum is
spanned
by iterated residues of Eisenstein series $E(\phi,\lambda)$ with respect to
complete flags of affine admissible subspaces of $\af_{M,\C}^\ast$. For
$\lambda\ge 0$ let
\[
\Pi_{\cu}(M(\A);\lambda):=\{\pi\in\Pi_{\cu}(M(\A))\colon -\lambda_{\pi_\infty}\le
\lambda\}.
\]
We note that there exists $C\in\R$ such that $C\le -\lambda_{\pi_\infty}$ for all
$\pi\in\Pi_{\di}(M(\A))$. 
Using \eqref{eigen-val} and  \eqref{est-dim-A2}, it follows that there exist
$C_0,C_1>0$ such that
\begin{equation}\label{estim5}
\begin{split}
N_{\res}^{K_f,\sigma}(\lambda)\le C_1
\sum_{P\supset P_0}&(1+\log(1+\lambda^2))^{d_{P}}\\
&\cdot\sum_{\tau\in\cF_{M_P}(\sigma)}\sum_{\pi\in\Pi_{\cu}(M_P(\A);\lambda+C_0)}
m(\pi)\dim(\cH_{\pi_f}^{K_{M_P,f}})\dim(\cH_{\pi_\infty}(\tau)).
\end{split}
\end{equation}
For a given $P$ let $M=M_P$. As in \eqref{adel-quot1} there exist finitely
many lattices
$\Gamma_{M,i}\subset M(F_\infty)$, $i=1,...,k$, such that
\begin{equation}
\Ai M(F)\bs M(\A)/K_{M,f}\cong
\bigsqcup_{i=1}^k(\Gamma_{M,i}\bs M(F_\infty)^1). 
\end{equation}
Thus we get an isomorphism of $M(F_\infty)$-modules
\begin{equation}
L^2_{\cu}(\Ai M(F)\bs M(\A))^{K_{M,f}}\cong \bigoplus_{i=1}^k
L^2_{\cu}(\Gamma_{M,i}\bs M(F_\infty)^1).
\end{equation}
And hence for $\tau\in\Pi(K_{M,\infty})$ we get
\begin{equation}\label{compar3}
L^2_{\cu}(\Ai M(F)\bs M(\A))^{K_{M,f},\tau}\cong \bigoplus_{i=1}^k
L^2_{\cu}(\Gamma_{M,i}\bs (M(F_\infty)^1\otimes V_\tau)^{K_{M,\infty}}.
\end{equation}
Let $\widetilde X_M=\Ai\bs M(F_\infty)/K_{M,\infty}$. Let
$E_{\tau,i}\to\Gamma_{M,i}\bs \widetilde X_M$ be the locally homogeneous vector
bundle associated to $\tau$. Let $N_{\cu}^{\Gamma_{M,i}}(\lambda;\tau)$
be the eigenvalues counting function for the Casimir operator acting
in $L^2_{\cu}(\Gamma_{M,i}\bs \widetilde X_M, E_{\tau,i})$. 
It follows from \eqref{compar3} that
\begin{equation}\label{count-fct4}
\sum_{\pi\in\Pi_{\cu}(M(\A);\lambda)} m(\pi)\dim(\cH^{K_{M,f}}_{\pi_f})
\dim(\cH_{\pi_\infty}(\tau))=\sum_{i=1}^k N_{\cu}^{\Gamma_{M,i}}(\lambda;\tau).
\end{equation}
Let $m_M:=\dim\widetilde X_M$. Then by \cite[Theorem 9.1]{Do} we get
\begin{equation}
N_{\cu}^{\Gamma_{M,i}}(\lambda;\tau)\ll 1+\lambda^{m_M/2}. 
\end{equation}
Thus by \eqref{count-fct4} it follows that there exists $C>0$ such that
\begin{equation}\label{estim-res4a}
\sum_{\tau\in\cF_{M}(\sigma)}\sum_{\pi\in\Pi_{\cu}(M(\A);\lambda+C_0)} m(\pi)
\dim(\cH^{K_{M,f}}_{\pi_f})\dim(\cH_{\pi_\infty}(\tau))\le C(1+\lambda^{m_M/2}).
\end{equation}
Let $n:=\dim \widetilde X$ and $m_0(G):=\max\{m_{M_P}\colon P\neq G\}$. Note
that $d_P\le r$ for all $P$.
Then by \eqref{estim5} and \eqref{estim-res4a} we obtain
\begin{equation}\label{estim-res5}
N_{\res}^{K_f,\sigma}(\lambda)
\le C(1+\log(1+\lambda^2))^n(1+\lambda^{m_0(G)/2}).
\end{equation}
Now observe that $\widetilde X\cong\widetilde X_M\times\Ai\times N_P(F_\infty)$
\cite[Sect. 4.2, (3)]{Bo3}.
Hence $m_M\le n-2$, and therefore $m_0(G)\le n-2$. Thus we get
\begin{equation}\label{estim-res6}
N_{\res}^{K_f,\sigma}(\lambda)\le C(1+\lambda^{(n-1)/2}),\quad \lambda\ge 0,
\end{equation}
where $C>0$ depends on $K_f$ and $\sigma$. This completes the proof of the
second statement of Theorem \ref{main-thm-adelic}. 

Next we wish to extend this result to any Levi subgroup $L\in\cL(M)$.
Recall that for any pair of elements $Q\in\cP(L)$  and $R\in\cP^L(M)$
there exists a unique $P\in\cP(M)$ such that $P\subset Q$ and
$P\cap L=R$. Then $P$ is denoted by $Q(R)$. Let $R,R^\prime\in\cP^L(M)$,
$\pi\in\Pi_{\di}(M(\A))$, and $Q\in\cP(L)$. Then for any $k\in \K$ and
$\phi\in\cA^2_\pi(Q(R))$,
the function $\phi_k$ on $M(\A)$, which is defined by
$\phi_k(m):=\phi(mk)$, $m\in M(\A)$, belongs to $\cA^2_\pi(R)$, and one
has
\begin{equation}\label{intertw-rel1}
(M_{Q(R^\prime)|Q(R)}(\pi,\lambda)\phi)_k=M_{R^\prime|R}(\pi,\lambda)\phi_k)
\end{equation}
\cite[(1.3)]{Ar9}. Furthermore, the normalizing factors satisfy
\begin{equation}\label{rel-norm-fact}
n_{Q(R^\prime)|Q(R)}(\pi,\lambda)=n_{R^\prime|R}(\pi,\lambda)
\end{equation}
\cite[Sect. 2]{Ar4}. 
Thus the considerations above continue to hold for each standard Levi
subgroup $L$ of $G$. For $M\in\cL(M_0)$ and $\sigma\in\Pi(\K_\infty)$ let
$\sigma_M:=\sigma|_{\K_\infty\cap M(F_\infty)}$. Denote by
$N_{M,\res}^{K_{M,f},\sigma_M}(\lambda)$ the counting function of the
residual spectrum for $M$ with respect to $(K_{M,f},\sigma_M)$. Let
$m_0(M):=\max\{m_{M_R}\colon R\in\cP^M,\; R\neq M\}$. 
Then, summarizing our results, we get
\begin{equation}
N_{M,\res}^{K_{M,f},\sigma_M}(\lambda)\le C(1+\log(1+\lambda^2))^{r_M}
\lambda^{m_0(M)/2}
\end{equation}
for $\lambda\ge 0$. As above it follows that $m_0(M)\le \dim\widetilde X_M-2$,
and we obtain the following proposition.
\begin{prop}\label{prop-res}
Assume that $G$ satisfies condition (L). Let $K_f$ be an open compact
subgroup of $G(\A_f)$ and $\sigma\in\Pi(\K_\infty)$. Let $M\in\cL$ and
$m_M:=\dim \widetilde X_M$. Then there exists $C>0$ such that
\begin{equation}
N_{M,\res}^{K_{M,f},\sigma_M}(\lambda)\le C(1+\lambda^{(m_M-1)/2})
\end{equation}
for $\lambda\ge 0$.
\end{prop}

\section{The spectral side of the trace formula}
\setcounter{equation}{0}
In this section we apply the spectral side of the (non-invariant) trace
formula of Arthur \cite{Ar1}, \cite{Ar2}, to the heat kernel. The goal
is to prove that the leading term of the asymptotic expansion as $t\to 0$ is
given by the trace of the heat operator, restricted to the point spectrum. 

To begin with we briefly recall the structure of the spectral side.
Let $L \supset M$ be Levi subgroups in 
$\levis$, $P \in \PPP (M)$, and
let $m=\dim\aaa_L^G$ be the co-rank of $L$ in $G$.
Denote by $\bases_{P,L}$ the set of $m$-tuples $\bss=(\beta_1^\vee,\dots,
\beta_m^\vee)$
of elements of $\rts_P^\vee$ whose projections to $\af_L$ form a basis for 
$\af_L^G$.
For any $\bss=(\beta_1^\vee,\dots,\beta_m^\vee)\in\bases_{P,L}$ let
$\vol(\bss)$ be the co-volume in $\af_L^G$ of the lattice spanned by $\bss$ 
and let
\begin{align*}
\Xi_L(\bss)&=\{(Q_1,\dots,Q_m)\in\cF_1(M)^m: \ \ \beta_i^{\vee}\in\af_M^{Q_i}, 
\, i = 1, \dots, m\}\\&=
\{\langle P_1,P_1'\rangle,\dots,\langle P_m,P_m'\rangle): \ \ 
P_i|^{\beta_i}P_i', \, 
i = 1, \dots, m\}.
\end{align*}
Given $Q,P\in\cP(M)$, let $M_{Q|P}(\lambda)\colon \cA^2(P)\to\cA^2(Q)$,
$\lambda\in\af^\ast_{M,\C}$, be the intertwining operator defined by
\eqref{intertw1}.

For any smooth function $f$ on $\af_M^*$ and
$\mu\in\af_M^*$ denote by 
$D_\mu f$ the directional derivative of $f$ along $\mu\in\af_M^*$.
For a pair $P_1|^\alpha P_2$ of adjacent parabolic subgroups in $\cP(M)$ write
\begin{equation}\label{intertw2}
\delta_{P_1|P_2}(\lambda)=M_{P_2|P_1}(\lambda)D_\varpi M_{P_1|P_2}(\lambda):
\cA^2(P_1)\rightarrow\cA^2(P_2),
\end{equation}
where $\varpi\in\af_M^*$ is such that $\sprod{\varpi}{\alpha^\vee}=1$.
\footnote{Note that this definition differs slightly from the definition of
$\delta_{P_1|P_2}$ in \cite{FLM1}.} Equivalently, writing
$M_{P_1|P_2}(\lambda)=\Phi(\sprod{\lambda}{\alpha^\vee})$ for a
meromorphic function $\Phi$ of a single complex variable, we have
\[
\delta_{P_1|P_2}(\lambda)=\Phi(\sprod{\lambda}{\alpha^\vee})^{-1}
\Phi'(\sprod{\lambda}{\alpha^\vee}).
\]
Recall that for $P,Q\in\cP(M)$, $\langle P,Q\rangle$ denotes the group
generated by $P$ and $Q$.
For any $m$-tuple $\dtup=(Q_1,\dots,Q_m)\in\Xi_L(\bss)$
with $Q_i=\langle P_i,P_i'\rangle$, $P_i|^{\beta_i}P_i'$, denote by 
$\Delta_{\dtup}(P,\lambda)$
the expression
\begin{equation}\label{intertw3}
\frac{\vol(\bss)}{m!}M_{P_1'|P}(\lambda)^{-1}\delta_{P_1|P_1'}(\lambda)M_{P_1'|P_2'}(\lambda) \cdots
\delta_{P_{m-1}|P_{m-1}'}(\lambda)M_{P_{m-1}'|P_m'}(\lambda)\delta_{P_m|P_m'}(\lambda)M_{P_m'|P}(\lambda).
\end{equation}
Recall the (purely combinatorial) map $\dtup_L: \bases_{P,L} \to \FFF_1 (M)^m$ with the property that
$\dtup_L(\bss) \in \Xi_L (\bss)$ for all $\bss \in \bases_{P,L}$ as defined in \cite[pp.\ 179--180]{FLM1}.\footnote{The map $\dtup_L$ depends in fact on the
additional choice of
a vector $\underline{\mu} \in (\aaa^*_M)^m$ which does not lie in an explicit finite
set of hyperplanes. For our purposes, the precise definition of $\dtup_L$ is immaterial.}

For any $s\in W(M)$ let $L_s$ be the smallest Levi subgroup in $\levis(M)$
containing $w_s$. We recall that $\aaa_{L_s}=\{H\in\aaa_M\mid sH=H\}$.
Set
\[
\iota_s=\abs{\det(s-1)_{\aaa^{L_s}_M}}^{-1}.
\]
For $P\in\FFF(M_0)$ and $s\in W(M_P)$ let
$M(P,s):\AF^2(P)\to\AF^2(P)$ be as in \cite[p.~1309]{Ar3}.
$M(P,s)$ is a unitary operator which commutes with the operators $\rho(P,\lambda,h)$ for $\lambda\in\iii\aaa_{L_s}^*$.
Finally, we can state the refined spectral expansion.

\begin{theo}[\cite{FLM1}] \label{thm-specexpand}
For any $h\in C_c^\infty(G(\A)^1)$ the spectral side of Arthur's trace formula is given by
\begin{equation}\label{specside1}
J_{\spec}(h) = \sum_{[M]} J_{\spec,M} (h),
\end{equation}
$M$ ranging over the conjugacy classes of Levi subgroups of $G$ (represented by members of $\mathcal{L}$),
where
\begin{equation}\label{specside2}
J_{\spec,M} (h) =
\frac1{\card{W(M)}}\sum_{s\in W(M)}\iota_s
\sum_{\bss\in\bases_{P,L_s}}\int_{\iii(\aaa^G_{L_s})^*}
\tr(\Delta_{\dtup_{L_s}(\bss)}(P,\lambda)M(P,s)\rho(P,\lambda,h))\ d\lambda
\end{equation}
with $P \in \PPP(M)$ arbitrary.
The operators are of trace class and the integrals are absolutely convergent
with respect to the trace norm and define distributions on $\Co(G(\A)^1)$.
\end{theo}
Now we apply the trace formula to the heat kernel. We recall its definition.
For details see \cite[$\S$ 3]{MM1}. Recall that the underlying symmetric space
is 
\[
\widetilde X=G(F_\infty)^1/\K_\infty,
\]
where $G(F_\infty)^1=G(\A)^1\cap G(F_\infty)$. Note that $G(F_\infty)^1$ is
semisimple and
\begin{equation}
G(F_\infty)=G(F_\infty)^1\cdot \Ag.
\end{equation}
Given $\nu\in\Pi(\K_\infty)$, let $\widetilde E_\nu\to \widetilde X$ be the
associated homogeneous vector bundle. Let $\widetilde\Delta_\nu=
(\widetilde\nabla^\nu)^\ast\widetilde\nabla^\nu$ be the Bochner-Laplace operator
acting in the space $C^\infty(\widetilde X,\widetilde E_\nu)$ of smooth sections
of $\widetilde E_\nu$. This is a $G(F_\infty)^1$-invariant second order elliptic
differential operator. Since 
$\widetilde{X}$ is complete, $\widetilde{\Delta}_{\nu}$, regarded as operator in
$L^2(\widetilde X,\widetilde E_\nu)$  with domain the smooth compactly
supported sections, is essentially self-adjoint \cite[p. 155]{LaM}. Its
self-adjoint extension
will also be denoted by $\widetilde{\Delta}_\nu$. Let $\Omega\in\cZ(\gf_\C)$ 
and $\Omega_{\K_\infty}\in\cZ(\kf)$ be the Casimir operators of $\gf$ and
$\kf$, respectively, where the latter is defined with respect to  the
restriction of the normalized Killing form of $\mathfrak{g}$ to $\mathfrak{k}$.
Then with respect to the isomorphism \eqref{iso-glsect} we have
\begin{align}\label{BLO}
\widetilde{\Delta}_{\nu}=-R(\Omega)+\nu(\Omega_{\K_\infty}),
\end{align} 
where $R$ denotes the right regular representation of $G(F_\infty)$ in 
$C^\infty(G(F_\infty),\nu)$ (see \cite[Proposition 1.1]{Mia}). 

Let $e^{-t\widetilde\Delta_{\nu}}$, $t>0$, be the heat semigroup generated by 
$\widetilde\Delta_\nu$. It commutes with the action of $G(F_\infty)^1$.
With respect
to the isomorphism \eqref{iso-glsect} we may regard $e^{-t\widetilde\Delta_{\nu}}$
as a bounded operator in $L^2(G(F_\infty)^1,\nu)$, which commutes with the
action of $G(F_\infty)^1$.
Hence it is a convolution operator, i.e., there exists a smooth map
\begin{equation}\label{heatker1}
H_t^\nu\colon G(F_\infty)^1\to \End(V_\nu)
\end{equation}
such that
\[
(e^{-t\widetilde\Delta_{\nu}}\phi)(g)=\int_{G(F_\infty)^1} H_t^\nu(g^{-1}g^\prime)
(\phi(g^\prime))\;dg^\prime,\quad \phi\in L^2(G(F_\infty)^1,\nu).
\]
The kernel $H_t^\nu$ satisfies
\begin{align}\label{propH}
{H}^{\nu}_{t}(k^{-1}gk')=\nu(k)^{-1}\circ {H}^{\nu}_{t}(g)\circ\nu(k'),
\:\forall k,k'\in \K_\infty, \forall g\in G(F_\infty)^1.
\end{align}
Moreover, proceeding as in the proof of \cite[Proposition 2.4]{BM} it follows 
that $H^\nu_t$ belongs to
$(\ccC^{q}(G(F_\infty)^1)\otimes\End(V_{\nu}))^{K_\infty\times K_\infty}$ for all $q>0$,
where 
$\ccC^q(G(F_\infty)^1)$ is Harish-Chandra's Schwartz space of $L^q$-integrable 
rapidly decreasing functions on $G(F_\infty)^1$. Put
\begin{equation}\label{loctrace}
h_t^\nu(g):=\tr H_t^\nu(g),\quad g\in G(F_\infty)^1.
\end{equation}
Then $h_t^\nu\in\ccC^q(G(F_\infty)^1)$ for all $q>0$. We extend $h_t^\nu$ to a
function on $G(F_\infty)$ by
\[
h_t^\nu(ag)=h_t^\nu(g),\quad g\in G(F_\infty)^1,\; a\in \Ag.
\]
Let ${\bf 1}_{K_f}\colon G(\A_f)\to \C$ be the characteristic function of $K_f$.
Put
\begin{equation}
\chi_{K_f}:=\frac{{\bf 1}_{K_f}}{\vol(K_f)}
\end{equation}
and
\begin{equation}\label{adelic-heat-ker}
\phi_t^\nu(g):=h_t^\nu(g_\infty)\chi_{K_f}(g_f)
\end{equation}
for $g=g_\infty\cdot g_f\in G(\A)=G(F_\infty)\cdot G(\A_f)$. Now observe that all
derivatives of $\phi_t^\nu$ belong to $L^1(G(\A)^1)$. Thus $\phi_t^\nu$
belongs to $\Co(\bG(\A);K_f)$ (see section \ref{sec-prelim} for its definition).
By Theorem \ref{thm-specexpand}, $J_{\spec}$  is a distribution on 
$\Co(\bG(\A);K_f)$. Thus we can insert $\phi_t^\nu$ into the trace formula and
by Theorem \ref{thm-specexpand} we get
\begin{equation}
J_{\spec}(\phi_t^\nu)=\sum_{[M]} J_{\spec,M} (\phi_t^\nu),
\end{equation}
where the sum ranges over the conjugacy classes of Levi subgroups of $G$ and 
$J_{\spec,M} (\phi_t^{\tau,p})$ is given by \eqref{specside2}.
To analyze these terms, we proceed as in \cite[Section 13]{MM1}.
Recall that the operator $\Delta_{\mathcal{X}}(P,\lambda)$, which 
appears in the formula \eqref{specside2}, is defined by \eqref{intertw3}.
Its definition involves the intertwining operators $M_{Q|P}(\lambda)$. If we 
replace $M_{Q|P}(\lambda)$ by its restriction $M_{Q|P}(\pi,\lambda)$ to
$\cA^2_\pi(P)$, we obtain the restriction $\Delta_{\mathcal{X}}(P,\pi,\lambda)$ of
$\Delta_{\mathcal{X}}(P,\lambda)$ to $\cA^2_\pi(P)$. Similarly, let $\rho_\pi(P,\lambda)$
be the induced representation in $\bar{\cA}^2_\pi(P)$. 
Fix $s\in W(M)$ and $\beta\in\bases_{P,L_s}$.  Then for the integral 
on the right of \eqref{specside2} with $h=\phi_t^\nu$ we get
\begin{equation}\label{specside3}
\sum_{\pi\in\Pi_{\di}(M(\A))}\int_{i(\af^G_{L_s})^*}\Tr\left(
\Delta_{\dtup_{L_s}(\bss)}(P,\pi,\lambda)M(P,\pi,s)\rho_\pi(P,\lambda,\phi^\nu_t)
\right)\;d\lambda.
\end{equation}
In order to deal with the integrand, we need the following result. Let $\pi$
be a unitary  admissible representation of $\bG(F_\infty)$. Let
$A\colon \H_\pi\to\H_\pi$ be a bounded operator which is an intertwining
operator for $\pi|_{\K_\infty}$. Then $A\circ\pi(h_t^\nu)$ is a finite rank
operator. Define an operator $\tilde A$ on $\H_\pi\otimes V_\nu$ by
$\tilde A:=A\otimes\Id$. Then by \cite[(9.13)]{MM1} we have
\begin{equation}\label{TrFT}
\Tr(A\circ\pi(h_t^\nu))=e^{t(\pi(\Omega)-\nu(\Omega_{\K_\infty}))}
\Tr\left(\tilde A|_{(\H_\pi\otimes V_\nu)^{\K_\infty}}\right).
\end{equation}
We will apply this to the induced representation $\rho_\pi(P,\lambda)$. 
Let $m(\pi)$ denote the multiplicity with which $\pi$ occurs in the regular
representation of $M(\A)$ in $L^2_{\di}(\Ai M(F)\bs M(\A))$. Then
\begin{equation}\label{iso-ind}
\rho_\pi(P,\lambda)\cong \oplus_{i=1}^{m(\pi)}\Ind_{P(\A)}^{G(\A)}(\pi,\lambda).
\end{equation}
Fix positive restricted roots of $\af_P$ and let $\rho_{\af_P}$ denote the 
corresponding half-sum of these roots. For $\xi\in \Pi(M(F_\infty))$ and
$\lambda\in\af^\ast_P$ let 
\[
\pi_{\xi,\lambda}:=\Ind_{P(F_\infty)}^{G(F_\infty)}(\xi\otimes e^{i\lambda})
\]
be the unitary induced representation. Let $\xi(\Omega_M)$ be the Casimir
eigenvalue of $\xi$. Define a constant $c(\xi)$ by
\begin{equation}\label{casimir4}
c(\xi):=-\langle\rho_{\af_P},\rho_{\af_P}\rangle+\xi(\Omega_M).
\end{equation}
Then for $\lambda\in\af^\ast_P$ one has
\begin{equation}\label{casimir5}
\pi_{\xi,\lambda}(\Omega)=-\|\lambda\|^2+c(\xi)
\end{equation}
(see \cite[Theorem~8.22]{Kn}). Denote by 
$\widetilde\Delta_{\dtup_{L_s}(\bss)}(P,\pi,\nu,\lambda)$ resp.
$\widetilde M(P,\pi,\nu,s)$ the extensions of
$\widetilde\Delta_{\dtup_{L_s}(\bss)}(P,\pi,\nu,\lambda)$ resp. $M(P,\pi,\nu,s)$
to operators on $\cA^2_\pi(P)\otimes V_\nu$ as above. Using the definition of
$\phi^\nu_t$, \eqref{iso-ind} and \eqref{TrFT}, it follows that 
\eqref{specside3} is equal to
\begin{equation}\label{specside4}
\begin{split}
\sum_{\pi\in\Pi_{\di}(M(\A))}&e^{tc(\pi_\infty)}\\
&\cdot\int_{i(\af^G_{L_s})^*}e^{-t\|\lambda\|^2}\Tr\left(
\bigl(\widetilde\Delta_{\dtup_{L_s}(\bss)}(P,\pi,\nu,\lambda)
\widetilde M(P,\pi,\nu,s)\bigr)\big|_{(\cA^2_\pi(P)^{K_f}\otimes V_\nu)^{\K_\infty}}\right)
\;d\lambda.
\end{split}
\end{equation}
Since $M(P,\pi,s)$ is unitary, \eqref{specside4} can be estimated by
\begin{equation}\label{specside5}
\begin{split}
\sum_{\pi\in\Pi_{\di}(M(\A))}\dim\left(\cA^2_\pi(P)^{K_f,\nu}\right)\cdot e^{tc(\pi_\infty)}
\int_{i(\af^G_{L_s})^*}e^{-t\|\lambda\|^2}\|
\widetilde\Delta_{\dtup_{L_s}(\bss)}(P,\pi,\nu,\lambda)\|_{1,\cA^2_\pi(P)^{K_f,\nu}}\;d\lambda.
\end{split}
\end{equation}
For $\pi\in\Pi(M(\A))$ denote by $\lambda_{\pi_\infty}$ the Casimir eigenvalue
of the restriction of $\pi_\infty$ to $M(F_\infty)^1$. Given $\lambda>0$, let 
\[
\Pi_{\di}(M(\A);\lambda):=\{\pi\in\Pi_{\di}(M(\A))\colon |\lambda_{\pi_\infty}|\le
\lambda\}.
\]
For the estimation of \eqref{specside5} we need the following
auxiliary result.
\begin{lem}\label{lem-est-sum}
Let $d_M=\dim M(F_\infty)^1/\K_\infty^M$. For every open compact subgroup
$K_f\subset G(\A_f)$ and every $\nu\in\Pi(\K_\infty)$ there
exists $C>0$ such that
\begin{equation}
\sum_{\pi\in\Pi_{\di}(M(\A);\lambda)}\dim\cA^2_\pi(P)^{K_f,\nu}\le C(1+\lambda^{d_M/2})
\end{equation}
for all $\lambda\ge 0$.
\end{lem}
\begin{proof}
By \eqref{est-dim-A2} it suffices to fix $\tau\in\Pi(K_{M,\infty})$ and to
estimate the sum
\[
\sum_{\Pi_{\di}(M(\A),\lambda)} m(\pi)\dim(\cH_{\pi_f}^{K_{M,f}})
\dim(\cH_{\pi_\infty}(\tau)).
\]
To estimate the sum over $\Pi_{\cu}(M(\A),\lambda)$ we use
\cite[Lemma 3.2]{Mu3}, which holds for general reductive groups. Thus we get
\begin{equation}\label{est-sum-cusp}
\sum_{\pi\in\Pi_{\cu}(M(\A);\lambda)}\dim\cA^2_\pi(P)^{K_f,\nu}\le C(1+\lambda^{d_M/2}).
\end{equation}
Next observe that for any $\tau\in\Pi(K_{M,\infty})$, by definition of the
counting function we have
\[
\sum_{\pi\in\Pi_{\res}(M(\A);\lambda)}m(\pi)\dim(\cH_{\pi_f}^{K_{M,f}})
\dim(\cH_{\pi_\infty}(\tau))=N_{M,\res}^{K_{M,f},\tau}(\lambda).
\]
By Proposition \ref{prop-res} the right hand side is bounded by a constant
$C>0$ times $1+\lambda^{(d_M-1)/2}$. By \eqref{est-dim-A2} we get
\[
\sum_{\pi\in\Pi_{\res}(M(\A);\lambda)}\cA^2_{\pi}(P)^{K_f,\nu}\le C(1+\lambda^{(d_M-1)/2}).
\]
Together with \eqref{est-sum-cusp} the lemma follows.
\end{proof}
Next we estimate the integral in \eqref{specside5}. We use the notation
introduced above. Let $\bss=(\beta_1^\vee,\dots,\beta_m^\vee)$ and
$\dtup_{L_s}(\bss)=(Q_1,\dots,Q_m)\in\Xi_{L_s}(\bss)$ with
$Q_i=\langle P_i,P_i'\rangle$, $P_i|^{\beta_i}P_i'$, $i=1,\dots,m$.
 Using the definition \eqref{intertw3} of 
$\Delta_{\dtup_{L_s}(\bss)}(P,\pi,\nu,\lambda)$, it follows that we can bound the
integral by a constant multiple of
\begin{equation}\label{est-integral}
\dim(\nu)\int_{i(\af^G_{L_s})^*}e^{-t\|\lambda\|^2}\prod_{j=1}^m\left\| 
\delta_{P_j|P_j^\prime}(\lambda)\Big|_{\cA^2_\pi(P_j^\prime)^{K_f,\nu}}\right\|
\;d\lambda,
\end{equation}
where $\delta_{P_j|P_j^\prime}(\lambda)$ is defined by \eqref{intertw2}.
We introduce new coordinates $s_j:=\langle\lambda,\beta_j^\vee\rangle$,
$j=1,\dots,m$, on $(\af^G_{L_s,\C})^\ast$. 
Using \eqref{normalization} and \eqref{intertw2}, we can write
\begin{equation}\label{delta}
\delta_{P_i|P_i^\prime}(\lambda)=\frac{n^\prime_{\beta_i}(\pi,s_i)}{n_{\beta_i}(\pi,s_i)}
+j_{P_i^\prime}\circ(\Id\otimes R_{P_i|P_i^\prime}(\pi,s_i)^{-1}
R^\prime_{P_i|P_i^\prime}(\pi,s_i))\circ j_{P_i^\prime}^{-1}.
\end{equation}
Now assume that $G$ satisfies property (L). By \cite[Prop. 3.8]{FL1}, $G$
satisfies property (TWN+) (tempered winding numbers, strong version). This
means that for any proper Levi subgroup $M$ of $G$ defined over $\Q$, and any
root $\alpha\in\Sigma_M$, and $T\in\R$ the following estimate holds
\begin{equation}
\int_T^{T+1}|n_\alpha^\prime(\pi,it)| dt\ll\log(|T|+\Lambda(\pi_\infty;p^{\sico})+
\level(\pi;p^{\sico}))
\end{equation}
for all $\Pi_{\di}(M(\A))$. In the proof of Corollary \ref{cor-est-poles}
it was proved that there exists $C>0$ such that for all
$\pi\in\Pi_{\di}(M(\A);K_f,\sigma)$ one has
\begin{equation}\label{estim6}
\level(\pi;p^{\sico})\le C,\quad{\text and}\quad \Lambda(\pi_\infty;p^{\sico})\le
C(1+\lambda_{\pi_\infty}^2).
\end{equation}
Hence we get
\begin{equation}\label{int-log-deriv}
\int_T^{T+1}|n_\alpha^\prime(\pi,it)| dt\ll\log(|T|+1+\lambda_{\pi_\infty}^2)
\end{equation}
For all $T\in\R$ and $\pi\in\Pi_{\di}(M(\A);K_f,\sigma)$.
\begin{lem}\label{lem-log-der}
There exists $C>0$ such that
\begin{equation}
\int_\R\left|\frac{n_\alpha^\prime(\pi,i\lambda)}{n_\alpha(\pi,i\lambda)}\right|
e^{-t\lambda^2}
d\lambda\le C(1+\log(1+\lambda_{\pi_\infty}^2))\frac{1+|\log t|}{\sqrt{t}}
\end{equation}
for all $0< t\le 1$ and $\pi\in\Pi_{\di}(M(\A);K_f,\sigma)$.
\end{lem}
\begin{proof}
By \eqref{funct-equ-5} it follows that $|n_\alpha(\pi,i\lambda)|=1$ for $\lambda
\in\R$. Furthermore, by \eqref{int-log-deriv} we have
\begin{equation}\label{int-log-deriv1}
\int_0^\lambda|n_\alpha^\prime(\pi,iu)|du\le C|\lambda|\log(|\lambda|+1+
\lambda_{\pi_\infty}^2)
\end{equation} 
for $\lambda\in\R$ and $\pi\in\Pi_{\di}(M(\A);K_f,\sigma)$. Hence, using
integration by parts, the integral on the left hand side of the claimed
inequality equals
\[
2t\int_\R\int_0^\lambda|n_\alpha^\prime(\pi,iu)|du \lambda e^{-t\lambda^2} d\lambda.
\]
By \eqref{int-log-deriv1} we get
\[
\begin{split}
\int_\R|n_\alpha^\prime(\pi,i\lambda)|e^{-t\lambda^2}d\lambda&\le
C t\int_\R\log(|\lambda|+1+\lambda_{\pi_\infty}^2)\lambda^2 e^{-t\lambda^2} d\lambda\\
&\le C_1(1+\log(1+\lambda_{\pi_\infty}^2))\frac{1+|\log t|}{\sqrt{t}}
\end{split}
\]
for all $0< t\le 1$ and $\pi\in\Pi_{\di}(M(\A);K_f,\sigma)$.
\end{proof}
Let $l_s=\dim(A_{L_s}/A_G)$.
Combining \eqref{delta}, Lemma \ref{lem-log-der} and Proposition
\ref{prop-log-der} it follows that there exists $C>0$ 
\begin{equation}\label{estim7}
\int_{i(\af^G_{L_s})^*}e^{-t\|\lambda\|^2}\prod_{i=1}^m\left\| 
\delta_{P_i|P_i^\prime}(\lambda)\Big|_{\cA^2_\pi(P_i^\prime)^{K_f,\nu}}\right\|
\;d\lambda\le C(1+|\log t|)^{l_s} t^{-{l_s}/2} \left(1+\log(1+\lambda_{\pi_\infty}^2)
\right)^{l_s}
\end{equation}
for all $0<t\le 1$ and $\pi\in\Pi_{\di}(M(\A))$ with $\cA^2_\pi(P)^{K_f,\nu}$. 
Now we can estimate \eqref{specside5}. Note that $c(\pi_\infty)=\lambda_{\pi_\infty}
-\|\rho_{\af_P}\|^2$. Using \eqref{estim7}, it follows that \eqref{specside3} can
be estimated by a constant multiple of
\begin{equation}\label{estim8}
(1+|\log t|)^{l_s} t^{-{l_s}/2}\sum_{\pi\in\Pi_{\di}(M(\A))}\dim\cA^2_\pi(P)^{K_f,\nu}
\left(1+\log(1+\lambda_{\pi_\infty}^2)\right)^{l_s} e^{t\lambda_{\pi_\infty}}.
\end{equation}
 Let $X_M=M(F_\infty)^1/K_{M,\infty}$ and
$m=\dim X_M$. Using Lemma \ref{lem-est-sum} it follows that for every
$\varepsilon>0$ there
exists $C>0$ such that the series is bounded by $Ct^{-m/2-\varepsilon}$ for
$0<t\le 1$. Together with \eqref{estim8} this yields the following proposition.
\begin{prop}\label{prop-bound-spec-side}
Let $M\in\cL$. Let $m=\dim X_M$ and $l=\max_{s\in W(M)}\dim(A_{L_s}/A_G)$. For
every $\varepsilon>0$ there exists $C>0$ such that
\[
|J_{\spec,M}(\phi_t^\nu)|\le C t^{-(m+l)/2-\varepsilon}
\]
for all $0<t\le 1$.
\end{prop}
Now we distinguish two cases. First assume that $M=G$. Then $A_M=A_G$. Let
$R^1_{\di}$ be the restriction of the regular representation $R^1$ of $G(\A)^1$
in $L^2(G(F)\bs G(\A)^1)$ to the discrete subspace. Then $J_{\spec,G}(\phi_t^\nu)=
Tr(R^1_{\di}(\phi_t^{\nu,1})$. Let $R_{\di}$ be the regular representation of
$G(\A)$ in $L^2(\Ag G(F)\bs G(\A))$. Then the operator $R_{\di}(\phi_t^\nu)$
is isomorphic to $R^1(\phi_t^{\nu,1})$. Thus
\begin{equation}\label{spec-side-G}
J_{\spec,G}(\phi_t^{\nu,1})=\Tr(R_{\di}(\phi_t^\nu)).
\end{equation}
Given $\pi\in\Pi_{\di}(G(\A))$, let $m(\pi)$ denote the multiplicity with
which $\pi$ occurs in the regular representation of $G(\A)$ in
$L^2(\Ag G(F)\bs G(\A))$. Then, using Corollary 2.2 in \cite{BM}, we get
\begin{equation}\label{spec-side-G-1}
J_{\spec,G}(\phi_t^{\nu,1})=\sum_{\pi\in\Pi_{\di}(G(\A),\xi_0)}m(\pi)\dim(\cH_{\pi_f}^{K_f})
\dim(\cH_{\pi_\infty}\otimes V_\nu)^{K_\infty} e^{t\lambda_{\pi_\infty}}.
\end{equation}
Now assume that $M$ is a proper Levi subgroup. Let $P=M\ltimes N$. Let
$\widetilde X=G(F_\infty)^1/\K_\infty$. Then
\[
\widetilde X\cong X_M\times \Ai/\Ag\times N(F_\infty).
\]
Since $l=\max_{s\in W(M)}\dim(A_{L_s}/A_G)\le \dim(A_M/A_G)$, it follows that
$m+l\le \dim\widetilde X-1$. Thus using this together with Proposition
\ref{prop-bound-spec-side}, we get
\begin{theo}\label{thm-spec-side}
Suppose that $G$ satisfies property (L). Let $n=\dim\widetilde X$. For every
open compact subgroup $K_f$ of $G(\A_f)$ and every $\nu\in\Pi(\K_\infty)$ the
spectral side of the trace formula, evaluated at $\phi_t^{\nu,1}$, satisfies
\begin{equation}
\begin{split}
J_{\spec}(\phi_t^{\nu,1})=&\sum_{\pi\in\Pi_{\di}(G(\A))}m(\pi)\dim(\cH_{\pi_f}^{K_f})
\dim(\cH_{\pi_\infty}\otimes V_\nu)^{K_\infty} e^{t\lambda_{\pi_\infty}}\\
&+O(t^{-(n-1)/2})
\end{split}
\end{equation}
as $t\to 0^+$.
\end{theo}

\section{Geometric side of the trace formula}
\setcounter{equation}{0}
As before, $G$ is a reductive group over a number field $F$.
In this section we consider the geometric side of the Arthur trace formula
$J_{\geo}$ evaluated at $\phi_t^\nu$ and determine the asymptotic behavior of
$J_{\geo}(\phi_t^\nu)$ as $t\to 0$. The geometric side $J_{\geo}$ of the trace
formula was introduced in \cite{Ar1}. See also \cite{Ar5}. For
$f\in C_c^\infty(G(\A)^1)$, Arthur has defined $J_{\geo}(f)$ as the value at a
point $T_0\in\af_0$, specified in \cite[Lemma 1.1]{Ar3}, of a polynomial
$J^T(f)$ on
$\af_0$. By \cite[Theorem 7.1]{FL3}, $J_{\geo}(f)$ is absolutely convergent
for all $f\in\Co(G(\A);K_f)$. Let $\phi_t^\nu\in \Co(G(\A);K_f)$ be the
function which is defined by \eqref{adelic-heat-ker}.
Then $J_{\geo}(\phi_t^\nu)$ is well defined. In \cite[(1.5)]{MM2},
the regularized trace of the heat operator $e^{-t\Delta_\nu}$ was defined as
\[
\Tr_{\reg}\left(e^{-t\Delta_\nu}\right):=J_{\geo}(\phi_t^\nu).
\]
Then in \cite[Theorem 1.1]{MM2} an asymptotic expansion  of
$\Tr_{\reg}\left(e^{-t\Delta_\nu}\right)$ as $t\to 0$ has been established. For our
purpose we need to know the precise form of the term of order $t^{-n/2}$, where
$n=\dim X$. To this end we briefly recall the derivation of the asymptotic
expansion. The first step is to replace $\phi_t^\nu$ by an appropriate
compactly supported function $\tilde \phi_t^\nu$ with support concentrated near
the identity element. Such a function is constructed as follows.

Let $d(\cdot, \cdot): \tilde X\times\tilde X\longrightarrow [0,\infty)$ be the geodesic distance on $\tilde X$, and put $r(g_\infty)=d(g_\infty x_0, x_0)$ where $x_0=\K_\infty\in \tilde X$ is the base point. Let $0<a<b$ be sufficiently small real numbers and let $\beta:\R\longrightarrow [0,\infty)$ be a smooth function supported in $[-b,b]$ such that $\beta(y)=1$ for $0\le |y|\le a$, and $0\le\beta(y)\le1$ for $|y|>a$. Define 
\begin{equation}\label{eq:def:psi}
 \psi^\nu_t(g_\infty) 
 = \beta(r(g_\infty)) h_t^\nu(g_\infty).
\end{equation}
and
\begin{equation}\label{eq:def:phi:tilde}
 \tilde\phi^\nu_t(g)
 = \psi^\nu_t(g_\infty)\chi_{K_f}(g_f)
\end{equation}
for $g=g_\infty\cdot g_f\in G(\A)=G(F_\infty)\cdot G(\A_f)$.
Then $\tilde\phi^\nu_t\in C_c^\infty(\bG(\A)^1)$ and $\psi^\nu_t\in
C^\infty_c(\bG(F_\infty)^1)$.
By \cite[Proposition~12.1]{MM1} there is some $c>0$ such that for every $0<t\le 1$ we have
\begin{equation}\label{geom-approx}
 \left|J_{\geo}(\phi_t^\nu)- J_{\geo}(\tilde\phi_t^\nu)\right|
 \ll e^{-c/t}.
\end{equation}
We note that in \cite[Sect. 12]{MM1} we made the assumption that $\bG=\GL(n)$
or $\bG=\SL(n)$. However, the proof of the  proposition holds without any 
restriction on $\bG$. The next result reduces the considerations to the 
unipotent contribution to the geometric side.
Before we state it, we recall the coarse geometric expansion of Arthur's trace formula \cite[Sect. 10]{Ar5}: Two elements $\gamma_1, \gamma_2\in G(F)$ are called coarsely equivalent if their semisimple parts (in the Jordan decomposition)  are conjugate in $\bG(F)$. Then for any $f\in C^\infty_c(\bG(\A)^1)$ we have
\[
 J_{\geo}(f)=\sum_{\of} J_{\of} (f),
\]
where $\of$ runs over the coarse equivalence classes in $\bG(F)$, and the distribution $J_{\of} $ is supported in the set of all $g\in \bG(\A)^1$ whose semisimple part is conjugate in $\bG(\A)$ to some semisimple element in $\of$. If $\of\neq\of'$, the supports of $J_{\of}$ and $J_{\of'}$ are disjoint.
Note that the set of unipotent elements in $\bG(F)$ constitute a single
equivalence class $\of_{\text{unip}}$ and we write
$J_{\text{unip}}=J_{\of_{\text{unip}}}$. 
Assume that $K_f$ is neat. If the support of $\beta$ is sufficiently small
then by \cite[Prop. 3.1]{MM2} we have 
\begin{equation}\label{unipotent1}
J_{\geo}(\tilde\phi_t^\nu)=J_{\unip}(\tilde\phi_t^\nu).
\end{equation}
By \eqref{geom-approx} and \eqref{unipotent1} the problem is reduced to the
study of the asymptotic behavior of $J_{\unip}(\tilde\phi_t^\nu)$ as $t\searrow0$.
For this purpose we use Arthur's fine geometric expansion of  $J_{\text{unip}}$.
In order to state it we need to introduce some notation.

Let $S$ be a finite
set of places of $F$, which includes the archimedean places, such that
$K_v=\cpt_v$ for $v\not\in S$. Let $G(F_S)^1=G(F_S)\cap G(\A)^1$.
Let $M\in\CmL$. Following Arthur, we introduce an equivalence relation on the set of unipotent elements in $M(F)$  that depends on the set $S$: Two unipotent elements $u,v\in M(F)$ are $(M,S)$-equivalent if and only if $u$ and $v$ are conjugate in $M(F_S)$. We denote the equivalence class of $u$ by $[u]_S\subseteq M(F)$ and let $\CmU_S^M$ denote the set of all such equivalence classes.

Note that two equivalent unipotent elements define the same unipotent conjugacy class in $M(F_S)$, so we can view $\CmU^M_S$ also as the set of unipotent conjugacy classes in $M(F_S)$ which have at least one $F$-rational representative,
and we denote the corresponding conjugacy class by $[u]_S$ as well. 
\begin{remark}
 \begin{enumerate}[label=(\roman{*})]
  \item If $T\subseteq S$, then we get a well-defined map $\CmU^M_S\ni[u]_S\mapsto [u]_T\in\CmU^M_T$. 

  \item If $G=\GL(n)$, the equivalence relation is independent of $S$ and is the same as conjugation in $M(F)$.
 \end{enumerate}
\end{remark}
For $[u]_S\in \CmU^M_S$ and $f_S\in C^\infty_c(\bG(F_S)^1)$, Arthur associates a weighted orbital integral $J_M^G([u]_S, f_S)$ \cite{Ar6} which is a distribution supported on  the $\bG(F_S)$-conjugacy class induced from
$[u]_S\subseteq M(F_S)$. Let $\One_{\cpt^S}\in C_c^\infty(\bG(\A^S))$ be the characteristic function
of $\cpt^S$, if $f_S\in C_c^\infty(\bG(F_S)^1)$. Put
$f=f_S\One_{\cpt^S}\in C^\infty_c(\bG(\A)^1)$. By \cite[Corollary 8.3]{Ar7}
there exist unique constants $a^M([u]_S, S)\in\C$ and conjugacy classes
$[u]_S\in \CmU^M_S$, such that for all $f_S\in C_c^{\infty}(\bG(F_S)^1)$ we have
\begin{equation}\label{fine-geom-expans}
J_{\text{unip}}(f)
= \sum_{M\in \CmL}\sum_{[u]_S\in\CmU^M_S} a^M([u]_S, S) J_M^{\bG}([u]_S, f_S).
\end{equation}
In fact, Corollary 8.3 in \cite{Ar7} is stated only for reductive groups over
$\Q$. However, at the end of the article, Arthur explains that all results of
the article hold equally well for reductive groups over a number field $F$.

In general, there is not much known about the coefficients $a^M([u]_S, S)$.
However, for our purpose we only need to know $a^G(1,S)$, which by
\cite[Corollary 8.5]{Ar7} is given by
\begin{equation}\label{lead-coeff}
a^G(1,S)=\vol(G(F)\bs G(\A)^1).
\end{equation}
Write $S=S_\infty\sqcup S_0$. Then $K_f=K_{S_0}\K^S$. Recall that by
\eqref{eq:def:phi:tilde}
\[
\tilde\phi_t^\nu=\frac{1}{\vol(K_f)}\psi_t^\nu\cdot{\bf 1}_{K_{S_0}}
\cdot{\bf 1}_{\K^S}.
\]
Then by \eqref{fine-geom-expans} we get
\begin{equation}\label{orb-int}
J_{\unip}(\tilde\phi_t^\nu)=\sum_{M\in \CmL}\sum_{[u]_S\in\CmU^M_S}
\frac{a^M([u]_S, S)}{\vol(K_f)} J_M^{\bG}([u]_S, \psi_t^\nu\cdot {\bf 1}_{K_{S_0}}).
\end{equation}
Using \eqref{lead-coeff}, the term that corresponds to $(G,1)$ equals
\begin{equation}\label{lead-term2}
\frac{\vol(G(F)\bs G(\A)^1)}{\vol(K_f)}h_t^\nu(1).
\end{equation}
To deal with the weighted orbital integrals in general, we use Arthur's
splitting formula
\cite[(18.7)]{Ar5}, which we recall next. Let $S$ be any finite set of places
of $F$ which not necessarily contains the archimedean places. Let $L\in\cL(M)$
and $Q\in\cP(L)$. Given $f_S\in G(F_S)$ let
\[
f_{S,Q}(m)=\delta_Q(m)^{1/2}\int_{\K_S}\int_{N_Q(F_S)} f_S(k^{-1}mnk)\, dn\, dk,\quad
m\in L.
\]
Suppose that $S=S_1\cup S_2$ with $S_1$, $S_2$ non-empty and disjoint, and that $f_S$ is the restriction of a product $f_{S_1}f_{S_2}$ to $G(F_S)^1$ with
$f_{S_j}\in C^\infty(G(F_{S_j}))$, $j=1,2$. Then the splitting formula states that
\begin{equation}\label{eq-splitting}
 J_M^{\bG}([u]_S, f_S) 
 = \sum_{L_1,\,L_2\in\CmL(M)} d_M^{\bG}(L_1,L_2) J_{M}^{L_1}([u]_{S_1}, f_{S_1,Q_1}) J_M^{L_2}([u]_{S_2}, f_{S_2,Q_2}),
\end{equation}
where the notation is as follows: The $d_M^{\bG}(L_1,L_2)\in\R$ are certain constants which depend only on $M, L_1, L_2, \bG$ but not on $S$. In fact, $d_M^{\bG}(L_1, L_2)$ is non-zero only if the natural map $\af_M^{L_1}\oplus \af_M^{L_2}\longrightarrow\af_M^{\bG}$ is an isomorphism. The $Q_j$ are arbitrary elements in $\CmP(L_j)$ and $[u]_{S_j}\in\cU_{S_j}^M$ is the image of $[u]_S$ under the
canonical map $\cU_S^M\to \cU_{S_j}^M$.
Finally, $J_M^{L_j}([u]_{S_j}, \cdot)$ denotes the $S_j$-adic distribution which
is supported on the $L_j(F_{S_j})$-conjugacy class which is induced from
$[u]_{S_j}\subseteq M(F_{S_j})$ and is defined as in \cite{Ar6}.

We apply the splitting formula to the weighted orbital integral on the right
of \eqref{orb-int} with $S_1=S_\infty$ and $S_2=S_0$. We obtain
\begin{equation}\label{splitting1}
 J_M^{\bG}([u]_S, \psi_t^\nu \cdot \One_{K_{S_0}})
 = \sum_{L_1,\,L_2\in\CmL(M)} d_M^{\bG}(L_1,L_2) J_{M}^{L_1}([u]_{\infty}, \psi^\nu_{t,Q_1}) J_M^{L_2}([u]_{S_0}, \One_{K_{S_0},Q_2}).
\end{equation}
This is a finite sum with $d_M^{\bG}(L_1,L_2)$ and
$J_M^{L_2}([u]_{S_0}, \One_{K_{S_0},Q_2})$ independent of $t$. The asymptotic
expansion in $t$ of weighted orbital integrals of the form
$J_{M}^{L_1}([u]_{\infty},\psi^\nu_{t,Q_1})$ has been determined in
\cite[Prop. 7.2]{MM2}. This Corollary has been proved for groups over $\Q$.
However, the proof can be easily extended to reductive groups over $F$, either
by repeating the arguments or using restriction of scalars,

We recall the proposition. Let $M\in\cL$, $P_1=M_1N_1\in
\cF(M)$ and $\cO\subset M(F_\infty)$ a unipotent conjugacy class in
$M(F_\infty)$. Let
$d_{\cO}=\dim \cO^{G(F_\infty)^1}$ be the dimension of the unipotent orbit in
$G(F_\infty)^1$ induced from $M(F_\infty)$, and let $r^{M_1}_M=\dim \af_M^{M_1}$.
Then there
exist constants $b_{ij}=b_ij(M,\cO)\in\C$, $j\ge 0$, $0\le i\le r_M^{M_1}$, such
that for $0<t\le 1$
\begin{equation}\label{asympt-exp7}
 J_M^{M_1}(\CmO, \left(\psi_t^\nu\right)_{P_1})
\sim t^{-n/2  + d_\CmO^G/4} \sum_{j=0}^\infty \sum_{i=0}^{r_M^{M_1}} b_{ij} t^{j/2}
(\log t)^i.
\end{equation}
If $K_f$ is neat, then $d_{\cO}^G>0$ for $\cO\neq 1$. Combining
\eqref{geom-approx}--\eqref{asympt-exp7} it follows that for every $\nu\in
\Pi(K_\infty)$ there exist $\varepsilon>0$ such that
\begin{equation}\label{asympt-exp4}
J_{\geo}(\phi_t^\nu)=\vol(X(K_f))h_t^\nu(1)+O(t^{-n/2+\varepsilon})
\end{equation}
for all $0<t\le 1$.
By \cite[Lemma 2.3]{Mu3} we have
\begin{equation}
h_t^\nu(1)=\frac{\dim(\nu)}{(4\pi)^{n/2}}t^{-n/2}+O(t^{-(n-1)/2})
\end{equation}
as $t\searrow0$.
Together with \eqref{asympt-exp4} we obtain the following proposition.
\begin{prop}\label{prop-geo-side}
Let $G$ be a reductive group over a number field $F$. Let $K_f$ be an open
compact subgroup of $G(\A)$. Assume that $K_f$ is neat.
Then for every $\nu\in\Pi(K_\infty)$
there exists $\varepsilon>0$ such that we have
\[
J_{\geo}(\phi_t^\nu)=\frac{\dim(\nu)\vol(X(K_f))}{(4\pi)^{n/2}}t^{-n/2}+
O(t^{-n/2+\varepsilon})
\]
for all $0<t\le 1$.
\end{prop}

\section{Proof of the main theorem}\label{sect-proof-main-thm}
\setcounter{equation}{0}
First we establish the adelic version of the Weyl law, which is Theorem
\ref{main-thm-adelic}. Let $G_0$ be a reductive algebraic group over a number
field $F$ and let $G=\Res_{F/\Q}(G_0)$ be the reductive group over $\Q$ which
is obtained from $G_0$ by restriction of scalars. We shall use
the (non-invariant) Arthur trace formula for reductive groups over $F$
to deduce the Weyl law for $G_0$. Then we use the properties of the restriction
of scalars to show that this is equivalent to the Weyl law for $G$.

To begin with we recall that the coarse Arthur trace formula over $F$ is the
identity
\[
J_{\spec}(f)=J_{\geo}(f),\quad f\in\Co(G(\A_F)^1).
\]
Applied to $\phi_t^\nu$ we get the equality
\[
J_{\spec}(\phi_t^\nu)=J_{\geo}(\phi_t^\nu),\quad t>0.
\]
Assume that $G_0$ satisfies property (L). Let $K_{0,f}$ be an open compact
subgroup of $G_0(\A_{F,f})$. We assume that $K_{0,f}$ is neat.
Combining Theorem \ref{thm-spec-side} and Proposition \ref{prop-geo-side},
we obtain
\begin{equation}
\begin{split}
\sum_{\pi\in\Pi_{\di}(G_0(\A_F))}m(\pi)\dim(\cH_{\pi_f}^{K_{0,f}})
\dim\left(\cH_{\pi_\infty}\otimes V_\nu\right)^{K_\infty}\,e^{t\lambda_{\pi_\infty}}=&
\frac{\dim(\nu)\,\vol(X(K_{0,f}))}{(4\pi)^{n/2}}t^{-n/2}\\
&+O(t^{-(n-1)/2})
\end{split}
\end{equation}
as $t\searrow0$, where $X(K_{0,f})$ is defined by \eqref{adel-quot2}.
Let $N^{K_{0,f},\nu}_{\di}(\lambda)$ be the adelic counting
function defined by \eqref{adelic-count-fct1}. 
Applying Karamata's theorem \cite[p. 446]{Fe}, we obtain
\begin{equation}\label{adelic-weyl-law3}
N^{K_{0,f},\nu}_{\di}(\lambda)=
\frac{\dim(\nu)\vol(X(K_{0,f}))}{(4\pi)^{n/2}\Gamma(n/2+1)}\lambda^{n/2}+
o(\lambda^{n/2})
\end{equation}
as $\lambda\to\infty$. By Proposition \ref{prop-res} we get
\begin{equation}\label{adelic-weyl-law4}
N^{K_{0,f},\nu}_{\cu}(\lambda)=
\frac{\dim(\nu)\vol(X(K_{0,f}))}{(4\pi)^{n/2}\Gamma(n/2+1)}\lambda^{n/2}+
o(\lambda^{n/2}).
\end{equation}
This is the first part of the Weyl law for $G_0$. The second part is the
estimation of the counting function of the residual spectrum which
follows from Proposition \ref{prop-res} for $M=G$.

Next we show that Theorem \ref{main-thm-adelic} is compatible with the
restriction of scalars. To begin with we recall some facts about the Weil
restriction of scalars \cite{We}, \cite{Bo2}. By \cite[Theorem 1.3.2]{We} we
have
\begin{equation}\label{restr-scal1}
G(\Q_v)=\prod_{w|v}G_0(F_w).
\end{equation}
for all places $v$ of $\Q$. In particular, we get
\begin{equation}\label{restr-scal2}
G(\A_\Q)=G_0(\A_F),\quad G(\R)=G_0(F_\infty)=\prod_{w\in S_\infty}G_0(F_w),\quad
G(\Q)=G_0(F).
\end{equation}
Therefore we obtain a bijection of the automorphic representations of $G_0$
with those of $G$. Also the regular representation of $G(\A_\Q)$ on
$L^2(G(\Q)\bs G(\A_\Q))$ is equivalent to the regular representation of
$G_0(\A_F)$ on $L^2(G_0(F)\bs G_0(\A_F))$. 
Furthermore, by \cite[5.2]{Bo2}, the map
$P_0\mapsto \Res_{F/\Q}(P_0)$
induces a bijection between parabolic subgroups of $G_0$, defined over $F$, and
parabolic subgroups of $G$, defined over $\Q$, and \eqref{restr-scal1}
and \eqref{restr-scal2} continue to hold for $F$-parabolic subgroups of $G_0$.
Let $P_0=M_{P_0}N_{P_0}$ be a $F$-parabolic subgroup of $G_0$ and
$P=\Res_{F/\Q}(P_0)$. Let $f\in L^2(G(\Q)\bs G(\A_Q)^1)$ and $\tilde f\in
L^2(G_0(F)\bs G_0(\A_F)^1)$ correspond to each other. Then
\begin{equation}
\int_{N_P(\Q)\bs N_P(\A_\Q)}f(nx) dn=\int_{N_{P_0}(F)\bs N_{P_0}(\A_F)}\tilde f(n_0x)dn_0.
\end{equation}
Hence we get
\[
L^2_{\cu}(G(\Q)\bs G(\A_\Q)^1)\cong L^2_{\cu}(G_0(F)\bs G_0(\A_F)^1).
\]
The same holds for the residual spectrum. It follows that the counting
functions for $G$ and $G_0$ coincide. Thus \eqref{adelic-weyl-law3} and
\eqref{adelic-weyl-law4} hold for the counting function of $G=\Res_{F/\Q}(G_0)$.
This proves Theorem \ref{main-thm-adelic}.

Next we deduce Theorem \ref{main-thm} from Theorem \ref{main-thm-adelic}.
To this end we express the counting function in a different way. Let $\sigma
\in\Pi(\K_\infty)$. Let
$L^2(A_G\Gamma\bs G(\R)),\sigma)$ be defined as in \eqref{sigma-iso}.  Given
$\pi\in\Pi(G(\R))$ let $m_\Gamma(\pi)$ be the multiplicity with which
$\pi$ occurs in the regular representation $R_\Gamma$ in
$L^2(A_G\Gamma\bs G(\R))$. Let $L^2_{\di}(A_G\Gamma\bs G(\R))$
the span of all irreducible subrepresentations. Then
\begin{equation}
(L^2_{\di}(A_G\Gamma\bs G(\R))\otimes V_\sigma)^{\K_\infty}=
\bigoplus_{\pi\in\Pi(G(\R))}m_\Gamma(\pi)(\cH_\pi\otimes V_\sigma)^{\K_\infty}.
\end{equation}
For $\tau\in\Pi(G(\R))$ let $\lambda_\tau$ be the Casimir eigenvalue of $\tau$,
i.e., the eigenvalue of
$R_\Gamma(\Omega_{G(\R)})$ on $\cH_{\tau}$. Then it follows that
$(\cH_\tau\otimes V_\sigma)^{K_\infty}$ is an eigenspace of
$\Delta_\sigma=-R_\Gamma(\Omega_{G(\R)})$ with
eigenvalue $-\lambda_\tau$. It follows that
\begin{equation}\label{eigenv-count-fct}
N_{\Gamma,\di}(\lambda;\sigma):=\sum_{\substack{\pi\in\Pi_(G(\R))\\-\lambda_\pi\le\lambda}}
m_\Gamma(\pi)\dim(\cH_\pi\otimes V_\sigma)^{\K_\infty}.
\end{equation}
There are similar formulas for $N_{\Gamma,\cu}(\lambda,\sigma)$ and 
$N_{\Gamma,\res}(\lambda,\sigma)$

Now we establish the relation between the adelic
and real counting functions. Let $K_f\subset G(\A_f)$ be an open compact
subgroup. Let $\Gamma_i\subset G(\Q)$, $i=1,...,l$, be determined by
\eqref{adel-quot1}. 
The relation between the classical and adelic counting functions is described
by the following lemma.
\begin{lem}\label{lem-counting-fct}
Let $\sigma\in\Pi(\K_\infty)$. Then
\[
N^{K_f,\sigma}_{\di}(\lambda)=\sum_{i=1}^l N_{\Gamma_i,\di}(\lambda,\sigma)
\]
for $\lambda\ge 0$.
The same equality holds for the counting functions of the cuspidal and residual
spectrum.
\end{lem}
\begin{proof} Given $\tau\in\Pi(G(\R))$, let $m(\tau)$ be the multiplicity
with which $\tau$ occurs in the regular representation
By \eqref{multipl1} with respect to $F=\Q$ we have
\[
\begin{split}
\sum_{\substack{\tau\in\Pi_{\di}(G(\R))\\-\lambda_\tau\le\lambda}}
m(\tau)(\cH_{\tau}\otimes V_\nu)^{\K_\infty}&=
\sum_{\substack{\tau\in\Pi_{\di}(G(\R))\\-\lambda_\tau\le\lambda}}
\sum_{\substack{\pi\in\Pi_{\di}(G(\A))\\\pi_\infty=\tau}}
m(\pi)\dim(\cH_{\pi_f}^{K_f})\dim(\cH_{\pi_\infty}\otimes
V_\nu)^{\K_\infty}\\
&=\sum_{\substack{\pi\in\Pi_{\di}(G(\A))\\-\lambda_{\pi_\infty}\le\lambda}}
m(\pi)\dim(\cH_{\pi_f}^{K_f})\dim(\cH_{\pi_\infty}\otimes V_\nu)^{\K_\infty}\\
&=N^{K_f,\sigma}_{\di}(\lambda).
\end{split}
\]
Combined with \eqref{multipl2} it follows that
\begin{equation}\label{iso-multipl}
\begin{split}
N^{K_f,\sigma}_{\di}(\lambda)=\sum_{\substack{\tau\in\Pi_{\di}(G(\R))\\
-\lambda_{\tau}\le\lambda}}m(\tau)(\cH_{\tau}\otimes V_\nu)^{K_\infty}&=\sum_{i=1}^l
\sum_{\substack{\tau\in\Pi(G(\R))\\-\lambda_\tau\le\lambda}}m_{\Gamma_i}(\tau)
(\cH_\tau\otimes V_\nu)^{K_\infty}\\
&=\sum_{i=1}^l N_{\Gamma_i,\di}(\lambda,\nu).
\end{split}
\end{equation}
\end{proof}

Let $K_f$ and $\Gamma_i$, $i=1,...,l$, be as above. By Lemma
\ref{lem-counting-fct} we have 
\begin{equation}\label{comp-count-fct5}
N^{K_f,\nu}_{\cu}(\lambda)=\sum_{i=1}^l N_{\Gamma_i,\cu}(\lambda;\nu).
\end{equation}
Furthermore, by \eqref{adel-quot1} we have
\begin{equation}\label{vol-comp}
\vol(X(K_f))=\vol(A_G G(\Q)\bs G(\A)/K_\infty K_f)=\sum_{i=1}^l
\vol(\Gamma_i\bs \widetilde X),
\end{equation}
where $\widetilde X=A_G\bs G(\R)/K_\infty$. Thus by \eqref{adelic-weyl-law3} we
obtain
\begin{equation}\label{sum-weyl-laws}
\lim_{\lambda\to\infty}\sum_{i=1}^l \frac{N_{\Gamma_i,\cu}(\lambda;\nu)}
{\lambda^{n/2}}=\sum_{i=1}^l\frac{\dim(\nu)\vol(\Gamma_i\bs\widetilde X)}
{(4\pi)^{n/2}\Gamma(n/2+1)}.
\end{equation}
Now we argue as in \cite[Sect. 6.3]{LV} \eqref{upper-bd}. By \eqref{upper-bd}
we have
\[
\limsup_{\lambda\to\infty}\frac{N_{\Gamma_i,\cu}(\lambda;\nu)}{\lambda^{n/2}}
\le \frac{\dim(\nu)\vol(\Gamma_i\bs\widetilde X)}{(4\pi)^{n/2}\Gamma(n/2+1)}
\]
for $i=1,...,l$. Combined with \eqref{sum-weyl-laws} it follows that
\begin{equation}\label{weyl-law-real}
N_{\Gamma_i,\cu}(\lambda;\nu)=\frac{\dim(\nu)\vol(\Gamma_i\bs \widetilde X)}{(4\pi)^{n/2}\Gamma(\frac{n}{2}+1)}\lambda^{n/2}+o(\lambda^{n/2})
\end{equation}
for $i=1,...,l$. 

Now let $\Gamma\subset G(\Q)$ be a congruence subgroup. 
By the definition of a congruence subgroup (see sect. \ref{sect-arithm-mfds})
there exists a compact open subgroup $K_f\subset G(\Q)$ such that
$\Gamma=K_f\cap G(\Q)$. Let $\Gamma_i$, $i=1,...,l$, be defined by
\eqref{congr-subgr}. Then $\Gamma=\Gamma_1$ and the first part of Theorem
\ref{main-thm} follows from \eqref{weyl-law-real}.

To establish the second part of Theorem \ref{main-thm}, we observe that by
Lemma \ref{lem-counting-fct} we have
\begin{equation}\label{comp-res-count-fct}
N_{\res}^{K_f,\nu}(\lambda)=\sum_{i=1}^l N_{\Gamma_i,\res}(\lambda,\nu)
\end{equation}
for $\lambda\ge 0$. Since each summand on the right hand side is $\ge 0$,
and $\Gamma=\Gamma_1$, \eqref{estim-res6} yields 
\begin{equation}
N_{\Gamma,\res}(\lambda,\nu)\le C(1+\lambda^{(n-1)/2}),\quad \lambda\ge 0.
\end{equation}
This completes the proof of Theorem \ref{main-thm}.


\begin{thebibliography}{MMM1}

\bibitem[Ar1]{Ar1} J. Arthur, {\it A trace formula for reductive groups. 
I. Terms associated to classes in $G(\Q)$}. 
Duke Math. J. {\bf 45} (1978), no. 4, 911 -- 952.

\bibitem[Ar2]{Ar2} J. Arthur, {\it A trace formula for reductive groups. II. 
Applications of a truncation operator}. Compositio Math. {\bf 40} (1980), 
no. 1, 87 --121. 

\bibitem[Ar3]{Ar3} J.\ Arthur, {\it The trace formula in invariant form}, 
Ann. of Math. (2) {\bf 114} (1981), no.\ 1, 1--74.

\bibitem[Ar4]{Ar4} J. Arthur, {\it Intertwining operators and residues I.
Weighted characters}, Journal of Funct. Analysis {\bf 84} (1989), 19 -- 84.

\bibitem[Ar5]{Ar5} J. Arthur, {\it An Introduction to the trace
eformula}, Clay Math. Proceedings, Vol. {\bf 4}, 2005.

\bibitem[Ar6]{Ar6} J. Arthur, {\it The local behavior of weighted
orbital integrals}, Duke Math. J. \textbf{56} (1988), no. 2, 223--293.

\bibitem[Ar7]{Ar7} J.\ Arthur, {\it A measure on the unipotent variety}, Can.
J. Math., Vol. XXXVII,  (1985), no.\ 6, 1237--1274.


\bibitem[Ar8]{Ar8} J. Arthur.
\newblock {\it On a family of distributions obtained from {E}isenstein series. 
{I}. {A}pplication of the {P}aley-{W}iener theorem}. \newblock 
{\em Amer. J. Math.},  (1982) 104(6):1243--1288.

 \bibitem[Ar9]{Ar9} J. Arthur.
\newblock {\it On a family of distributions obtained from {E}isenstein series. 
{II}.
  {E}xplicit formulas}. \newblock {\em Amer. J. Math.},  (1982) 
104(6):1289--1336.

\bibitem[Av]{Av} V. G. Avakumovi\'c, {\it \"Uber die Eigenfunktionen auf 
geschlossenen Riemannschen Mannigfaltig\-keiten.}  
Math. Z. {\bf 65} (1956), 327--344.

\bibitem[BM]{BM} D. Barbasch, H. Moscovici, {\it $L^2$-index and
the trace formula}, J. Funct. Analysis {\bf 53} (1983), 151--201.

\bibitem[Bo1]{Bo1} A. Borel, {\it Some finiteness properties of
adele groups
over number fields}, Inst. Hautes \'Etudes Sci. Publ. Math. {\bf 16} (1963),
5 -- 30.

\bibitem[Bo2]{Bo2} A. Borel, {\it Automorphic $L$-function}, In: Automorphic forms, representations and L-functions, Part 2, pp. 27–61, Proc. Sympos. Pure Math., XXXIII, Amer. Math. Soc., Providence, \\R.I., 1979. 

\bibitem[Bo3]{Bo3} A. Borel, {\it Stable real cohomology of
arithmetic groups}, Annales scient. de l'ÉNS {\bf 7}, no 2, (1974), 235--272.

\bibitem[BJ]{BJ} A. Borel and H. Jacquet, {\it Automorphic forms
and automorphic representations},\\ Proc. Symp. Pure Math. {\bf 33},
Part I, Amer. Math. Soc. (1979), 184--202.

\bibitem[BE]{BE} P. Borwein and T. Erdélyi, {\it Sharp extensions
of Bernstein's inequality to rational spaces}, Mathematika {\bf 43}(2) (1997),
413--423.

\bibitem[BG]{BG} A.\ Borel and H.\ Garland, {\it Laplacian and the discrete
  spectrum of an arithmetic group}, Amer. J. Math. {\bf 105} (1983),
  no.\ 2, 309--335.

\bibitem[CD]{CD} L.\ Clozel, P.\ Delorme, {\it Le th\'eor\`eme de {P}aley-{W}iener invariant pour les groupes de
 {L}ie r\'eductifs}, Invent. Math. {\bf 77} (1984), no.\ 3, 427--453.

\bibitem[CLL]{CLL} L. Clozel, J.-P. Labesse, and R. Langlands, {\it
Morning seminar on the trace formula}, mimeographed notes, IAS, Princeton,
1984.

\bibitem[CV]{CV} Y. Colin de Verdière, {\it Pseudo-laplaciens, II}.
Ann. Inst. Fourier {\bf 33} (1983), no. 2, 87–-113.

\bibitem[Do]{Do} H. Donnelly, {\it On the cuspidal spectrum for finite volume
symmetric spaces}, J. Differential Geometry {\bf 17} (1982), 239--253.

\bibitem[GL]{GL} S. Gelbart, E. Lapid, {\it Lower bounds for $L$-functions
at the edge of the critical strip}, Amer. J. Math. {\bf 128} (2006), no. 3,
619 -- 638.

\bibitem[He]{He} D.A. Hejhal, {\it The Selberg trace formula for $\PSL(2,\R)$}
Vol. 2, Lecture Notes Math. {\bf 1001}, Springer-Verlag, Berlin, 1983. 


\bibitem[Fe]{Fe} W. Feller, {\it An Introduction to Probability
Theory and its Applications}, 2nd ed. ,Vol. II, John Wiley \& Sons, New York,
1971.

\bibitem[FLM1]{FLM1} T. Finis, E. Lapid and W. M\"uller, {\it On the spectral side of Arthur's trace formula - absolute convergence}, Ann. of Math.,
vol. {\bf 174} (2011), no. 1, 173--195.

\bibitem[FLM2]{FLM2} T. Finis, E. Lapid, W. M\"uller, {\it Limit multiplicities for
principal congruence subgroups of $\GL(n)$ and $\SL(n)$}, J. Inst. Math.
Jussieu {\bf 14} (2015), no. 3, 589--638.

\bibitem[FL1]{FL1} T. Finis, E. Lapid, {\it On the analytic properties of 
intertwinig operators I: global normalizing factors},  Bull. Iranian Math. 
Soc. {\bf 43} (2017), no. 4, 235--277.

\bibitem[FL2]{FL2} T. Finis, E. Lapid, {\it On the remainder term of the Weyl
 law for congruence subgroups of Chevalley groups}, Duke Math. J. {\bf 170}
(2021), no. 4, 653–695. 

\bibitem[FL3]{FL3} T. Finis, E. Lapid, {\it On the continuity of the geometric
side of the trace formula.} Acta Math. Vietnam. {\bf 41} (2016), no. 3,
425--455.

\bibitem[FM]{FM} T. Finis, J. Matz, {\it On the asymptotics of Hecke
operators for reductive groups}, Math. Annalen {\bf 380} (2021), no. 3--4,
1037--1104.

\bibitem[HC1]{HC1} Harish-Chandra,{\it The Plancherel formula for reductive
p-adic groups}, In: Collected papers, IV, pp. 353-367, Springer-Verlag, 
New York-Berlin-Heidelberg, 1984.

\bibitem[HC2]{HC2} Harish-Chandra,{\it Automorphic Forms on Semisimple Lie
groups}, Lecture Notes in Math. Vol. {\bf 62}, Springer-Verlag,
Berlin-Heidelberg-New York, 1968.
  
\bibitem[Kn]{Kn} A.W. Knapp, {\it Representation theory of semisimple groups},
Princeton University Press, Princeton and Oxford, 2001.

\bibitem[La1]{La1} R.P. Langlands, {\it On the functional equations satisfied
by Eisenstein series}, Lecture Notes Math., Vol. {\bf 544}, 
Springer-Verlag, Berlin-New York, 1976.

\bibitem[La2]{La2} R. P. Langlands, {\it Problems in the theory of automorphic
forms}, Lectures in modern analysis and apllications, III, Springer, Berlin,
1970, pp. 18--61. Lect. Notes in Math., Vol {\bf 170}.

\bibitem[La3]{La3} R.P. Langlands, {\it Eisenstein series.} In: Proc.
Symp. Pure Math. Vol. {\it 9}, pp. 235--252, A.M.S., Providence, R.I., 1966.

\bibitem[LM]{LM} E. Lapid, W. M\"uller, {\it Spectral asymptotics for
arithmetic quotients of $\SL(n,\R)/\SO(n)$}. Duke Math. J. {\bf 149} (2009),
no. 1, 117--155. 

\bibitem[LaM]{LaM} H.B. Lawson and M.-L. Michelsohn, {\it Spin geometry},
Princeton Mathematical Series, {\bf 38}. Princeton University Press,
Princeton, NJ, 1989. 

\bibitem[LV]{LV} E. Lindenstrauss, A. Venkatesh, {\it Existence and Weyl's law
for spherical cusp forms}, GAFA {\bf 17} (2007), 220 -- 251.

\bibitem[Ma]{Ma} A. Maiti, {\it Weyl's law for arbitrary Archimedean type},
arXiv:2210.05986v1.

\bibitem[MM1]{MM1} J. Matz and W. M\"uller, {\it Analytic torsion of arithmetic quotients of the symmetric space {$\SL_n(\R)/\SO(n)$}}, Geom. Funct. Anal. {\bf 27} no.\ 6 (2017), 1378--1449.

\bibitem[MM2]{MM2} J. Matz and W. M\"uller, {\it Analytic torsion
for arithmetic locally symmetric manifolds and approximation of $L^2$-torsion},
 J. Funct. Anal. 284 (2023), no. 1, Paper No. 109727.

\bibitem[Mia]{Mia} R.J. Miatello, {\it The Minakshisundaram-Pleijel 
coefficients for the vector-valued heat kernel on compact locally symmetric 
spaces of negative curvature.}  Trans. Amer. Math. Soc. {\bf 260} (1980), 
1--33. 

\bibitem[Mil]{Mil} S.D. Miller, {\it On the existence and temperedness of cusp
forms for $\SL_3(\Z)$}. J. Reine Angew. Math. {\bf 533} (2001), 127--169. 

\bibitem[MW]{MW} C. M{\oe}glin, J.-L. Waldspurger, {\it Spectral Decomposition
and Eisenstein Series}, Cambridge Tracts in Mathematics, {\bf 113}.
Cambridge University Press, Cambridge, 1995. 

\bibitem[Mu1]{Mu1} W. M\"uller, {\it The trace class conjecture in the theory of
automorphic forms}, Annals of Math. {\bf 130}, no. 3, (1989), 473--529.

\bibitem[Mu2]{Mu2} W. M\"uller, {\it On the spectral side of the Arthur trace 
formula}, GAFA {\bf 12} (2002), 669 -- 722.

\bibitem[Mu3]{Mu3} W. M\"uller, {\it Weyl's law for the cuspidal spectrum of
$\SL_n$}, Annals of Math. {\bf 165} (2007), 273--333.

\bibitem[MS]{MS} W.\ M\"uller, B.\ Speh, {\it Absolute convergence
of the spectral side of the Arthur trace formula for $\GL(n)$
With an appendix by E.M.\ Lapid.}  Geom. Funct. Anal. {\bf 14} (2004), no.\ 1, 
58--93.

\bibitem[Sa1]{Sa1} P. Sarnak, {\it On cusp forms}, In: The Selberg trace
formula and related topics, Contemp. Math. {\bf 53}, Amer. Math. Soc.,
Providence, RI, 1986, pp: 393--407.

\bibitem[Sa2]{Sa2} P. Sarnak, {\it Spectra of hyperbolic surfaces}, Bull. Amer.
Math. Soc. (N.S.) {\bf 40} (2003), no. 4, 441--478.

\bibitem[Se1]{Se1} A. Selberg, {\it Harmonic analysis}, in ''Collected Papers'',
Vol. I, Springer-Verlag, Berlin-Heidelberg-New York (1989), 626--674.

\bibitem[Sh1]{Sh1} F. Shahidi, {\it On certain $L$-functions}, Amer. J. Math.
{\bf 103} (1981), 297--356.

\bibitem[Sh2]{Sh2} F. Shahidi, {\it Eisenstein series and automorphic L-functions}, American Mathematical Society Colloquium Publications, 58.
American Mathematical Society, Providence, RI, 2010. 

\bibitem[Si]{Si} A.J.\ Silberger,
{\it Harish-Chandra's Plancherel Theorem for $p$-adic Groups}, Trans. Amer. Math. Soc. {\bf 348} (1996), no.\ 11, 4673 -- 4686.

\bibitem[We]{We} A. Weil, {\it Adeles and algebraic groups},
Birkh\"auser, Boston, Basel, Stuttgart, 1982.


\end{thebibliography}
\end{document}